\documentclass{article}
\usepackage{arxiv,times} 

\usepackage[T1]{fontenc}    
\usepackage{hyperref}       
\usepackage{url}            
\usepackage{booktabs}       
\usepackage{nicefrac}       
\usepackage{microtype}      
\usepackage{enumitem}
\usepackage{placeins}
\usepackage{multirow}

\usepackage{amsmath, amsthm, amsfonts}
\theoremstyle{plain}
\theoremstyle{definition}
\newtheorem{definition}{Definition}
\newtheorem{theorem}{Theorem}
\newtheorem{remark}{Remark}
\newcommand{\dist}{\mathrm{dist}}

\usepackage{graphicx, calc} 
\usepackage{xcolor}
\usepackage{pgfplots}
\usepgfplotslibrary{fillbetween}
\xdefinecolor{nina_green}{RGB}{23, 197, 30}
\xdefinecolor{nina_blue}{RGB}{71, 201, 227}
\usepackage{pgfplots}
\usepgfplotslibrary{fillbetween}
\usepackage{tikz}
\usetikzlibrary{arrows, positioning, shapes}
\usetikzlibrary{calc, patterns, angles, quotes}
\tikzset{
    right angle quadrant/.code={
        \pgfmathsetmacro\quadranta{{1.5,1.5,-1.5,-1.5}[#1-1]}   
        \pgfmathsetmacro\quadrantb{{1.5,-1.5,-1.5,1.5}[#1-1]}},
    right angle quadrant=1, 
    right angle length/.code={\def\rightanglelength{#1}},   
    right angle length=2ex, 
    right angle symbol/.style n args={3}{
        insert path={
            let \p0 = ($(#1)!(#3)!(#2)$) in     
                let \p1 = ($(\p0)!\quadranta*\rightanglelength!(#3)$), 
                \p2 = ($(\p0)!\quadrantb*\rightanglelength!(#2)$) in 
                let \p3 = ($(\p1)+(\p2)-(\p0)$) in  
            (\p1) -- (\p3) -- (\p2)
        }
    }
}

\title{On the Effectiveness of Persistent Homology} 

\author{%
  Renata Turke\v{s} \\
  University of Antwerp \\
  \texttt{renata.turkes@uantwerpen.be} \\ 
  \And  
  Guido Mont\'{u}far \\
  University of California, Los Angeles\\
  \texttt{montufar@math.ucla.edu} \\  
  \And  
  Nina Otter \\
  Queen Mary University of London \\
  \texttt{n.otter@qmul.ac.uk} \\
}

\begin{document}
\maketitle

\begin{abstract}
\noindent 
Persistent homology (PH) is one of the most popular methods in Topological Data Analysis. Even though PH has been used in many different types of applications, the reasons behind its success remain elusive; in particular, it is not known for which classes of problems it is most effective, or to what extent it can detect geometric or topological features. The goal of this work is to identify some types of problems where PH performs well or even better than other methods in data analysis. We consider three fundamental shape analysis tasks: the detection of the number of holes, curvature and convexity from 2D and 3D point clouds sampled from shapes. Experiments demonstrate that PH is successful in these tasks, outperforming several baselines, including PointNet, an architecture inspired precisely by the properties of point clouds. In addition, we observe that PH remains effective for limited computational resources and limited training data, as well as out-of-distribution test data, including various data transformations and noise. For convexity detection, we provide a theoretical guarantee that PH is effective for this task in $\mathbb{R}^d$, and demonstrate the detection of a convexity measure on the FLAVIA data set of plant leaf images. Due to the crucial role of shape classification in understanding mathematical and physical structures and objects, and in many applications, the findings of this work will provide some knowledge about the types of problems that are appropriate for PH, so that it can --- to borrow the words from Wigner 1960 --- ``remain valid in future research, and extend, to our pleasure", but to our lesser bafflement, to a variety of applications. 
\end{abstract}

\section{Introduction} 
\label{section_introduction}

Persistent homology (PH) is an extension of homology, which gives a way to capture topological information about connectivity and holes in a geometric object. PH can be regarded as a framework to compute representations of raw data that can be used for further processing or as inputs to learning algorithms. There have been numerous successful applications of PH in the last decade, from prediction of biomolecular properties \cite{cang2017topologynet, cang2018representability, wang2020topology}, face, gait and activity recognition \cite{zhou2017exploring, lamar2012human, leon2014topological, jimenez2016designing} or digital forensics \cite{asaad2017topological}, to discriminating breast-cancer subtypes \cite{singh2014topological}, quantifying the porosity of nanoporous materials \cite{lee2017quantifying}, classifying fingerprints \cite{giansiracusa2017persistent}, or studying the morphology of leaves \cite{li2018topological}.\footnote{A database of applications of persistent homology is being maintained at \cite{giunti2021applications}.} At the same time, the reasons behind these successes are not yet well understood. Indeed, the data used in real-world applications is complex, so that there are numerous effects at play and one is often left unsure why PH worked, i.e., what type of topological or geometric information it captured that facilitated the good performance.

The title of our manuscript is inspired by a famous paper from 1960, ``The unreasonable effectiveness of mathematics in the natural sciences'' \cite{wigner1960unreasonable}, in which Wigner discusses, with wonder, how mathematical concepts have applicability far beyond the context in which they were originally developed. The same, we believe, is true for persistent homology. While this method has been applied successfully to a wide range of application problems, we believe that for PH to remain relevant, there is a need to better understand why it is so successful. Thus, we distinguish between the {\em usefulness} of PH for applications, which has been attested in hundreds of applications and publications, and its {\em effectiveness}, namely that PH is capable of producing an intended or desired result. Thus, here we initiate an investigation into the effectiveness of PH, or in other words, we investigate \emph{what} is seen by persistent homology: Given a data set, i.e., a point cloud, which underlying topological and geometric features can we detect with PH? This question is related to manifold learning and, specifically, topological and geometric inference: Given a finite point cloud $X$ of (noisy) samples from an unknown manifold $M,$ how can one infer properties of $M$ \cite{chazal2013geometric, chazal2017robust, boissonnat2018geometric, bobrowski2015topology}? Obtaining a representation of a shape that can be used in statistical models is an important task in data analysis and numerous approaches to modeling surfaces and shapes \cite{turner2014persistent}.

To pursue our investigation, we set out to identify some fundamental data-analysis tasks that can be solved with PH. Since PH is inspired by homology, which provides a measure for the number of components, holes, voids, and higher-dimensional cycles of a space --- to which we collectively refer as ``topological features'' ---, we start with the obvious question of whether PH applied to a point cloud sampled from a geometric object can detect the \textbf{\textit{number of ($1$-dimensional) holes}} of the underlying object. Unlike homology, however, PH registers also the persistence of topological features across scales, and can thereby capture geometric information, such as size or position of holes. We therefore also investigate how well PH can detect fundamental geometric notions of \textbf{\textit{curvature}} and \textbf{\textit{convexity}}. For each of the three problems, we first discuss theoretical results that provide a guarantee that PH can solve these tasks. Detection of convexity with PH has not been investigated in the literature to date, and we prove a new result.

To investigate how well the PH pipeline works in practice, we compare its performance against several baselines on synthetic point-cloud data sets. As a first machine learning (ML) baseline we take an SVM trained on the distance matrices of point clouds. We further consider fully-connected neural networks (NN) with a single or multiple hidden layers, also trained on distance matrices. As a stronger baseline we consider a PointNet trained on the point clouds directly. PointNet \cite{qi2017pointnet, griffiths2020point} is designed specifically for point cloud data. Similar architectures with convolutional (and fully-connected and pooling) layers have been applied for Betti-number and curvature estimation \cite{paul2019estimating, guerrero2018pcpnet}. For convexity detection, we also evaluate the performance of PH on real-world data. The theoretical guarantees above imply that the results for PH would generalize to new data.

Finally, we note that our goal is not to claim the superiority of PH compared to other approaches in the literature, in particular, with the state-of-the-art methods for each of the problems. We do not necessarily expect that on well-specified mathematical problems PH will beat state-of-the-art algorithms that have been specifically designed for those tasks. 
Instead, what we think is interesting and remarkable is that PH can in fact solve  tasks it is not specifically or uniquely designed for. Moreover, an advantage of PH is that it can reveal, e.g., both topology and curvature at the same time, avoiding the need to employ and combine state-of-the-art models for each of the tasks.

\paragraph{Related work} In spite of the growing interest in PH, so far there is only limited work in the direction that we pursue here. There is indeed theoretical evidence that the number of holes of the underlying space can be detected from PH (under some conditions about the target space, the sample density and closeness to the space) \cite{chazal2008towards, niyogi2008finding, kim2019homotopy}, and there is significant interest in investigating how well this works in practice \cite{chazal2013bootstrap}. However, so far there are only few available results. Some works demonstrate that PH can be used to detect the number of holes, but only on individual toy examples (e.g. \cite{robins2002computational}, \cite[Figure 19]{chazal2013geometric}, \cite[Figures 9-20]{kurlin2014fast}, \cite[Figures 2, 3, 6, 11, 12]{chazal2017robust}) without looking into the statistical significance between different classes of data, or the accuracy of some classification algorithms on a comprehensive data set. There are also some works where PH is used to estimate the Betti numbers on a possibly larger data set, but only with the goal of using this information to e.g., study the behavior of deep neural networks \cite{naitzat2020topology} or ensure topologically correct dimensionality reduction \cite{paul2017study} or image segmentation \cite{hu2019topology}, so that the soundness of this estimation is not investigated, which is the focus of our work.

Some insights about PH and curvature have been obtained in the literature, starting with an illustrative example in \cite[Figure 12]{collins2004barcode} which shows that PH on the filtered tangent complex can distinguish between letters (C and I) that have the same topology, since their curvature is different. Recently, \cite{bubenik2020persistent} show both theoretically and experimentally that PH can predict curvature (with computational experiments replicated in \cite{vipond2020multiparameter}), which inspired us to investigate this problem in more detail.

Regarding the important geometric problem of classification between convex and concave shapes, we were not able to identify any previous works investigating the applicability of PH to this task.

Some further recent work investigating the topological and geometric features seen by PH are the following. Bubenik and D{\l}otko \cite{bubenik2017persistence} show that using PH of points sampled from spheres one can determine the dimension of the underlying spheres. A connection has also been established between PH and  the magnitude of a metric space (an isometric invariant) \cite{otter2018magnitude}.  There have been several efforts in using PH to estimate fractal dimensions, such as  \cite{schweinhart2020fractal} in which Schweinhart proves that the fractal dimension of some metric spaces can be recovered from the PH of random samples.

\paragraph{Main contributions} Our contributions can be summarized as follows.
\begin{itemize}[leftmargin=*]

\item We prove that PH can detect convexity in $\mathbb{R}^d$ (Theorem~\ref{thm_convexity}).

\item We define a new tubular filtration function (medium through which PH is extracted from data), that is crucial for the detection of convexity (Definition~\ref{def_tubular}).

\item We demonstrate experimentally that PH can detect the number of holes (Section~\ref{section_holes}), curvature (Section~\ref{section_curvature}), and convexity (Section~\ref{section_convexity}) from synthetic point clouds in $\mathbb{R}^2$ or $\mathbb{R}^3$, outperforming SVMs and fully-connected networks trained on distance matrices, and PointNet trained on point clouds. For convexity detection, we also show that PH obtains a good performance on a real-world data set of plant leaf images.

\item We demonstrate experimentally that PH features allow to solve the above tasks even in the case of limited training data (Section~\ref{section_holes}), noisy (Section~\ref{section_holes}) 
and out-of-distribution (Section~\ref{section_convexity}) test data, and limited computational resources (Section~\ref{section_holes}, Section~\ref{section_curvature}, Section~\ref{section_convexity}). 

\item We provide insights about the topological and geometric features that are captured with long and short persistence intervals (Section~\ref{section_discussion}), and formulate guidelines for applications that are suitable for PH (Section~\ref{section_conclusions}). 

\item We provide data sets that can be directly used as a benchmark for our tasks or other related point-cloud-analysis or classification problems. We provide computer code to construct more data and replicate our experiments. 

\end{itemize}

\section{Background on persistent homology}
\label{section_ph}

Homology is a topological concept that attempts to distinguish between topological spaces by constructing algebraic invariants that reflect their connectivity properties \cite{robins2002computational}, i.e., $k$-dimensional cycles (components, holes, voids, \dots). The number of independent $k$-dimensional cycles is called $k$-th Betti number and denoted by $\beta_k.$ For example, the circle has Betti numbers $\beta_0=1,$ $\beta_1=1$, $\beta_2=0$, and for a torus, we have $\beta_0=1$, $\beta_1=2,$ $\beta_2=1$.

Persistent homology is an extension of this idea \cite{zomorodian2005computing} that has found success in applications to data. To calculate PH from some data $X,$ we must first build a filtration, i.e., a family of nested topological spaces $\{ K_r \}_{r \in \mathbb{R}}$ which, in a suitable sense, approximate $X$ at different scales $r \in \mathbb{R}.$ Typically, $X = \{x_1, x_2, \dots, x_n\}$ is a point cloud in $\mathbb{R}^d,$ and $K_r$ is a simplicial complex, a set of simplices $\sigma$ (which we can think of as vertices, edges, triangles, \dots) such that if $\sigma \in K_r$ and $\tau \subseteq \sigma,$ then $\tau \in K_r$ \cite{collins2004barcode}. A common choice is $K_r = VR(X, r),$ where $VR(X, r)$ is the Vietoris-Rips simplicial complex, in which $\sigma=\{x_1,\dots , x_m\}\in VR(X, r)$ when $\dist(x_i, x_j)\leq r$ for all $1\leq i,j\leq m$. PH can then be summarized with a persistence diagram (PD), a scatter plot with the $x$ and $y$ axes respectively depicting the scale $r \in \mathbb{R}$ at which each cycle is born and dies or is identified (i.e., merges) with another cycle within a filtration. The length $l = d - b$ of a persistence interval $(b, d)$ measures the lifespan --- the so-called ``persistence'' --- of the corresponding cycle in the filtration.

Instead of working directly with persistence diagrams, these are often represented by other signatures that are better suited for  machine learning frameworks. A common choice is a persistence image (PI) \cite{adams2017persistence} (discretized sum of Gaussian kernels centered at the PD points), or a persistence landscape (PL) (functions obtained by ``stacking isosceles triangles" above persistence intervals, with height reflecting their lifespan) \cite{bubenik2015statistical}. The steps for extracting PH features are visualized in Appendix~\ref{app_exp_pipelines}. For good choices of filtration and signature \cite{turkevs2021noise}, there are theoretical results that guarantee that PH is stable under small perturbations \cite{chazal2009stability, skraba2020wasserstein}. After the PH signature is calculated, statistical hypothesis testing \cite{bubenik2015statistical, berry2018functional}, or machine learning techniques such as SVM or k-NN \cite{adams2017persistence, garin2019topological, obayashi2018persistence} can be used on these features to study the differences within the data set of interest. 
It is important to note that PH is very flexible, as different choices can be made in every step of the pipeline, regarding the input and output of PH, as detailed in the remainder of this section.

\subsection{Approximation of a space at scale $r \in \mathbb{R}$}
\label{section_ph_subsection_approx}

Instead of the Vietoris-Rips complex, other types of complexes can be used to approximate the data $X$ at the given scale $r \in \mathbb{R}.$ For instance, since Vietoris-Rips simplicial complex is large \cite{otter2017roadmap}, one might rather choose the alpha complex \cite{edelsbrunner2000topological} (for a visualization, see Appendix~\ref{app_holes_pipeline}), which is closely related to the Vietoris-Rips complex \cite{kim2019homotopy}, but consists of significantly less simplices and is faster to construct when the dimension of the ambient space is $2$ or $3$ (for details, see Appendix~\ref{app_thm_computational}). If data $X$ is an image rather than a point cloud, or if the point cloud can be seen as an image without losing important information, cubical complexes \cite{kaczynski2004computational} might be a more suitable choice, where vertices, edges and triangles are replaced by vertices, edges and squares (for a visualization, see Appendix~\ref{app_convexity_pipeline}).

\subsection{Filtration}
\label{section_ph_subsection_filtration}

A filtration $\{ K_r \}_{r \in \mathbb{R}}$ of a point cloud in $\mathbb{R}^d$ can be constructed from any function $f: \mathbb{R}^d \rightarrow \mathbb{R}$ by considering $K_r$ to be the sublevel set of $f$ thresholded by $r \in \mathbb{R}$:  $\{ y \in \mathbb{R}^d \mid f(y) \leq r \}$.  
The underlying filtration function in the common PH pipeline introduced above is the distance function $\delta_X: \mathbb{R}^d \rightarrow \mathbb{R},$ where $\delta_X(y) = \min \{ \dist(y, x) \mid x \in X \}$ is the distance to point cloud $X.$ Indeed, the Vietoris-Rips simplicial complex $VR(X, r)$ approximates the sublevel set 
$$K_r = \delta_X^{-1}((-\infty, r]) = \{ y \in \mathbb{R}^d \mid \delta_X(y) \leq r \} = \cup_{x \in X} B(x, r),$$
where $B(x, r)$ is a ball with radius $r$ centered around $x \in X$ \cite{chazal2011geometric}.

However, PH on such a filtration is very sensitive to outliers, since even a single outlier changes $\delta_X$ significantly. In the presence of outliers, it is better to replace the distance function with the Distance-to-Measure (DTM) $\delta_{X, m}: \mathbb{R}^d \rightarrow \mathbb{R},$ where $\delta_{X, m}(x)$ is the average distance from a number of neighbors on the point cloud \cite{chazal2011geometric, anai2020dtm} (for a visualization, see Appendix~\ref{app_holes_pipeline}). However, depending on the task, there are many other filtration functions one could choose, such as rank \cite{petri2013topological}, height, radial, erosion, dilation \cite{garin2019topological}, and the resulting PH captures completely different information about the cycles \cite{turkevs2021noise}. For example, whereas PH with respect to the Vietoris-Rips filtration encodes the size of the hole, PH on the height filtration informs about the position of the hole. For a more detailed discussion about the influence of filtration, see Appendix~\ref{app_guidelines}.

\subsection{Persistence signature} 
\label{section_ph_subsection_signature}

Next to PIs and PLs, a plethora of persistence signatures has been introduced in the literature, e.g., Betti numbers \cite{islambekov2019harnessing, umeda2017time} or Euler characteristic \cite{li2018persistent} (across scales), or even scalar summaries such as amplitude \cite{garin2019topological}, entropy \cite{chintakunta2015entropy, rucco2016characterisation}, or algebraic functions of the birth and death values \cite{adcock2013ring,kalivsnik2019tropical}. Some of these signatures summarize the same information, but lie in different metric spaces \cite{turner2020same}. Others, however, such as the scalar summaries listed above, discard information compared to PDs, as it might sometimes be useful to, e.g., only capture the (total or maximum) persistence of intervals, but not all the detailed information about all birth and death values (see Appendix~\ref{app_guidelines}).

\section{Number of holes}
\label{section_holes}

In this section, we focus on the task of (ordinal) classification of point clouds by the number of $1$-dimensional holes. Research in psychology shows that global properties often dominate perception, and, in particular, that topological invariants such as number of holes, inside versus outside, and connectivity can be effective primitives for recognizing shapes \cite{pomerantz2003wholes}.  Extracting such topological information can therefore prove useful for many computer vision tasks. There are theoretical results in the literature that ensure that PH with respect to the alpha simplicial complex can be successful for this problem (Appendix~\ref{app_thm_holes}), and the computational experiments that follow demonstrate this success in practice.

\paragraph{Data} We consider $20$ different shapes in $\mathbb{R}^2$ and $\mathbb{R}^3$, with four different shapes having the same number of holes ($0, 1, 2, 4$ or $9$). For each shape, we construct $50$ point clouds each consisting of $1\,000$ points sampled from a uniform distribution over the shape, resulting in a balanced data set of $1\,000 = 20 \times 50$ point clouds. A few examples of these point clouds are shown in Figure~\ref{fig_holes_data}. The label of a point cloud is the number of holes in the underlying shape.

\begin{figure}[h]
\centering
\begin{tabular}{lccccc}
\toprule

shapes & 
\includegraphics[width = 0.06\linewidth]{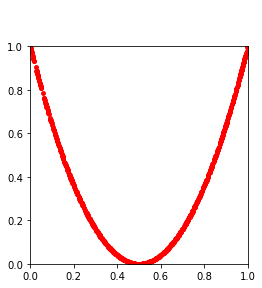} 
\includegraphics[width = 0.06\linewidth]{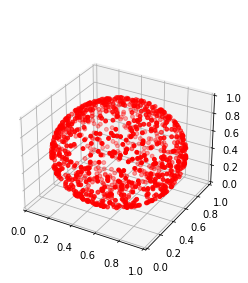} &
\includegraphics[width = 0.06\linewidth]{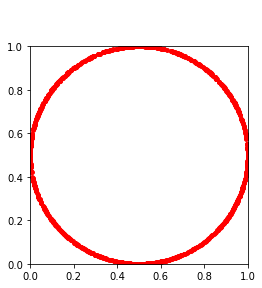} 
\includegraphics[width = 0.06\linewidth]{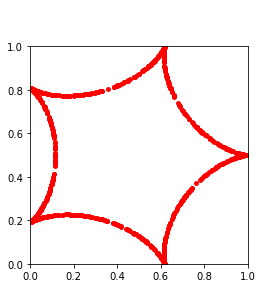} &
\includegraphics[width = 0.06\linewidth]{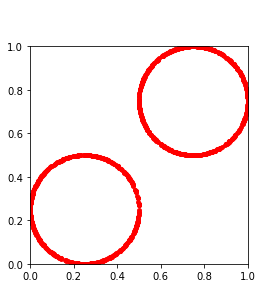} 
\includegraphics[width = 0.06\linewidth]{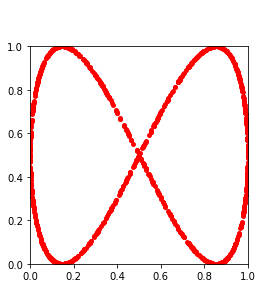} &
\includegraphics[width = 0.06\linewidth]{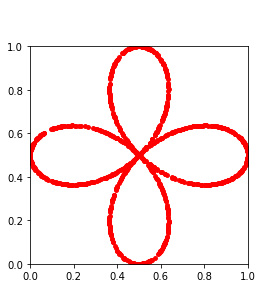} 
\includegraphics[width = 0.06\linewidth]{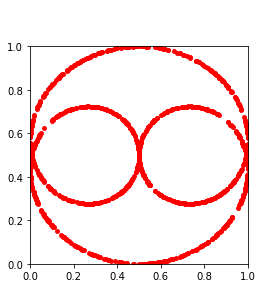} &
\includegraphics[width = 0.06\linewidth]{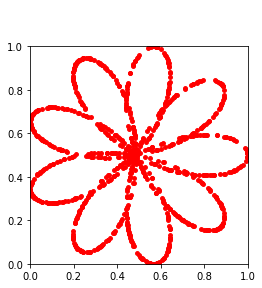} 
\includegraphics[width = 0.06\linewidth]{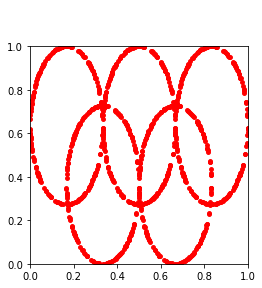} \\

&
\includegraphics[width = 0.06\linewidth]{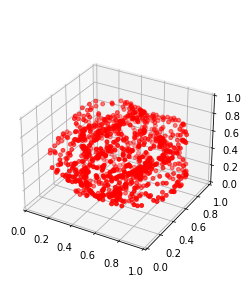} 
\includegraphics[width = 0.06\linewidth]{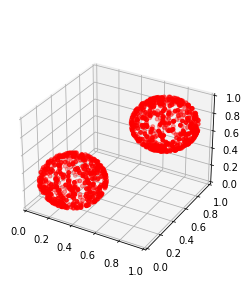} &
\includegraphics[width = 0.06\linewidth]{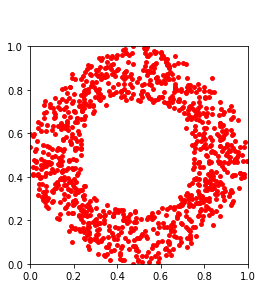} 
\includegraphics[width = 0.06\linewidth]{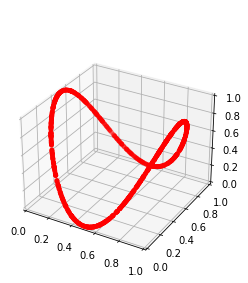} & 
\includegraphics[width = 0.06\linewidth]{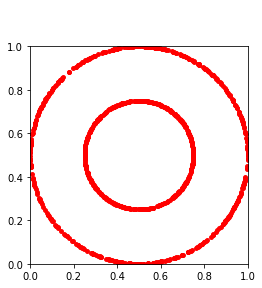} 
\includegraphics[width = 0.06\linewidth]{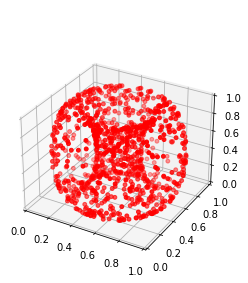} &
\includegraphics[width = 0.06\linewidth]{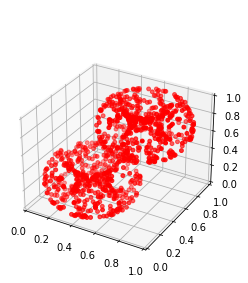}
\includegraphics[width = 0.06\linewidth]{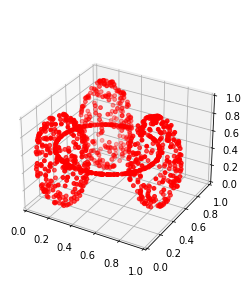} &
\includegraphics[width = 0.06\linewidth]{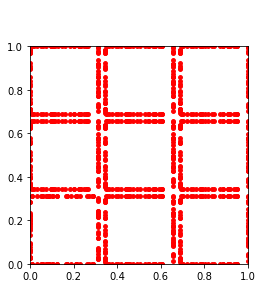} 
\includegraphics[width = 0.06\linewidth]{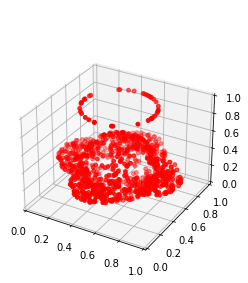} \\

\midrule
number of holes & 0 & 1 & 2 & 4 & 9 \\
\bottomrule
\end{tabular}

\caption{Number of holes data set.}
\label{fig_holes_data}
\end{figure}

\paragraph{PH pipeline} For each point cloud $X$ and scale $r \in \mathbb{R},$ we consider its alpha complex. We will look into scenarios in which data contains noise, and therefore, instead of the standard distance function, we consider Distance-to-Measure (DTM) as the filtration function. We extract $1$-dimensional PDs, that are then transformed to PIs, PLs, or a simple signature consisting  only of lifespans $l = d - b$ of the 10 most persisting cycles (as there are at maximum 9 holes of interest in the given data set)\footnote{Although a sphere has no $1$-dimensional holes, its PD might consist of many short intervals which correspond to the small holes on the surface. In addition, in the presence of noise, additional small holes might appear for any point cloud. Hence, it is not a good idea to consider the cardinality $|\text{PD}|$ of the PD as the signature.}, and classified with an SVM. We consider a ``PH simple" pipeline, which relies on the 10 lifespans, and a ``PH" pipeline wherein grid search is employed to choose the best out of the three aforementioned persistent signatures and the values of their parameters. For more details on the pipeline, see Appendix~\ref{app_exp_pipelines} and Appendix~\ref{app_holes_pipeline}.

\paragraph{Results} We investigate the clean and robust test accuracy under four types of transformations (translation, rotation, stretch, shear) and two types of noise (Gaussian noise, outliers). For more details on these transformations, see Appendix~\ref{app_holes_transformations}. We train the classifier on $80\%$ of the original point clouds, and test on the remaining $20\%$ of the data either in its original form or subject to transformations and noise. The results reported in Figure~\ref{fig_holes_results} (with detailed results across multiple runs in Appendix~\ref{app_holes_performance}) show that PH obtains very good test accuracy on this classification task, even in the presence of affine transformations or noise, outperforming baseline machine- and deep-learning techniques.\footnote{Interestingly, although PointNet was designed with the idea to be invariant to affine transformations, it performs poorly when the test data is translated or rotated (and this is consistent with some previous results \cite{zhang2019rotation, zhang2020learning, le2021geometric, xiao2021triangle, zhang2021revisiting, wang2022rotation, zhao2022rotation}), or when it contains outliers. Traditional neural networks perform very poorly, which might not come as a big surprise, since it was recently demonstrated that they transform topologically complicated data into topologically simple one as it passes through the layers, vastly reducing the Betti numbers (nearly always even reducing to their lowest possible values: $\beta_k = 0$ for $k>0$, and $\beta_0=1)$ \cite{naitzat2020topology}. Of course, the choice of activation function and hyperparameters might have an important influence on performance \cite{naitzat2020topology}.} We reach a similar conclusion in case of limited training data and computational resources (Appendix~\ref{app_holes_training_curves}, Appendix~\ref{app_holes_comp}). Firstly, the evolution of the test accuracy across different amounts of training data demonstrates that PH achieves good performance for a small number of training point clouds, which is not the case for other pipelines. Secondly, although the hyperparameter tuning of the PH pipeline does take time (as we consider a wide range of parameters for the different persistence signatures), it is still less than for PointNet. Moreover, Figure~\ref{fig_holes_results} shows that even the simple PH pipeline, where the SVM is used directly on the lifespans of the 10 most persisting cycles (without any tuning of PH-related parameters) performs well. 

\begin{figure}[h]
\centering
\includegraphics[width = \linewidth]{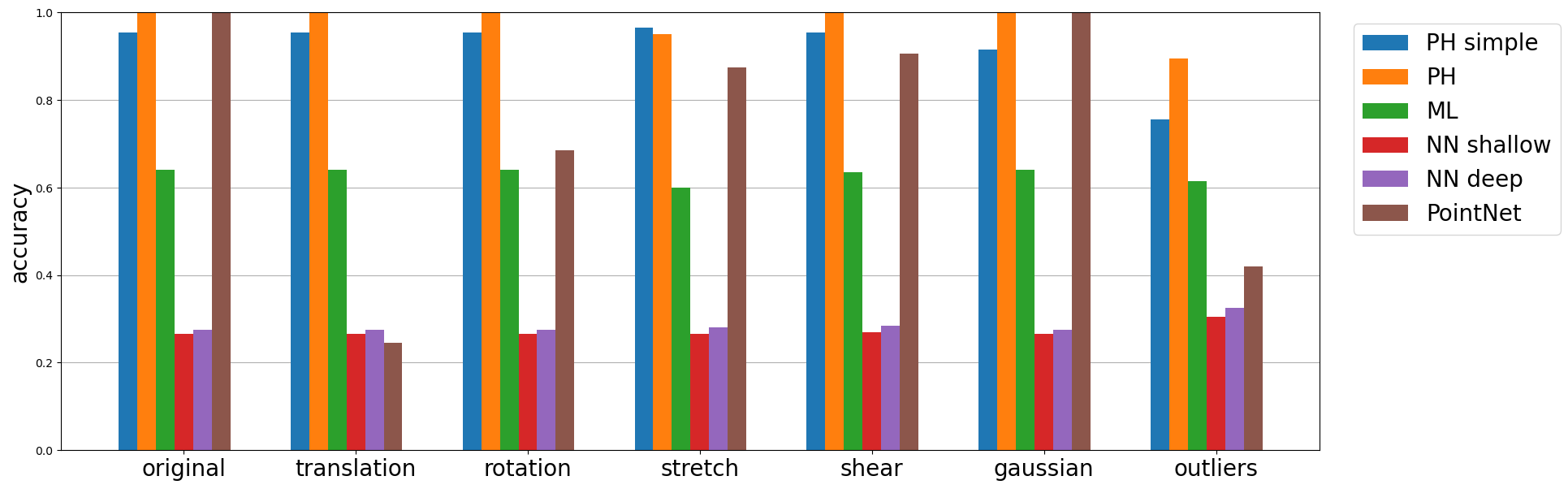}
\caption{Persistent homology can detect the number of holes.}
\label{fig_holes_results}
\end{figure}

\section{Curvature}
\label{section_curvature}

This section considers a regression task to predict the curvature of an underlying shape based on a point cloud sample. Estimating curvature-related quantities is of prime importance in computer vision, computer graphics, computer-aided design or computational geometry, e.g., for surface segmentation, surface smoothing or denoising, surface reconstruction, and shape design \cite{cazals2005estimating}. For continuous surfaces, normals and curvature are fundamental geometric notions which uniquely characterize local geometry up to rigid transformations \cite{guerrero2018pcpnet}. Recently, it has been shown that, using PH, curvature can be both recovered in theory (Appendix~\ref{app_thm_curvature}), and effectively estimated in practice \cite{bubenik2020persistent}. We run a similar experiment, evaluating the PH pipeline against our baselines, and also taking a closer look into the importance of short intervals.

\paragraph{Data} A balanced data set is generated in the same way as in \cite{bubenik2020persistent}: We consider unit disks $D_\kappa$ on surfaces of constant curvature $\kappa$: (i) $\kappa=0,$ Euclidean plane, (ii) $\kappa>0,$ sphere with radius $1/\sqrt{\kappa},$ and (iii) $\kappa<0,$ Poincaré disk model of the hyperbolic plane. Curvature $\kappa$ lies in the interval $[-2, 2]$ so that a disk with radius one can be embedded on the upper hemisphere of a sphere with constant curvature $\kappa$ (as it spherical cap). For each $\kappa \in \{-2, -1.96, \dots, -0.04, 0, 0.04, \dots, 1.96\},$ we construct $10$ point clouds by sampling $500$ points from the unit disk $D_\kappa$ with the probability measure proportional to the surface area measure \cite[Section 2.7, Section 4.1]{bubenik2015statistical}. A few examples with $\kappa \in \{-2, -1, -0.1, 0, 0.1, 1, 2\}$ are illustrated in Figure~\ref{fig_curvature_data}\footnote{The unit disks with negative curvature are here visualized on hyperbolic paraboloids. These saddle surfaces have non-constant curvature, but they locally resemble the hyperbolic plane.}.These $101 \times 10 = 1\,010$ point clouds are considered as the training data, whereas the test data set is built in a similar way for $100$ values of $\kappa$ chosen uniformly at random from $[-2, 2]$. The label of a point cloud is the curvature $\kappa$ of the underlying disk $D_\kappa$. Note that all these disks are homoemorphic: they are contractible, so that their homology is trivial and homology is thus unable to distinguish between them \cite{bubenik2020persistent}.

\begin{figure}[h]
\centering

\begin{tabular}{lccccccc}
\toprule

shapes &
\includegraphics[width = 0.095\linewidth]{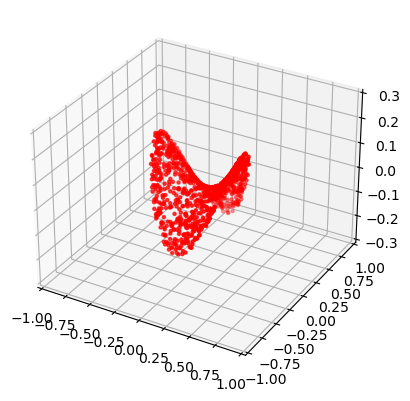} &
\includegraphics[width = 0.095\linewidth]{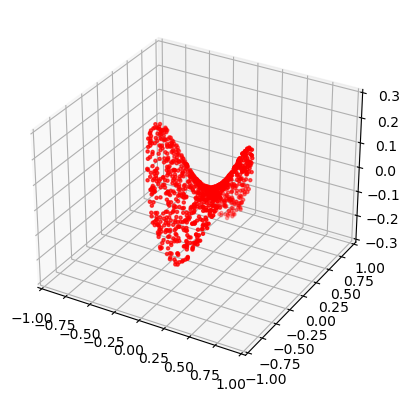} &
\includegraphics[width = 0.095\linewidth]{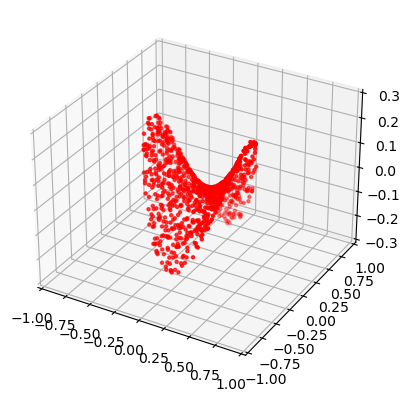} &
\includegraphics[width = 0.095\linewidth]{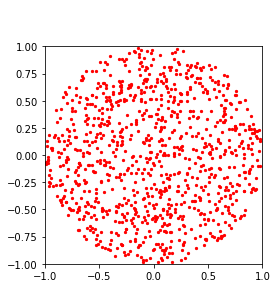} &
\includegraphics[width = 0.095\linewidth]{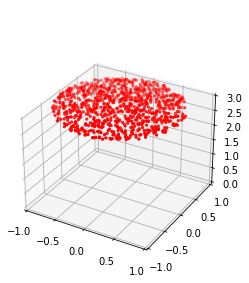} &
\includegraphics[width = 0.095\linewidth]{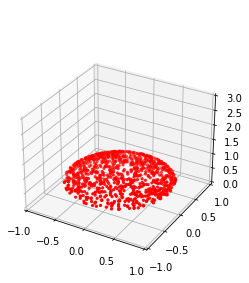} &
\includegraphics[width = 0.095\linewidth]{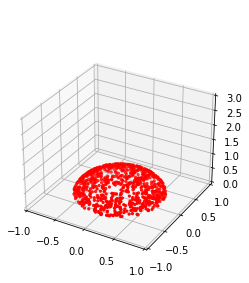} \\

\midrule
curvature & -2.00 & -1.00 & -0.10 & 0.00 & 0.10 & 1.00 &2.00 \\

\bottomrule
\end{tabular}

\caption{Curvature data set.}
\label{fig_curvature_data}
\end{figure}

\paragraph{PH pipeline} For each point cloud $X,$ we first calculate the suitable matrix of pairwise distances between the point-cloud points: hyperbolic, Euclidean or spherical, respectively for negative, zero and positive curvature \cite[Section 2.7]{bubenik2015statistical}. The input for PH is the filtered Vietoris-Rips simplicial complex.\footnote{Alpha complex is faster to compute,
but involves Delaunay triangulation, whose unique existence is guaranteed only in Euclidean spaces. To calculate PH, we rely on the Ripser software \cite{tralie2018ripser}, which is at the time the most efficient library to compute PH with Vietoris-Rips complex  \cite{otter2017roadmap}.} We extract $0$- and $1$-dimensional PDs, which are then transformed into PIs, PLs or lifespans, to be fed to SVM. More details on the pipeline are provided in Appendix~\ref{app_exp_pipelines} and Appendix~\ref{app_curvature_pipeline}.

\paragraph{Results} Figure~\ref{fig_curvature_results} shows the mean squared errors for the PH and other pipelines, together with their regression lines, with detailed results across multiple runs listed in Appendix~\ref{app_curvature_performance}. The results show that PH indeed detects curvature, outperforming other methods.\footnote{Simple machine and deep learning techniques are able to differentiate between positive and negative curvature, but perform poorly in predicting the actual value of the curvature of the underlying surface.} Next to the PH pipeline discussed above, wherein a grid search is used to tune the parameters (Appendix~\ref{app_exp_pipelines}), we also consider SVM on the lists of lifespans of all persistence intervals (PH simple), and SVM only on the $10$ longest lifespans (PH simple $10$), in order to investigate if all persistence intervals contribute to prediction. We see that the performance drops if we only focus on the longest $10$ intervals, so that the many short intervals together capture the geometry of interest for this problem. Similarly as in Section~\ref{section_holes}, the grid search across the different parameters for persistence signatures does take time (Appendix~\ref{app_curvature_comp}), but Figure~\ref{fig_curvature_results} shows that SVM on a simple signature of all (0-dimensional) lifespans performs well. We highlight that the data used here, as was the data in Bubenik's work \cite{bubenik2020persistent}, is sampled from surfaces with \emph{constant} curvature. In future work it would be interesting to conduct similar experiments on shapes with non-constant curvature.

\begin{figure}[h]
\centering
\begin{tabular}{ccccc}
0-dim PH simple & 0-dim PH simple 10 & 0-dim PH & ML & PointNet \\
\includegraphics[width = 0.165\linewidth]{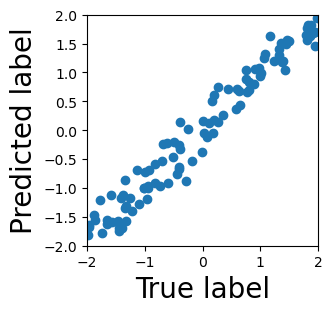} & 
\includegraphics[width = 0.165\linewidth]{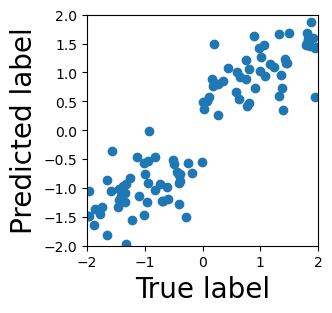} & 
\includegraphics[width = 0.165\linewidth]{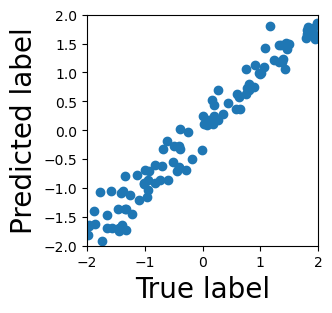} & 
\includegraphics[width = 0.165\linewidth]{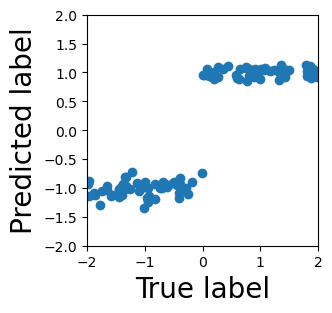} & 
\includegraphics[width = 0.165\linewidth]{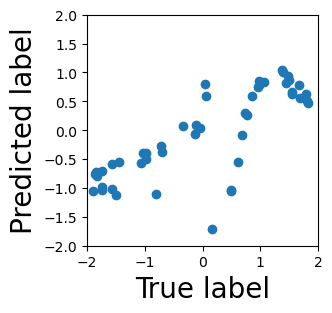} \\
MSE = 0.06 & MSE = 0.21 & MSE = 0.08 & MSE = 0.34 & MSE = 578.28 \\
\\
\\
1-dim PH simple & 1-dim PH simple 10 & 1-dim PH & NN shallow & NN deep \\
\includegraphics[width = 0.165\linewidth]{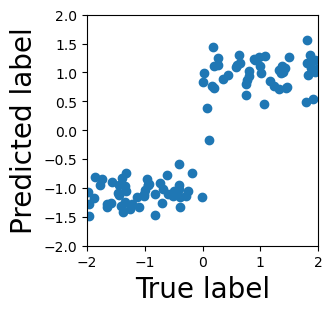} & 
\includegraphics[width = 0.165\linewidth]{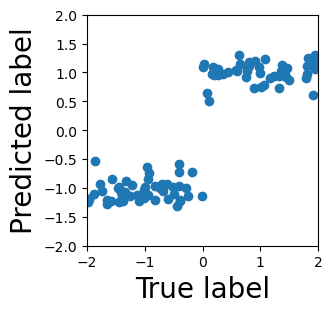} & 
\includegraphics[width = 0.165\linewidth]{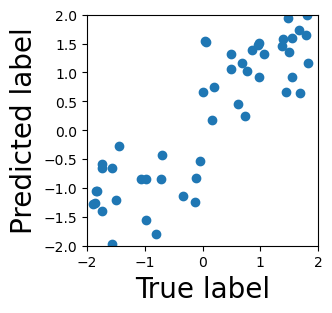} & 
\includegraphics[width = 0.165\linewidth]{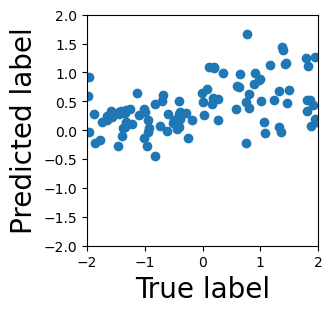} & 
\includegraphics[width = 0.165\linewidth]{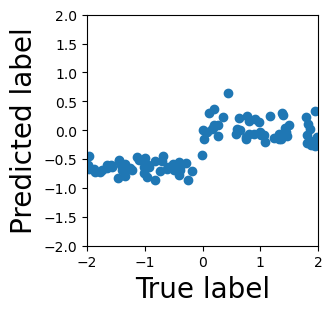} \\
MSE = 0.34 & MSE = 0.29 & MSE = 0.18 & MSE = 0.66 & MSE = 0.43 \\
\end{tabular}
\caption{Persistent homology can detect curvature.}
\label{fig_curvature_results}
\end{figure}

\section{Convexity}
\label{section_convexity}

In this section, we consider the binary classification task that consists of detecting whether a point cloud is sampled from a convex set. Convexity is a fundamental concept in geometry \cite{crombez2019efficient}, which plays an important role in learning, optimization \cite{berman2019testing}, numerical analysis, statistics, information theory, and economics \cite{ronse1989bibliography}. Furthermore, points of convexities and concavities have been demonstrated as crucial for human perception of shapes across many experiments \cite{schmidtmann2015shape}.

To the best of our knowledge, prior to our work PH has not been employed to analyze convexity, and it is a task for which PH's effectiveness might seem surprising. In the first decade after the introduction of PH, it was seen primarily as the descriptor of global topology. Recently, there have been many discussions and greater understanding that PH also captures local geometry  \cite{adams2021topology}. However, it is still suggested that the long persistence intervals capture topology (as was the case with the detection of holes in Section~\ref{section_holes}), and many---even too many for the human eye to count---short persistence intervals capture geometrical properties (as was the case with curvature prediction in Section~\ref{section_curvature}). However, as we show in Theorem \ref{thm_convexity} (proof in Appendix~\ref{app_thm_convexity}) and as our experiments suggest, it is a single, and the second-longest persistence interval that enables us to detect concavity. A crucial ingredient in our result is the introduction of tubular filtrations (Definition~\ref{def_tubular}), which, to the best of our knowledge, are a novel contribution to the TDA literature (details in Appendix~\ref{app_thm_convexity}).

\begin{definition}
\label{def_tubular}
Given a line $\alpha\subset \mathbb{R}^d$, we define the {\bf tubular function with respect to $\alpha$} as follows:
\begin{alignat}{3}
\tau_\alpha \colon & \notag \mathbb{R}^d && \notag \to \mathbb{R} \\ 
& x && \notag \mapsto \dist(x, \alpha) \, ,
\end{alignat}
where $\dist(x, \alpha)$ is the distance of the point $x$ from the line $\alpha$. Given $X \subset \mathbb{R}^d$ and a line $\alpha,$ we are interested in studying the sublevel sets of $\tau_\alpha,$ i.e., the subsets of $X$ consisting of points within a specific distance from the line. We define
\[
X_{\tau_\alpha, r} =\{x\in X  \mid \tau_\alpha(x) \leq r\}\ =\{x\in X  \mid \dist(x,\alpha)\leq r\}\, .
\]
We call $\{X_{\tau_\alpha,r}\}_{r\in \mathbb{R}_{\geq 0}}$ the {\bf tubular filtration with respect to $\alpha$}.
\end{definition}

\begin{theorem} 
\label{thm_convexity} 
Let $X\subset \mathbb{R}^d$ be triangulizable. We have that $X$ is convex if and only if for every line $\alpha$ in $\mathbb{R}^d$ the persistence diagram in degree $0$ with respect to the tubular filtration $\{X_{\tau_\alpha,r}\}_{r\in \mathbb{R}_{\geq 0}}$ contains exactly one interval. 
\end{theorem}

\paragraph{Data} We construct a balanced data set by sampling $5\,000$ points from convex and concave (nonconvex) shapes in $\mathbb{R}^2$. First, we consider the ``regular" convex shapes of triangle, square, pentagon and circle, and their concave variants, sampling $60$ point clouds of each of the eight shapes, $480$ point clouds in total. Next, we build $480$ ``random" convex and concave shapes, in order to be able to investigate if an algorithm is actually detecting convexity, or only the different basic shapes. A few examples are shown in Figure~\ref{fig_convexity_data}. To construct a random convex shape, we generate $10$ points at random, and then build their convex hull using the quickhull algorithm \cite{virtanen2020scipy}. We construct random concave shapes in a similar way, but instead of the convex hull, we build the alpha shape \cite{edelsbrunner1983shape, edelsbrunner1994three} with the optimized alpha parameter, which gives a finer approximation of a shape from a given set of points. If the alpha shape is convex (i.e., if the alpha shape and its convex hull are the same), we reconstruct the concave shape from scratch. A point cloud has label $1$ if it is sampled from a convex shape, and $0$ otherwise.

\begin{figure}[h]
\centering

\begin{tabular}{lcc}
\toprule

shapes &
\includegraphics[width = 0.09\linewidth]{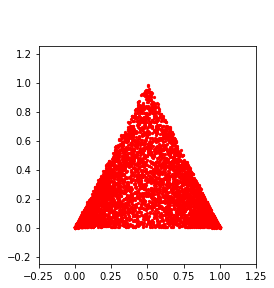} 
\includegraphics[width = 0.09\linewidth]{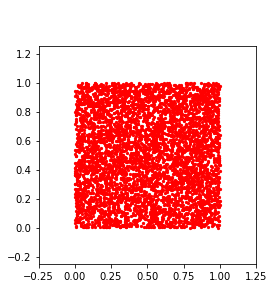} 
\includegraphics[width = 0.09\linewidth]{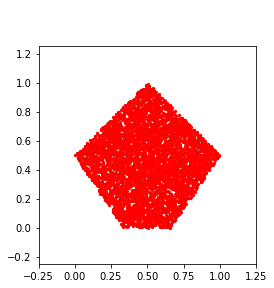}   
\includegraphics[width = 0.09\linewidth]{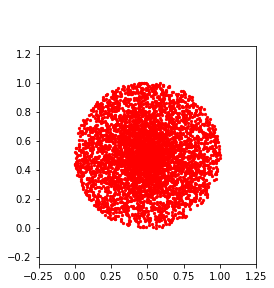} & 
\includegraphics[width = 0.09\linewidth]{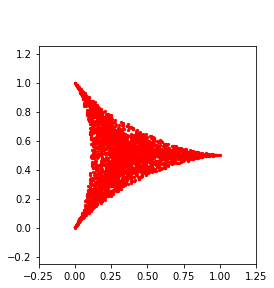}  
\includegraphics[width = 0.09\linewidth]{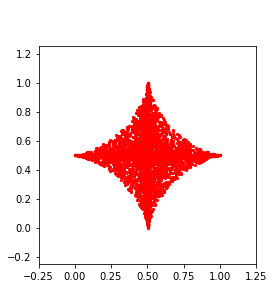}   
\includegraphics[width = 0.09\linewidth]{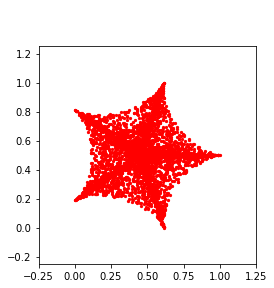}   
\includegraphics[width = 0.09\linewidth]{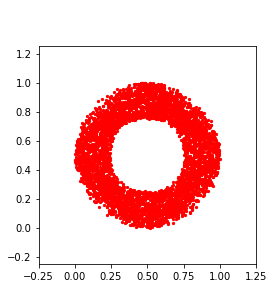} \\

& 
\includegraphics[width = 0.09\linewidth]{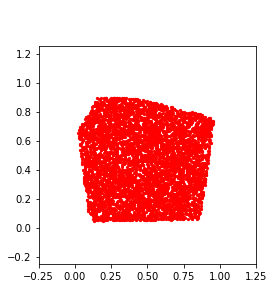} 
\includegraphics[width = 0.09\linewidth]{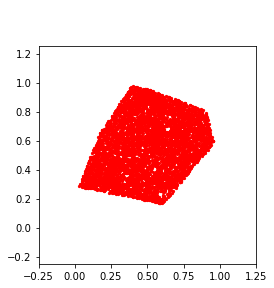}  
\includegraphics[width = 0.09\linewidth]{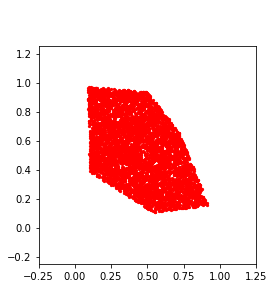}  
\includegraphics[width = 0.09\linewidth]{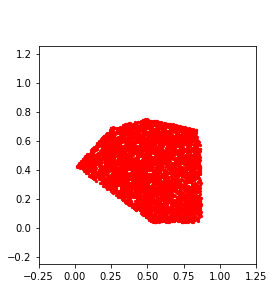} & 
\includegraphics[width = 0.09\linewidth]{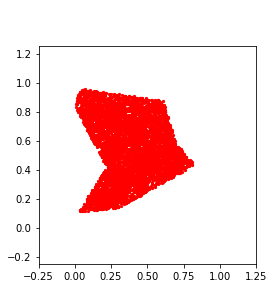}   
\includegraphics[width = 0.09\linewidth]{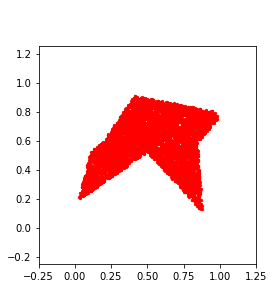}   
\includegraphics[width = 0.09\linewidth]{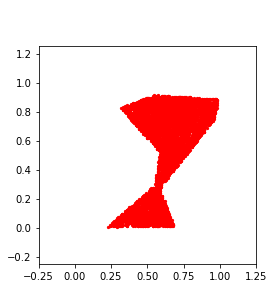} 
\includegraphics[width = 0.09\linewidth]{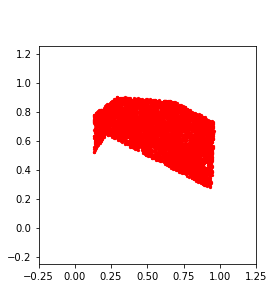} \\

\midrule
convexity & 1 & 0 \\

\bottomrule
\end{tabular}

\caption{Convexity data set.}
\label{fig_convexity_data}
\end{figure}

\paragraph{PH pipeline} To build a filtration, we consider cubical complexes filtered by tubular functions that measure the distance of points from a certain line, see Definition \ref{def_tubular} for a precise definition. For a good choice of line, multiple  components would be seen in the filtration of a point cloud sampled from a concave shape, at least for some values $r \in \mathbb{R}$ (see also the illustrations in Appendix~\ref{app_thm_convexity}). For this reason we consider the cubical complex, rather than the standard Vietoris-Rips simplicial complex wherein these separate components could be connected with an edge (for details, see Appendix~\ref{app_convexity_pipeline}). To build an image from the point cloud, we construct a $20 \times 20$ grid and define a pixel as black if it contains any point-cloud points, and white otherwise.

Since sources of concavity can lie anywhere on the point cloud, we consider nine different lines for the tubular filtration function (for a visualization of the pipeline, see Appendix~\ref{app_convexity_pipeline}). For each of the nine lines, we extract 0-dimensional PD, as it captures information about the components. If the point cloud is thus sampled from a convex shape, its PD will only see a single component for any line, whereas there will be multiple components at least for some lines for points clouds sampled from concave shapes. For this reason, for each of the nine lines, we focus our attention only on the lifespan of the second most persisting cycle. We can consider this 9-dimensional vector as our PH signature, but in our experiment choose an even simpler summary: the maximum of these lifespans, since we only care if there are multiple  components \emph{for at least one line}. This scalar could even be used as some measure of a level of concavity of a shape.

\paragraph{Results} As already indicated, to gain some insights into how well the different approaches discriminate convexity from concavity rather than differentiating between the different basic shapes, we look at the classification accuracies under different conditions (Figure~\ref{fig_convexity_results}, with detailed results across multiple runs in Appendix~\ref{app_convexity_performance}). We start with the easiest case, where both the train and test data consist of the simple regular convex and concave shapes (Figure~\ref{fig_convexity_data}, first row), and then proceed to the scenario where both train and test data are random shapes (Figure~\ref{fig_convexity_data}, second row). Next we proceed to out-of-distribution test data, where we train on the regular and test on random shapes, or vice versa. In every case, we train on $400$ and test on $80$ point clouds. The results show that PH is able to detect convexity, surpassing other methods significantly in all scenarios except for PointNet in the scenario on the data set of regular shapes which performs on par. Results reported in Appendix~\ref{app_convexity_comp} show PH is also computationally efficient.

\begin{figure}[h]
\centering
\includegraphics[width = 0.8 \linewidth]{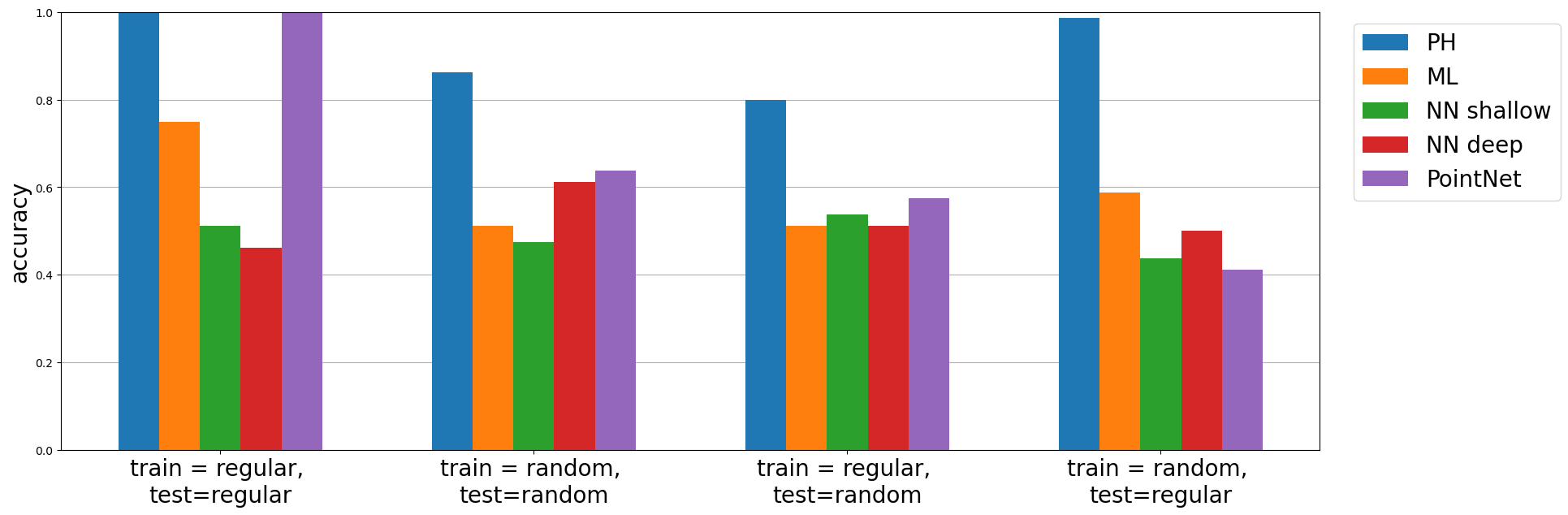}
\caption{Persistent homology can detect convexity.}
\label{fig_convexity_results}
\end{figure}

The PH pipeline above makes a wrong prediction when concavity is barely pronounced, or if it is missed by the selected tubular filtration lines (for details, see Appendix~\ref{app_convexity_wrong}). However, the accuracy of PH can easily be improved simply by considering a finer resolution for the cubical complexes and/or additional tubular filtration lines. The particular PH pipeline summarized in this section would also make a wrong prediction if the data set would include shapes that have small or non-central holes, e.g., a square with a hole in the top left corner. In this case, the accuracy could also be improved by considering a finer cubical complex resolution and by considering additional non-central tubular filtration lines within shapes, or by adding (the maximum lifespan of the) 1-dimensional PH which captures holes. The pipeline is not limited to polygons, or connected shapes, and it can be generalized to surfaces in higher dimensions (Theorem~\ref{thm_convexity}). In Appendix~\ref{app_real_data}, we also consider the real-world data set FLAVIA for which we demonstrate that a PH pipeline is effective in detecting a continuous measure of convexity.

\section{Implications for PH interpretation: topology vs.\ geometry}
\label{section_discussion}

Here we discuss how our results contribute to the important and ongoing discussion about the interpretation of long versus short persistence intervals. When PH was first introduced in the literature, the long intervals were commonly considered as important or ``signal'', and short intervals as irrelevant or ``noise'' \cite{cohen2007stability}. Subsequently the discussion has refined when it was shown that short and medium-length persistence intervals have the most distinguishing power for specific types of applications \cite{bendich2016persistent, stolz2017persistent}. The current understanding is roughly that long intervals reflect the topological signal, and (many) short intervals can help in detecting geometric features \cite{adams2021topology}. We believe that our work brings new insight into this discussion. We give a summary of the implications of our work in this section, and we provide a more detailed discussion in Appendix~\ref{app_guidelines}.

\paragraph{Topology and long persistence} Stability results guarantee that a number of longest persistence intervals reflect the topological signal, i.e., the number of cycles \cite{chazal2008towards}. These theorems give information about the threshold that differentiates between long and short persistence intervals. In Section~\ref{section_holes}, where we focus on the topology of underlying shapes, the experiments demonstrate that this threshold can be learned with simple machine learning techniques. However, it is important to highlight that the distinction between long and short persistence is vague in practice. Indeed, seemingly short persistence intervals capture the topology in Section~\ref{section_holes}, but the second-longest interval is topological noise in Section~\ref{section_convexity}, since every shape in the data set has only a single component (although this second-longest interval captures important \emph{geometric} information, what enabled us to discriminate between convex and concave shapes). These two problems also clearly indicate how the long intervals that encode topology might or might not be irrelevant, depending on the signal of the particular application domain.

\paragraph{Geometry and short persistence} The current understanding is that (many) short persistence intervals detect geometry. Section~\ref{section_curvature} confirms that this indeed can be the case. However, we highlight that \emph{all} cycles can encode geometric information, such as the information about their size (with respect to the Vietoris-Rips and related filtrations, as in Sections~\ref{section_holes} and \ref{section_curvature}) or their position (with respect to the height or tubular filtration, as in Section~\ref{section_convexity}). This further implies that, depending on the application, \emph{any number} of short or intervals of \emph{any persistence} can be important, which was clearly demonstrated in Section~\ref{section_convexity}, where we show that a single interval detects convexity.

\section{Conclusions}
\label{section_conclusions}

\paragraph{Main contribution} The goal of this work is to gain a better understanding of the topological and geometric features that can be captured with persistent homology. We focus on the detection of number of holes (Section \ref{section_holes}), curvature (Section \ref{section_curvature}), and convexity (Section \ref{section_convexity}). Theoretical evidence for the first two classes of problems has been established in the literature, and we prove a new result that guarantees that PH can detect convexity (Theorem~\ref{thm_convexity}). We also experimentally demonstrate that PH can solve all three problems for synthetic point clouds in $\mathbb{R}^2$ and $\mathbb{R}^3,$ outperforming a few baselines. This is true even when there is limited training data and computational resources, and for noisy or out of distribution test data. For convexity detection, we also show the effectiveness of PH in a real world plant morphology application.

\paragraph{Relevance} Firstly, the findings point the way to further advances in utilizing the potential of PH in applications: we can expect PH to be successful for classification or regression problems where the data classes differ with respect to number of holes, curvature and/or convexity. Detailed guidelines are discussed in Appendix~\ref{app_guidelines}. Due to the crucial role of shape classification in understanding and recognizing physical structures and objects, image processing and computer vision \cite{lim2010shape}, our results demonstrate that PH can---to borrow the words from Wigner \cite{wigner1960unreasonable}---``remain valid in future research, and extend, to our pleasure", and lesser bafflement, to a variety of applications. Secondly, the results advance the discussion about the importance of long and short persistence intervals, and their relationship to topology and geometry (Section~\ref{section_discussion}). Topology is captured by the long intervals, geometry is encoded in all persistence intervals, and any interval can encode the signal in the particular application domain.

\paragraph{Limitations} The results focus on three selected problems and data sets, and it would therefore be interesting to consider other tasks. In addition, we do not have an extensive comparison of the state-of-the-art for the given problems. Our work seeks to understand if PH is successful for a selected set of tasks by benchmarking it against some well-performing methods.

\paragraph{Future research} An in-depth analysis of the hypothetical applications discussed in this paper (Appendix~\ref{app_guidelines}) and selected success stories of PH from the literature could further improve our understanding of the topological and geometric information encoded in PH, and the interpretation of persistence intervals of different lengths. Alternative approaches for detection of convexity with PH (relying on higher homological dimensions, or multiparameter persistence) are particularly interesting avenues for further work. Furthermore, even though our results imply that PH features are recommended over baseline models for the three selected classes of problems, they also provide inspiration on how to improve existing learning architectures. Further work could investigate deep learning  models on PH (and standard) features or kernels \cite{hofer2017deep, bruel2019topology, rathore2019autism, yan2021link}, an additional network layer for topological signatures, or PH-based priors, regularization or loss functions \cite{bruel2019topology, chen2019topological, clough2019explicit, hu2019topology, clough2020topological, zhao2020persistence, wong2021persistent}.

\paragraph{Potential negative societal impact} While we recognize that the applications of shape analysis can take many different directions, we do not foresee a direct path of this research to negative societal impacts.

\paragraph{Funding transparency statement} This research was conducted while RT was a Fulbright Visiting Researcher at UCLA. GM acknowledges support from ERC grant 757983, DFG grant 464109215, and NSF-CAREER grant DMS-2145630. NO acknowledges support from the Royal Society, under grant RGS\textbackslash R2\textbackslash 212169.

\bibliographystyle{plain}
\bibliography{tda}

\newpage 
\appendix
\setcounter{theorem}{0}
\setcounter{definition}{0}

\section*{Appendix}

The appendix is organized into the following sections. 
\begin{itemize}[leftmargin=*]
\item Appendix~\ref{app_thms} Theoretical results
\item Appendix~\ref{app_exp}: Experimental details; pipelines, training and hyperparameter tuning 
\item Appendix~\ref{app_holes}: Additional experimental details for number of holes 
\item Appendix~\ref{app_curvature}: Additional experimental details for curvature 
\item Appendix~\ref{app_convexity}: Additional experimental details for convexity 
\item Appendix~\ref{app_guidelines}: Guidelines for persistent homology in applications; long and short persistence intervals \item Appendix~\ref{app_real_data}: Persistent homology detects convexity in FLAVIA data set
\end{itemize}

\section{Theoretical results}
\label{app_thms}

In this section, we provide the proof of our Theorem~\ref{thm_convexity} that guarantees that PH can be used to detect convexity. We then formulate Theorem~\ref{thm_holes} and Theorem~\ref{thm_curvature} from the literature, that summarize known theoretical guarantees that PH can detect the number of holes and curvature. At the end of the section, we also discuss some known results about the theoretical computational complexity of PH.

\subsection{Convexity}
\label{app_thm_convexity}

In our pipeline for the detection of convexity we consider tubular filtrations, akin to  the concept of tubular neighborhoods in differential topology. Given a subspace $X$ of $\mathbb{R}^d$, and a line $\alpha\subset \mathbb{R}^d,$ we consider all the points in $X$ that are within a specific distance $r$ from the line. By varying $r$, we then obtain a filtration of $X$. We note that while a  tubular neighborhood is defined with respect to any curve, here instead we focus on the special case of (closed) neighborhoods with respect to a line.

\begin{definition}[Tubular filtration]
Given a line $\alpha \subset \mathbb{R}^d$, we define the {\bf tubular function with respect to $\alpha$} as follows:
\begin{alignat}{3}
\tau_\alpha \colon & \notag \mathbb{R}^d && \notag\to \mathbb{R} \\ 
& x && \notag \mapsto \dist(x,\alpha) \, ,
\end{alignat}
where $\dist(x,\alpha)$ is the distance of the point $x$ from the line $\alpha$. Given $X \subset \mathbb{R}^d$ and a line $\alpha \subset \mathbb{R}^d$, we are interested in studying the sublevel sets of $\tau_\alpha,$ i.e., the subsets of $X$ consisting of points within a specific distance from the line. We define
\[
X_{\tau_\alpha, r} =\{x\in X  \mid \tau_\alpha(x) \leq r\}\ =\{x\in X  \mid \dist(x,\alpha)\leq r\}\, .
\]
We call $\{X_{\tau_\alpha,r}\}_{r\in \mathbb{R}_{\geq 0}}$ the {\bf tubular filtration with respect to $\alpha$}.
\end{definition}

We first formulate and prove our main theorem below, and then discuss the need for tubular filtration in Remark~\ref{remark_tubular}.

\begin{remark}[Different notions of components]
In the proof of Theorem \ref{thm_convexity}, we need to work with path-connected components. We note that while in the main part of this manuscript we always only use the term ``components'', more precisely one would need to distinguish between ``connected components'' and ``path-connected components''. For the purposes of the majority of the spaces that we consider in this work, the two notions are equivalent. Thus, we often simply refer to these as ``components''. 
\end{remark}

\begin{theorem}[Convexity] Let $X\subset \mathbb{R}^d$ be triangulizable\footnote{Namely, there exists a simplicial complex $K$ and a homeomorphism $h\colon |K|\to X$ from the geometric realization of $K$ to X.}. We have that $X$ is convex if and only if for every line $\alpha$ in $\mathbb{R}^d$ the persistence diagram in degree $0$ with respect to the tubular filtration $\{X_{\tau_\alpha,r}\}_{r\in \mathbb{R}_{\geq 0}}$ contains exactly one interval. \end{theorem}

\begin{proof}

We first note that the persistence diagram in degree $0$ of $\{X_{\tau_\alpha,r}\}_{r\in \mathbb{R}_{\geq 0}}$ exists, since the  singular homology in degree $0$ (and with coefficients in a field) of $X_{\tau_\alpha,r}$ is finite-dimensional for any $r\in \mathbb{R}_{\geq 0}$; the existence of the persistence diagram then follows from \cite[Theorem 2.8]{chazal2016structure}\footnote{We note that while in this proof we need to consider singular homology, when computing PH in practice one works with either simplicial or cubical homology (Remark~\ref{remark_convexity_practice}). For the types of spaces that we consider in our work, all homology theories are equivalent. See \cite{kaczynski2004computational} for a discussion of the equivalence between simplicial and cubical homology, and \cite[Chapter 2]{hatcher2005algebraic} for the equivalence of singular and simplicial homology.}.

Assume that $X$ is convex. By definition, we have that for all $p_1, p_2\in X$ the straight-line segment between $p_1$ and $p_2$ is contained in $X$. Let $\alpha$ be any line in $\mathbb{R}^d$ and $r \in \mathbb{R}$. By elementary properties of Euclidean spaces (similarity of triangles, see Figure~\ref{fig_convexity_thm_dist_to_line}), we have that if $\dist(p_1,\alpha)\leq r$ and $\dist(p_2,\alpha)\leq r$, then also $\dist(q,\alpha)\leq r$ for any point $q$ on the line segment between $p_1$ and $p_2$. By the definition of tubular function, this means that  $p_1, p_2\in X_{\tau_\alpha,r}$ implies that $q\in X_{\tau_\alpha,r}$. Therefore the straight-line segment between $p_1$ and $p_2$ is contained in $X_{\tau_\alpha,r}$, which means that $X_{\tau_\alpha,r}$ is convex, and thus path-connected. We therefore have that for any line $\alpha$, the persistence diagram in degree $0$ of $\{X_{\tau_\alpha,r}\}_{r\in \mathbb{R}}$ contains a single interval.

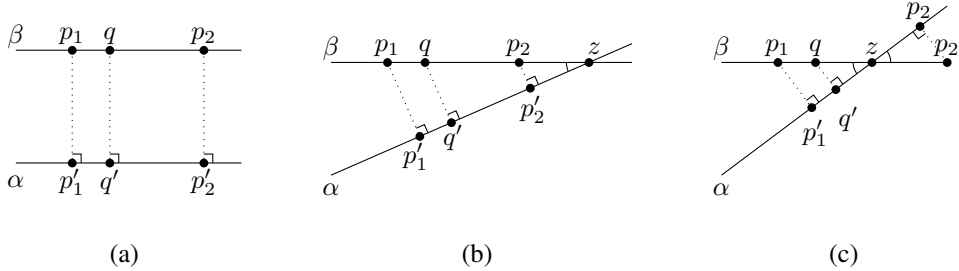
\begin{figure}[h]
\centering
\begin{tabular}{ccccc}

\begin{tikzpicture}[scale=0.25]
\node [scale = 3] (p1) at (3, 8) {.};
\node [scale = 3] (p2) at (10, 8) {.};
\node [scale = 3] (q) at (5, 8) {.};
\draw (0, 8) -- (12, 8);
\node at (3, 8.75) {$p_1$};
\node at (10, 8.75) {$p_2$};
\node at (5, 8.75) {$q$};
\node at (0, 8.75) {$\beta$};

\node (l1) at (0, 2) {};
\node (l2) at (12, 2) {};
\draw (l1.center) -- (l2.center);
\node at (0, 1) {$\alpha$};

\node [scale = 3] (p1') at ($(l1)!(p1)!(l2)$) {.};
\node [scale = 3] (p2') at ($(l1)!(p2)!(l2)$) {.};
\node [scale = 3] (q') at ($(l1)!(q)!(l2)$) {.};
\node at (3, 1) {$p_1'$};
\node at (10, 1) {$p_2'$};
\node at (5, 1) {$q'$};

\draw [dotted] (p1.center) -- (p1'.center);
\draw [dotted] (p2.center) -- (p2'.center);
\draw [dotted] (q.center) -- (q'.center);

\draw [right angle symbol = {l1}{l2}{p1}];
\draw [right angle symbol = {l1}{l2}{p2}];
\draw [right angle symbol = {l1}{l2}{q}];
\end{tikzpicture} &&

\begin{tikzpicture}[scale=0.25]
\node [scale = 3] (p1) at (3, 8) {.};
\node [scale = 3] (p2) at (10, 8) {.};
\node [scale = 3] (q) at (5, 8) {.};
\draw (0, 8) -- (16, 8);
\node at (3, 8.75) {$p_1$};
\node at (10, 8.75) {$p_2$};
\node at (5, 8.75) {$q$};
\node at (0, 8.75) {$\beta$};

\node (l1) at (0, 2) {};
\node (l2) at (16, 9) {};
\draw (l1.center) -- (l2.center);
\node at (0, 1.25) {$\alpha$};

\node [scale = 3] (z) at (intersection of p1--p2 and l1--l2){.};
\node at (14, 8.75) {$z$};

\node [scale = 3] (p1') at ($(l1)!(p1)!(l2)$) {.};
\node [scale = 3] (p2') at ($(l1)!(p2)!(l2)$) {.};
\node [scale = 3] (q') at ($(l1)!(q)!(l2)$) {.};
\node at (4.5, 3) {$p_1'$};
\node at (10.75, 5.5) {$p_2'$};
\node at (6.5, 4) {$q'$};

\draw [dotted] (p1.center) -- (p1'.center);
\draw [dotted] (p2.center) -- (p2'.center);
\draw [dotted] (q.center) -- (q'.center);

\pic [draw, angle radius = 0.3cm] {angle = p1--z--p1'};
\draw [right angle symbol = {l1}{l2}{p1}];
\draw [right angle symbol = {l1}{l2}{p2}];
\draw [right angle symbol = {l1}{l2}{q}];
\end{tikzpicture} &&

\begin{tikzpicture}[scale=0.25]
\node [scale = 3] (p1) at (3, 8) {.};
\node [scale = 3] (p2) at (12, 8) {.};
\node [scale = 3] (q) at (5, 8) {.};
\draw (0, 8) -- (12, 8);
\node at (3, 8.75) {$p_1$};
\node at (12, 8.75) {$p_2$};
\node at (5, 8.75) {$q$};
\node at (0, 8.75) {$\beta$};

\node (l1) at (0, 2) {};
\node (l2) at (12, 11) {};
\draw (l1.center) -- (l2.center);
\node at (0, 1.25) {$\alpha$};

\node [scale = 3] (z) at (intersection of p1--p2 and l1--l2){.};
\node at (8, 8.75) {$z$};

\node [scale = 3] (p1') at ($(l1)!(p1)!(l2)$) {.};
\node [scale = 3] (p2') at ($(l1)!(p2)!(l2)$) {.};
\node [scale = 3] (q') at ($(l1)!(q)!(l2)$) {.};
\node at (5, 4.25) {$p_1'$};
\node at (10.5, 10.75) {$p_2$};
\node at (6.75, 5) {$q'$};

\draw [dotted] (p1.center) -- (p1'.center);
\draw [dotted] (p2.center) -- (p2'.center);
\draw [dotted] (q.center) -- (q'.center);

\pic [draw, angle radius = 0.25cm] {angle = p1--z--p1'};
\pic [draw, angle radius = 0.25cm] {angle = p2--z--p2'};
\draw [right angle symbol = {l1}{l2}{p1}];
\draw [right angle symbol = {l2}{l1}{p2}];
\draw [right angle symbol = {l1}{l2}{q}];
\end{tikzpicture} \\
\\
(a) && (b) && (c) \\

\end{tabular}
\caption{The distance $\dist(p, \alpha)$ from a point $p$ to a line $\alpha$ is defined as the distance between $p$ and the projection $p'$ of $p$ to line $\alpha$ (denoted with dotted lines in the figure). For any two points $p_1, p_2 \in \mathbb{R}^d$, $q \in \mathbb{R}^d$ on a line segment between $p_1$ and $p_2,$ and a line $\alpha$ in $\mathbb{R}^d$, we have that the line $\beta$ passing through $p_1$ and  $p_2$ is either (a) parallel, or (b)-(c) intersects the line $\alpha$ in a point $z = \beta \cap \alpha$. In the former case, per definition of parallel lines, we have that $\dist(p_1, \alpha) = \dist(q, \alpha) = \dist(p_2, \alpha) = \dist(\beta, \alpha).$ In the latter cases, we have a similarity of triangles $\triangle p_1 p_1' z$, $\triangle q q' z$ and $\triangle p_2 p_2' z$, since they all have a right angle, and share the angle $\angle (\beta, \alpha).$ Since $\dist(q, z)$ lies between $\dist(p_1, z)$ and $\dist(p_2, z)$, the triangle similarity implies that $\dist(q, \alpha)$ lies between $\dist(p_1, \alpha)$ and $\dist(p_2, \alpha).$}
\label{fig_convexity_thm_dist_to_line}
\end{figure}

Assume now that $X$ is concave. Then by definition there exist $p_1, p_2\in X$ and a point $q$ on the straight-line segment between $p_1$ and $p_2$ such that $q\notin X$. Since $X$ is closed, we have that there exists $\epsilon>0$ such that $B(q,\epsilon)\subset \mathbb{R}^d\setminus X$ (Figure~\ref{fig_convexity_thm_concave}). Let $\alpha$ be the line passing through $p_1$ and $p_2,$ and let $0 \leq r\leq \epsilon$. We then have that $\dist(p_1, \alpha) = \dist(p_2, \alpha) = 0 \leq r,$ so that $p_1, p_2 \in X_{\tau_\alpha, r}.$ We claim that the subset $X_{\tau_\alpha,r}$ is not path-connected.

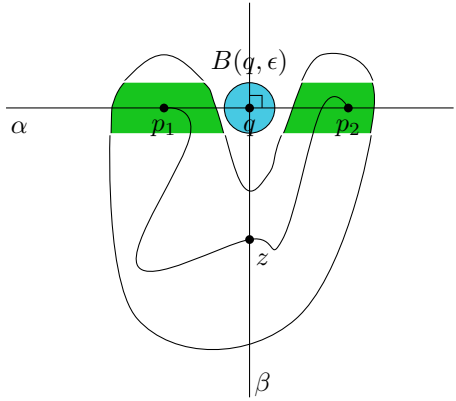
\begin{figure}[h]
\centering
\begin{tikzpicture}[scale=0.35]

\draw [name path = shape, thick, fill = nina_green]  plot[smooth, tension=.7] coordinates {(1, 8) (1.5, 9) (3, 10) (4.5, 9) (5, 8) (6, 5) (7, 5.5) (7.5, 7) (9, 10) (11, 8) (8, 0) (2, 0) (1, 8)};
\node at (4, 1) {$X$};
\draw [name path = rectangle, thick, white]  plot coordinates {(-1, 9) (12, 9) (12, 7) (-1, 7) (-1, 9)};
\tikzfillbetween[of=shape and rectangle]{white}

\node at (6.25, 8)[shape=circle, draw, inner sep=6.75pt, fill=nina_blue] {};
\node at (6.25, 9.75) {$B(q, \epsilon)$};

\node [scale = 3] (p1) at (3, 8) {.};
\node [scale = 3] (p2) at (10, 8) {.};
\node [scale = 3] (q) at (6.25, 8) {.};
\draw (-3, 8) -- (14, 8);
\node at (3, 7.25) {$p_1$};
\node at (10, 7.25) {$p_2$};
\node at (6.25, 7.25) {$q$};
\node at (-2.5, 7.25) {$\alpha$};

\draw [name path = shape]  plot[smooth, tension=.7] coordinates {(3, 8) (4, 7) (2, 2) (6.25, 3) (7.5, 3) (9, 8) (10, 8)};

\draw (6.25, 12) -- (6.25, -3);
\node at (6.75, -2.5) {$\beta$};
\node (alpha1) at (6.25, 12) {};
\draw [right angle symbol = {p1}{p2}{alpha1}];

\node [scale = 3] (z) at (6.25, 3) {.};
\node at (6.75, 2.25) {$z$};
\end{tikzpicture}
\caption{For a concave shape $X$, there exists a tubular filtration line $\alpha$ so that the resulting 0-dimensional PD sees multiple path-connected components (in green). Note that the path in the figure can exist, but it cannot lie completely in the particular sublevel set (in green).}
\label{fig_convexity_thm_concave}
\end{figure}

Let us assume otherwise, i.e., that that $X_{\tau_\alpha,r}$ is path-connected. Equivalently, there is a path connecting $p_1$ and $p_2$ and which is contained in $X_{\tau_\alpha,r}$. Then such a path would have to intersect the hyperplane $\beta$ passing through $q$ and orthogonal to line $\alpha$. To see why, we first note that the complement $\mathbb{R}^d \setminus \beta$ of the hyperplane is disconnected with two connected components, each containing one of the two points $p_1$ and $p_2$. If the path would not intersect the hyperplane, it would be contained in the complement of the hyperplane, but not entirely contained in one of the connected components, which yields a contradiction to the path being connected. By construction, this point of intersection $z \in X$ lies on the path between $p_1$ and $p_2$, and therefore $z \in X_{\tau_\alpha, r}$, or equivalently, $\tau_\alpha(z) = \dist(z, \alpha) \leq r$. Since $z$ is also contained in the hyperplane $\beta$ orthogonal to the line $\alpha$ and passing through $q$, we have that $\dist(z, q) = \dist(z, \alpha) \leq r,$ i.e., $z \in B(q,r)\subset \mathbb{R}^d\setminus X,$ which is a contradiction to $z \in X.$ Therefore, for any $0 \leq r\leq \epsilon$, the set $X_{\tau_\alpha,r}$ is not path-connected, so that the 0-dimensional PD on the tubular filtration with respect to $\alpha$ contains at least two intervals.

\end{proof}

\begin{remark}[Relationship with the height filtration]
\label{remark_tubular}

To illustrate the need for the tubular filtration, we discuss how it compares to the height filtration that is well-established in the literature. For a given shape $X \in \mathbb{R}^d$ and a unit vector $v \in S^{d-1}$, the height function $\eta_v: \mathbb{R}^d \rightarrow \mathbb{R}$ is defined via the scalar product, $\eta_v(x) = x \cdot v.$ If we consider the hyperplane that is orthogonal to the vector $v$ and passes through the origin $(0, 0, \dots, 0) \in \mathbb{R}^d$, the sublevel set $X_{\eta_v, r}$ corresponds to the area in $X$ above or in the hyperplane and below or at height $r \in \mathbb{R}$, and the complete area of $X$ below the hyperplane (where the scalar product is negative) (Figure~\ref{fig_tubular}, column 1). Note that it is possible to recenter the shape at the origin, so that the hyperplane does not need to pass through the origin.

We note that 0-dimensional PH with respect to the height filtration can help us detect \emph{some} concavities in $\mathbb{R}^d,$ what is the case for panel (a) in Figure~\ref{fig_tubular}, where we clearly see multiple path-connected components in green. However, for shapes in $\mathbb{R}^2$ where a source of concavity is a hole within the shape, such as the annulus-like shape in row 2 of Figure~\ref{fig_tubular}, the sublevel sets with respect to the height filtration will only see a single path-connected component, see panel (d) in Figure~\ref{fig_tubular}. Indeed, there is no unit vector $v$ for which the sublevel set $X_{\eta_v, r}=\{x \in X \mid \eta_v(x) = x \cdot v \leq r\}$ contains more than one path-connected component. If we adjust the definition of the filtration function and let $\eta'_v(x) = |x \cdot v|,$ the sublevel set $X_{\eta'_v, r}$ corresponds to the area in $X$ above and below the hyperplane, and within height $r \in \mathbb{R}$ in both directions. In other words, $X_{\eta'_v, r}$ is the area of $X$ within a given distance from the hyperplane. 0-dimensional PH with respect to the absolute height function $\eta'_v$ enables us to detect any concavity in $\mathbb{R}^2,$ see (b) and (e) in Figure~\ref{fig_tubular}, since the sublevel set consists of multiple path-connected components (in green). In $\mathbb{R}^2,$ the tubular filtration corresponds to the absolute height filtration. However, neither the height nor the absolute height filtration can detect concavity in higher dimensions. Indeed, consider a sphere-like shape as an example concave shape in $\mathbb{R}^3$ (Figure~\ref{fig_tubular}, row 3). The sublevel sets $X_{\eta_v, r}$ and $X_{\eta'_v, r}$ will result in a single path-connected component, see the green areas respectively in panels (g) and (h) in Figure~\ref{fig_tubular}. On the other hand, the area within a distance from some \emph{line} (points in $X$ that are within a given tube) can result in two (disconnected) disks on polar opposites on the sphere, see panel (i) in Figure~\ref{fig_tubular}.

However, we note that there are likely alternative approaches that can rely on PH with respect to the (absolute) height filtration to detect concavity. One possibility would be to also consider homology in higher dimensions (although, we note that $0$-dimensional PH is faster to compute, see Appendix~\ref{app_thm_computational}). For shapes in $\mathbb{R}^2,$ it is sufficient to consider the $0$- and $1$-dimensional PH with respect to the height filtration. Indeed, although the annulus-like shape (Figure~\ref{fig_tubular}, row $2$) does not see multiple path components with respect to any height filtration function, there is a $1$-dimensional hole which points to a concavity. This, however, does not generalize to higher dimensions: Indeed, a sublevel set of a sphere with respect to any height filtration will result in a spherical cap which has trivial homology in dimensions $0, 1$ and $2$, see panel (g) in Figure~\ref{fig_tubular}. On the other hand, the sublevel set of the absolute height function on the sphere will always consist of a single path-connected component, but we can capture $1$-dimensional holes, see panel (h) in Figure~\ref{fig_tubular}. Another interesting example of a concave surface in $\mathbb{R}^3$ is a ball with a dent on the north pole, i.e., a crater. This concavity cannot be detected with 0-dimensional PH with respect to any (absolute) height filtration. For the example surface, an interesting direction would be looking at the surface "from the top" (horizontal hyperplane), but PH would start by seeing a circle, and then the crater itself --- always a single path-connected component. However, $1$-dimensional PH with respect to this filtration would capture a hole, implying that the surface is concave.

Another possibility could be to study (the computable) multiparameter $0$-dimensional PH by scanning shapes from multiple directions \emph{simultaneously}. For the sphere, $0$-dimensional PH would capture the two path-connected components with the bi-filtration that looks at the shape with respect to the horizontal hyperplane denoted in the panel (h) in Figure~\ref{fig_tubular}, and the orthogonal hyperplane passing through the shape. We note that while the theory and computations for multiparameter persistence are hard, there have been some recent advances, see, e.g., \cite{rivet}. This is similar to slicing the shape with a hyperplane, and then studying the single-parameter $0$-dimensional PH of the slice. The alternative approaches that we briefly discuss here are an interesting avenue for further work.

\begin{figure}[h]
\begin{tabular}{lll}
\toprule
height & absolute height & tubular \\
\midrule
scalar product & distance from hyperplane & distance from line \\
$\eta_v(x)=x \cdot v$ & $\eta'_v(x)=|x \cdot v|=\dist(x, \alpha)$ & $\tau_\alpha(x)=\dist(x, \alpha) $ \\
\midrule

(a) \includegraphics[width = 3.4cm]{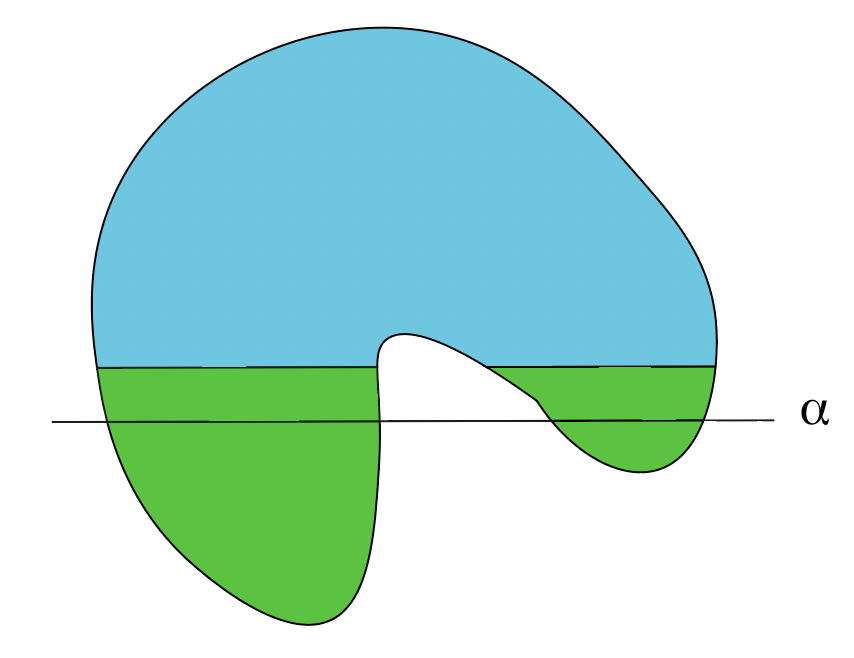} &
(b) \includegraphics[width = 3.1cm]{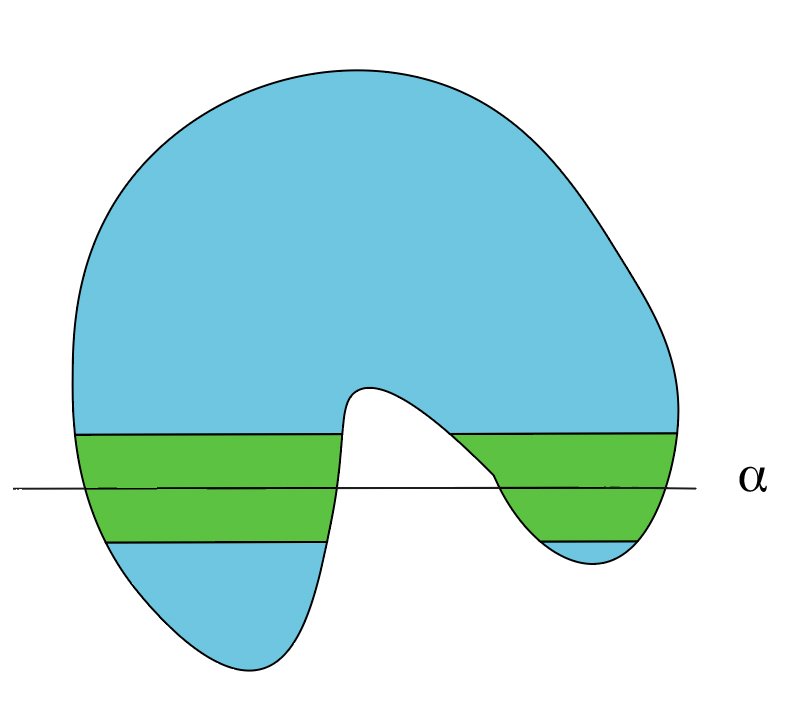} &
(c) \includegraphics[width = 3.1cm]{figures/convexity/shapes/fig9_b_c_n.png} \\

(d) \includegraphics[width = 3.75cm]{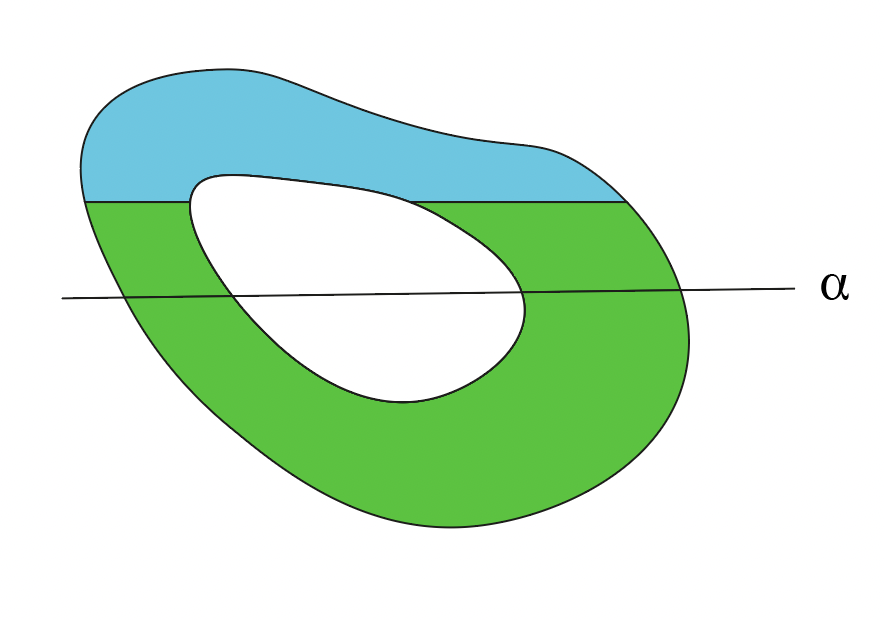} &
(e) \includegraphics[width = 3.75cm]{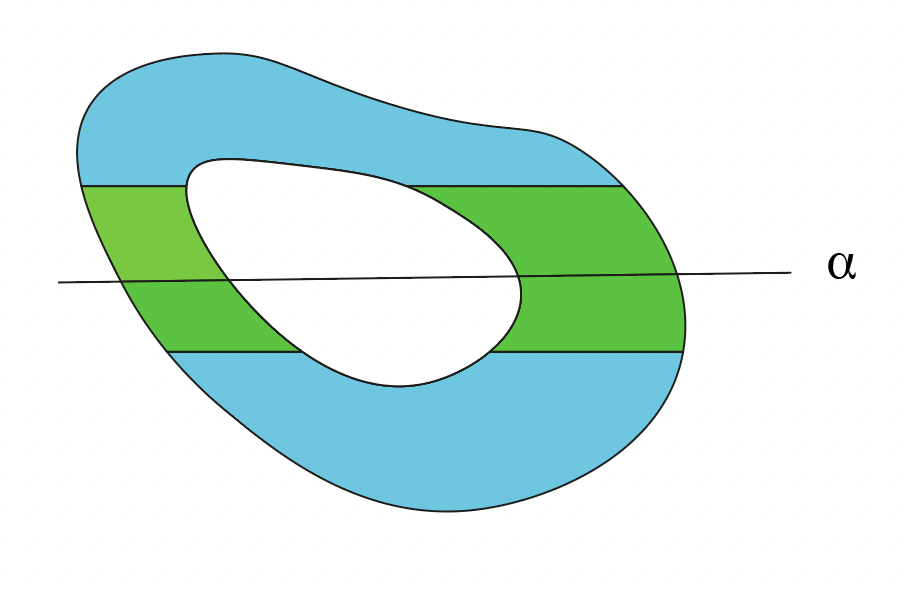} &
(f) \includegraphics[width = 3.75cm]{figures/convexity/shapes/fig9_e_f.png} \\

(g) \includegraphics[width = 3.75cm]{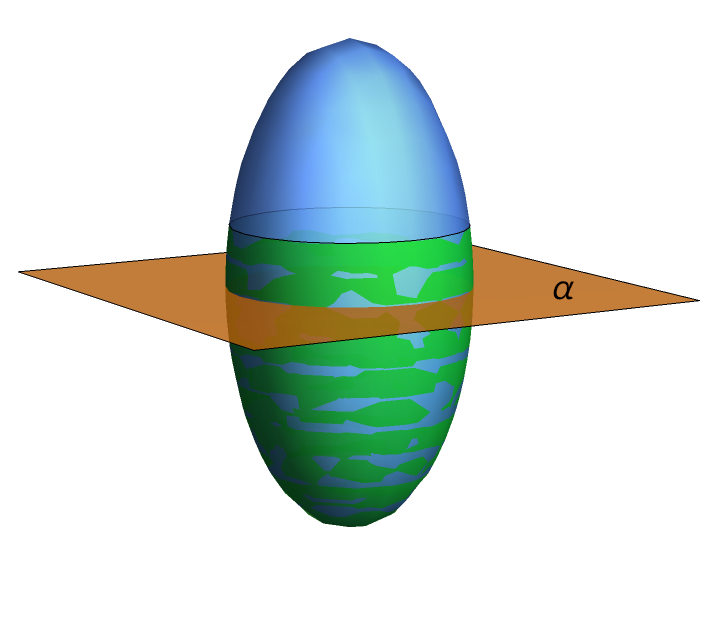} &
(h) \includegraphics[width = 3.75cm]{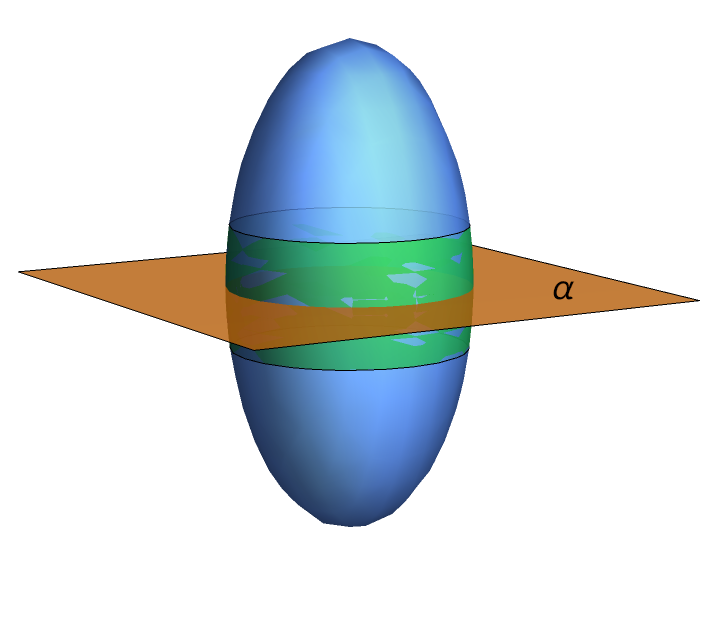} &
(i) \includegraphics[width = 3.75cm]{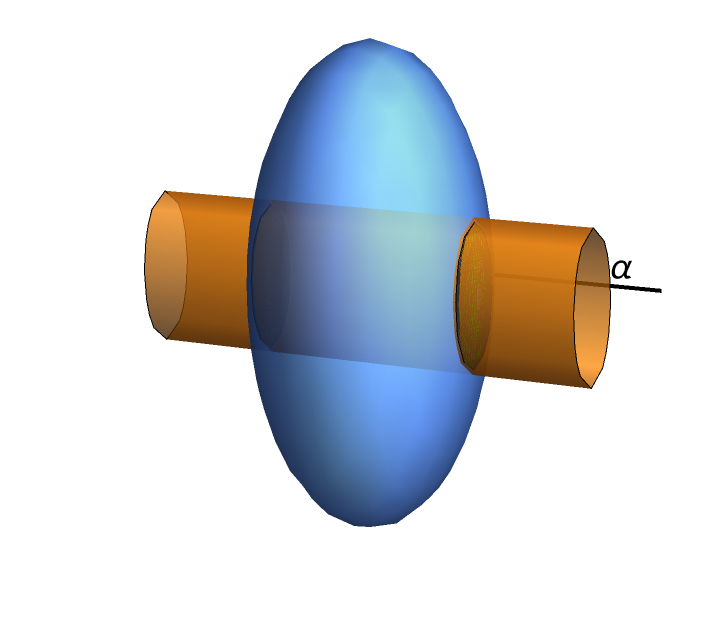} \\

\bottomrule
\end{tabular}
\caption{Sublevel sets of the height (column 1), absolute height (column 2) and tubular function (column 3) for three example concave shapes in $\mathbb{R}^2$ and $\mathbb{R}^3.$  Concavity is detected with multiple 0-dimensional persistence intervals, which reflect the multiple path-connected components in the filtration.
The sublevel set (in green) has multiple path-connected components for the height function only for the shape in row 1, for the absolute height only for shapes in rows 1 and 2, whereas for the tubular function this is the case for any concave shape in $\mathbb{R}^d$.}
\label{fig_tubular}
\end{figure}

\end{remark}

\begin{remark}[Relationship with the Persistent Homology Transform]
\label{remark_pth} 

There is a lot of work done on studying the so-called Persistent Homology Transform (PHT), which is given by considering the PH on the height filtration with respect to all unit vectors \cite{turner2014persistent, curry2018many, ghrist2018persistent}. Such a topological summary that has been shown to be a sufficient statistic for probability densities on the space of triangulizable subspaces of $\mathbb{R}^2$ and $\mathbb{R}^3$ \cite{turner2014persistent}.

For practical purposes it would not be feasible to have to consider all unit vectors in the PHT. Luckily, there are known results on the sufficient number of directions \cite{curry2018many}. In computational experiments in \cite{turner2014persistent} on the MPEG-7 silhouette database of simulated shapes in $\mathbb{R}^2$ \cite{sikora2001mpeg, latecki2000shape} and point clouds in $\mathbb{R}^3$ obtained from micro-CT scans of heel bones \cite{boyer2015new}, the PHT is approximated by looking respectively at $64$ evenly spaced directions and $162$ directions constructed by subdividing an icosahedron. Furthermore, $0$- and/or $1$-dimensional PH with respect to the height and/or related radial filtration has also been used as a $3$-dimensional shape descriptor in \cite{carriere2015stable}, for analysis of brain artery trees \cite{bendich2016persistent}, or classification of MNIST images of handwritten digits \cite{garin2019topological}.

We note that the theoretical results related to the Persistent Homology Transform focus on a complete description of shapes, whereas here we are interested in investigating to what extent PH can detect convexity and concavity. 

\end{remark}

\begin{remark}[Detection of convexity with PH in practice] 
\label{remark_convexity_practice} 

To calculate PH in practice, the sublevel sets $K_r$ need to be approximated with simplicial or cubical complexes (see Section~\ref{section_ph_subsection_approx}). 
The PH pipeline in our computational experiments for convexity detection (Section~\ref{section_convexity}) relies on cubical complexes, but it is possible to do so also with the Vietoris-Rips complex relying on geodesic distances. For details, see Appendix~\ref{app_convexity_pipeline}, where we conclude that it is important to ensure that the singular homology of each of the sublevel sets is properly reflected with the homology of the complex. For convexity detection, this means that the complexes of concave shapes should also see multiple connected components (that do not connect with a simplex).

Furthermore, we note that, to simply differentiate between convex and concave shapes (binary classification problem), it would be sufficient to only consider 0-dimensional \emph{homology} of the intersection of $X$ with the line $\alpha,$ i.e.,
$$X_{\tau_\alpha, 0} = \{ x \in X \mid \tau_\alpha(x) = \dist(x, \alpha) \leq 0 \} = X \cap \alpha,$$
which is easier to compute than the multi-scale PH. This intersection is a line segment for convex shapes, and for concave shapes it is a union of disconnected segments on a line. Therefore, convex shapes will have $\beta_0(X \cap \alpha)=1$, and concave shapes will have $\beta_0(X \cap \alpha) > 0$ (one could think of this as persistence intervals $[0, +\infty)$ that all have the same lifespan). However, in practical applications, it is often useful to capture a more detailed information about concavity. For example, for the plant morphology application we consider in Appendix~\ref{app_real_data}, the goal is to capture a continuous measure of concavity (regression problem). Then, the (different) lifespans of the second (or also third, fourth, ...) most-persisting connected component can provide important additional information. Moreover, in applications $X$ is typically finite, e.g., a point cloud, so that one would still need to approximate $X \cap \alpha$ (points on a line segment) with a complex in order to calculate the homology (see paragraph above). In other words, we would need to choose an appropriate scale $r \in \mathbb{R}$ that would ensure that the complex faithfully reflects the homology of $X \cap \alpha$, which is a non-trivial task.

\end{remark}

\subsection{Number of holes}
\label{app_thm_holes}

In our pipeline to detect the number of holes, we use the alpha complex, for which several theoretical guarantees have been proven.

The Nerve Lemma (see, e.g., \cite{edelsbrunner2010computational}) guarantees that the alpha complex of a set of points has the same homology-type as the space obtained by taking unions of balls of a certain radius centered around the points. Whether this union of balls has the same homology-type as the space from which the points are sampled depends on properties of the sample. If the sample is dense enough, then it has been shown that, for a suitable value of the scaling parameter, the alpha complex has the same number of holes as the original space, for instance under the assumption on the space being a smooth manifold \cite{niyogi2008finding}. For ease of reference, we reproduce here the result from \cite{niyogi2008finding}.

\begin{theorem}[Number of holes]
\label{thm_holes}
Let $M$ be a compact smooth manifold, and $X$ a set of points sampled uniformly at random from $M.$ Then there exists $r \in \mathbb{R}$ such that the homology of the alpha simplicial complex $\alpha(X, r)$ is isomorphic to the singular homology of the underlying manifold $M.$
\end{theorem}

\begin{proof} 
\cite[Theorem 3.1]{niyogi2008finding} implies that there exists $r \in \mathbb{R}$ such that the singular homology of the $\cup_{x \in X} B(x, r)$ is isomorphic to the singular homology of $M.$ By the Nerve Lemma we then know that the simplicial homology of the alpha complex $\alpha(X, r)$ is isomorphic to the singular homology of  $\cup_{x \in X} B(x, r)$.
\end{proof}

The alpha complex is known to approximate the Vietoris--Rips complex, in the sense that the respective persistence modules are interleaved, see, e.g. \cite{edelsbrunner2010computational}.

\subsection{Curvature}
\label{app_thm_curvature}

Here we reproduce the theoretical guarantee provided in \cite{bubenik2020persistent}.

\begin{theorem}[Curvature]
\label{thm_curvature}
Let $M$ be a manifold with constant curvature $\kappa$, and $D_\kappa$ be a unit disk on $M.$ Let further $X$ be a point cloud sampled from $X,$ according to the probability measure proportional to the surface area measure. Then, PH of $X$ recovers $\kappa.$
\end{theorem}

\begin{proof} 
Given $\kappa$, \cite[Theorem 14]{bubenik2020persistent} establishes an analytic expression for the persistence ($p = d/b$) of triangles to the curvature $\kappa$ of the underlying manifold. This function is continuous and increasing, so that persistence recovers curvature.
\end{proof}

\subsection{Computational complexity}
\label{app_thm_computational}

In this section, we discuss how our pipelines are affected by the size $n$ of the point cloud and the dimension $d$ of the embedding space (which are also related, since typically exponentially more points are needed to properly sample a shape in higher dimensions).

There exist several efficient algorithms for the computation of PH, many coming with heuristic guarantees on speed-ups for the computation (see survey \cite{otter2017roadmap} for an overview, and references therein). For the purposes of this discussion, we will focus on the standard algorithm, which has a computational complexity which is cubical in the number $N$ of simplices contained in the filtered complex \cite{zomorodian2005computing}, i.e., the computational complexity is $\mathcal{O}(N^3).$ Thus, to better understand how our pipelines generalize to higher-dimensional point clouds, in the following we explain how the sizes $N$ of the different types of complexes that we consider are affected by the size $n$ and dimension $d$ of the point cloud.

\paragraph{Number of holes} For the detection of the number of holes, in our experiments we relied on alpha simplicial complex. In the worst case, the size $N$ of the alpha complex is $\mathcal{O}\left(n^{\lceil d/2 \rceil}\right)$.

\paragraph{Curvature} We used the Vietoris-Rips simplicial complex for curvature detection. The size $N$ of the Vietoris--Rips complex is $\mathcal{O}\left ( n^{k+2}\right )$, where $k$ is the maximum PH degree that we are interested in computing.

In general, the choice between the alpha and Vietoris-Rips simplicial complex therefore depends on the point cloud size $n$ and dimension $d$, and the highest homological dimension $k$ that is of interest for the given problem at hand. For point clouds with a given number $n$ of points, the alpha complex is better suited in lower dimensions ($d=2,3$), and provided that the point cloud is embedded in Euclidean space.

\paragraph{Convexity} To detect convexity, we relied on cubical complexes. For data sets that have an inherently cubical structure, using cubical complexes may yield
significant improvements in both memory and runtime efficiency \cite{wagner2012efficient}. This is particularly true for high dimensional data, since the ratio between the size of the Vietoris-Rips simplicial complex compared to a cubical complex is exponential in dimension $d$ \cite{chen2011persistent}.

In our construction, given a point cloud in $\mathbb{R}^d$ and a fixed $c \in \mathbb{N}$, we bin the points into $c^d$ cubes of the same volume. In the worst case, the size $N$ of the cubical complex of the resulting structure is $\mathcal{O}(3^d c^d)$ (and thus does not depend on the number $n$ of point cloud points).

In addition, it is important to note that our convexity detection pipelines only uses the $0$-dimensional persistent homology, which has a reduced complexity since one only needs to construct the complex up to dimension $1$. It is fairly easy to compute 0-dimensional PH in near-linear time with respect to the number $N$ of simplices by using union-find data structures \cite{edelsbrunner2010computational, wagner2012efficient, dlotko2014simplification}. For this reason, it is an important advantage of our pipeline that it only relies on $0$-dimensional PH, without needing to calculate PH in higher homological dimensions.

Detecting convexity, however, poses additional challenges. Testing convexity is fundamentally a hard problem in high dimensions, related to the hardness of computing convex hulls in high dimensions, and unfortunately we cannot hope for free lunch. In our PH convexity detection pipeline, unlike for the detection of the number of holes or curvature, we calculate PH across multiple tubular-filtration lines, whose number also grows with the dimension $d$ since sufficiently many filtrations need to be considered (and the same would be the case - we would have to consider multiple tubular directions, if we considered simplicial instead of cubical complexes, see Figure~\ref{fig_convexity_simplicial_cubical}). This could be circumvented by considering (quasi-)random lines. To conclude, specifying a desired computation budget and number of filtrations in advance (leading to a corresponding accuracy tradeoff), our PH pipeline can be used to obtain fast estimates of convexity. It can also be used to compute a continuous measure of convexity (as we  demonstrate on the real-world FLAVIA data set of leaf images in Appendix~\ref{app_real_data}), or convexity at a given resolution, depending on the resolution of the filtration, which in some cases may be more useful than the binary label (convex or concave).

\section{Experimental details}
\label{app_exp}

\subsection{Reproducibility and computer infrastructure}
\label{app_exp_rep}

The data and code developed for this research are made publicly available at \url{https://github.com/renata-turkes/turkevs2022on}. All our computations were conducted using a 2.7Ghz vCPU core from a DGX-1 + DGX-2 station.

\subsection{Hyperparameter tuning and training procedure for the individual pipelines}
\label{app_exp_pipelines}

In this section, we provide more details about the pipelines that were compared in the computational experiments: 
\begin{itemize}[leftmargin=*]
\item SVM on persistent homology features (PH), 
\item simple machine learning (ML) baseline - SVM on distance matrices, 
\item fully connected neural network (NN) on distance matrices, and
\item PointNet (PointNet) on raw point clouds. 
\end{itemize}

For each pipeline, we list the hyperparameters that were tuned. To ensure a fair comparison of the different approaches, we used the same train and test data across all the pipelines. We used \texttt{sklearn GridSearchCV} based on cross validation with 3 folds and random splits, and returned the hyperparameter values that resulted in the highest accuracy for classification problems (Section~\ref{section_holes} and Section~\ref{section_convexity}), or the lowest mean squared error for regression problems (Section~\ref{section_curvature}). We also list relevant software and licenses.

\paragraph{PH} The general steps to extract PH features are visualized in Figure~\ref{fig_ph_pipeline}. To calculate PH in Section~\ref{section_holes} and Section~\ref{section_convexity} we use \texttt{GUDHI} \cite{gudhi:urm}, and in Section~\ref{section_curvature} we use \texttt{Ripser} \cite{bauer2021ripser, tralie2018ripser}, which are persistent homology libraries in Python, available under the MIT (GPL v3) license. For the DTM filtration (Section~\ref{section_ph_subsection_filtration}) in Section~\ref{section_holes}, we choose $m = 0.03,$ so that $0.03 \times 1\,000 = 30$ nearest neighbors are used to calculate the filtration function. Grid search is performed to choose the best persistence signature and classifier or regressor as described below.

\begin{itemize}
\item In Section~\ref{section_holes} and Section~\ref{section_curvature}, we select between:
\begin{itemize}[leftmargin=*]
\item simple signature of 10 longest lifespans, 
\item persistence images with resolution $10 \times 10$, 
bandwidth $\sigma \in \{0.1, 0.5, 1, 10\}$, 
and weight function $\omega(x, y) \in \{1, y, y^2\}$, and 
\item persistence landscapes with resolution of $100$, and considering the longest 1, 10 or all persistence intervals. 
\end{itemize}

\item We use SVM (\texttt{sklearn} \texttt{SVC} and \texttt{SVR} for classification and regression respectively) on the PH signature, with the regularization parameter $C \in \{0.001, 1, 100\}$. The latter tunes the trade off between correct classification of training data and maximization of the decision function's margin. 
\end{itemize}

\begin{figure}[h]
\centering

\fbox{
\begin{tikzpicture}[auto, >=latex'] 
\tikzstyle{block}=[draw, shape=rectangle, rounded corners=0.5em, fill = blue!15!white];

\node [block] (pc) {point cloud};
\node [block, right = 7em of pc] (filtration) {filtration};
\node [block, right = 7em of filtration] (pd) {$0$ and $1$-dim PD};
\node [block, right = 1em of pd] (ph) {0- and 1-dim PH};

\node [above = -0em of pc] {\includegraphics[height = 0.095\linewidth]{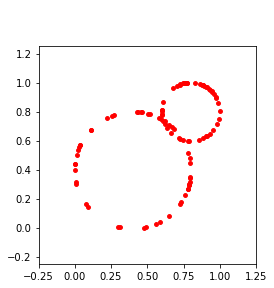}};
\node [above = -0em of filtration] {\includegraphics[height = 0.095\linewidth]{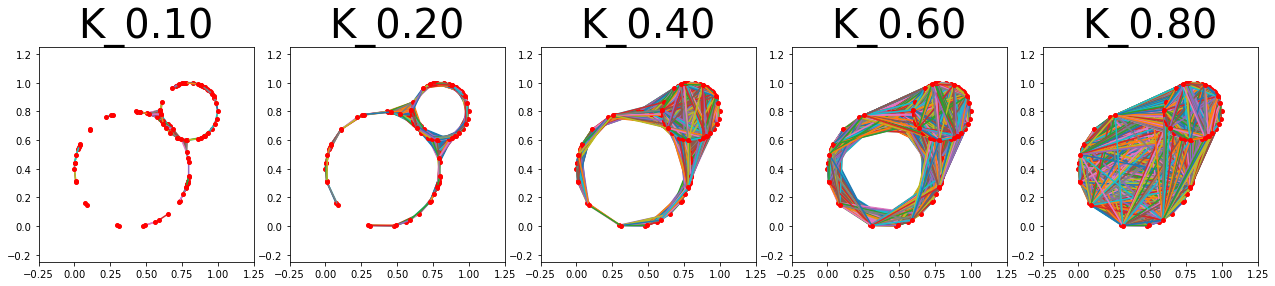}};
\node [above = -0em of pd] {\includegraphics[height = 0.095\linewidth]{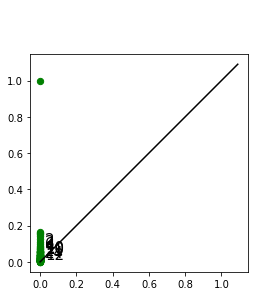} \includegraphics[height = 0.095\linewidth]{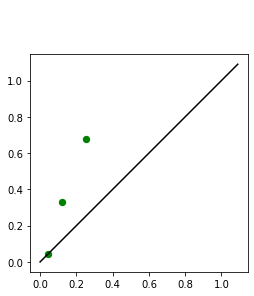}};
\node [above = -0em of ph] {\includegraphics[height = 0.085\linewidth]{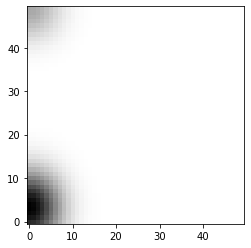} \includegraphics[height = 0.085\linewidth]{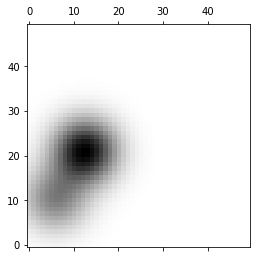}};

\path[draw,->] (pc) edge (filtration)
(filtration) edge (pd)
(pd) edge (ph);
\end{tikzpicture}
}

\caption{Persistent homology features. To calculate PH for the given point cloud in $\mathbb{R}^2,$ we first construct a filtration $\{ K_r \}_{r \in \mathbb{R}}$ which approximates $X$ at different scales $r \in \mathbb{R},$, where $K_r$ is the Vietoris-Rips simplicial complex. $0$-dimensional PD has one persistent cycle, reflecting the single  component, and a number of short cycles that correspond to the individual point-cloud points that are connected to other ones early in the filtration. $1$-dimensional PD summarizes the two holes, whose birth and death values respectively reflect the sparsity along the hole and the size of the hole, as these are the scales $r \in \mathbb{R}$ at which the hole appears and when it is filled in within the filtration. PDs are then represented by PIs, which are vector summaries that can be used in statistical learning frameworks, but many other signatures (denoted, in general, with PH) can be used.}
\label{fig_ph_pipeline}
\end{figure}
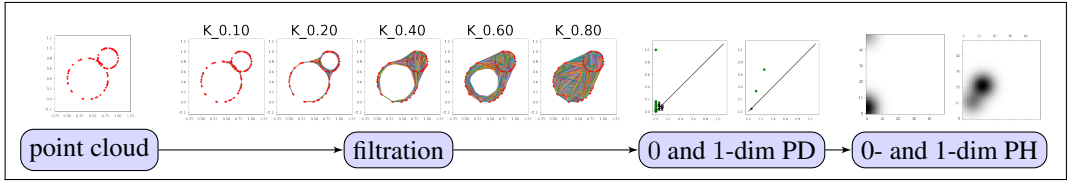

\paragraph{ML} In our experiment, the input for the simple machine learning (ML) pipeline is the normalized matrix of pairwise distances between point-cloud points. For a given point cloud $X=\{x_1, \ldots, x_n \}\subset \mathbb{R}^d$, the corresponding distance matrix is the $n \times n$ matrix $D \in \mathbb{R}^{n \times n},$ with entries $D_{ij}=\dist(x_i, x_j)$ corresponding to the Euclidean distance for the detection of the number of holes (Section~\ref{section_holes}) or convexity (Section~\ref{section_convexity}), and hyperbolic, Euclidean or spherical distance for curvature detection (Section~\ref{section_curvature}). We take the entries above the diagonal flattened into a vector. Since the dimension of distance matrices scales with the square of the number of points, we work with subsamples of $100$ distinct random points from each point cloud. Similarly as above, we use cross validation to choose the SVM regularization parameter among $C \in \{0.001, 1, 100\}$.

We note that while a distance matrix can be taken as input to a classifier, it depends on the particular and arbitrary labeling of the points in the point cloud and hence it does not account for the label symmetry of point clouds.

\paragraph{NN} The normalized distance matrices are also fed to the multi-layer perceptrons (MLPs). We consider the following hyperparameters:
\begin{itemize}
\item depth in $\{1, 2, 3, 4, 5\}$ (only for NN deep),
\item layer widths in $\{64, 256, 1\,024, 4\,096\}$,
\item learning rate in $\{0.01, 0.001\}$,
\end{itemize}
that are selected through a grid search, with each parameter setting trained for 2 epochs. We use a soft-max activation function, and cross entropy and mean squared error as loss functions for classification (Section~\ref{section_holes} and Section~\ref{section_convexity}) and regression (Section~\ref{section_curvature}) problems, respectively. Batch normalization (with zero momentum) and a drop out (with a rate of 0.5) are applied after every (input or hidden) layer.

\paragraph{PointNet} PointNet \cite{qi2017pointnet} is a neural network that takes point clouds as inputs, and is inspired by the invariance of point clouds to permutations and transformations. It incorporates fully-connected MLPs to approximate classification functions, and convolutional layers to capture geometric relationships between features. 

In our experiments, we rely on the PointNet model from \texttt{keras} \cite{griffiths2020point} under Apache License 2.0. This model implements the architecture from the original PointNet paper \cite{qi2017pointnet}, which is supplemented with a publicly available code \cite{qi2017pointnet_code}, licensed under MIT. We use grid search to tune: 
\begin{itemize}
\item number of filters in $\{32, 64\}$, 
\item learning rate in $\{0.01, 0.001\}$. 
\end{itemize}
For each of the problems we consider, the model is trained from scratch using the training set described in the corresponding section. Unlike in the original paper, we do not augment the data during training by randomly rotating the object or jittering position of each point by a Gaussian noise, in order to ensure a fair comparison with the other pipelines.

\FloatBarrier 
\section{Additional experimental details for number of holes}
\label{app_holes}

\subsection{Data transformations}
\label{app_holes_transformations}

To test the noise robustness of the different pipelines, in Section~\ref{section_holes} we consider the test data consisting of the original point clouds, or point clouds under different transformations (Figure~\ref{fig_holes_results}). A detailed description of the data transformations is given in Table~\ref{tab_transformations}, and the transformations are visualized on an example point cloud in Figure~\ref{fig_transformations}. To define reasonable values for the data transformations, we took inspiration from the affNIST\footnote{\url{https://www.cs.toronto.edu/~tijmen/affNIST/}} data set of MNIST images under affine transformations.

\begin{table}[h]
\caption{Data transformations.}
\label{tab_transformations}
\centering
\begin{tabular}{p{0.18\linewidth}p{0.75\linewidth}} 
\toprule
\textbf{Transformation} & \textbf{Definition of transformation} \\
\midrule
\noalign{\smallskip}
rotation & Clockwise rotation by an angle chosen uniformly from $[-20, 20]$ degrees clockwise. \\ \midrule
translation & Translation by random numbers chosen from $[-1, 1]$ for each direction. \\ \midrule
stretch & Scale by a factor chosen uniformly from $[0.8, 1.2]$ in the $x$-direction. The other coordinates remain unchanged, so that the point cloud is stretched. Stretching factor of $0.8$ results in shrinking the point cloud by $20\%,$ and the factor of $1.2$ makes it $20\%$ larger. \\ \midrule
shear & Shear by a factor chosen uniformly from $[-0.2, 0.2].$ A shearing factor of 1 means that a horizontal line turns into a line at $45$ degrees. \\ \midrule
Gaussian noise & Random noise drawn from normal distribution $\mathcal{N}(0, \sigma)$ with the standard deviation $\sigma$ uniformly chosen from $[0, 0.1]$ is added to the point cloud.  \\ \midrule
outliers & A percentage, chosen uniformly from $[0, 0.1]$, of point cloud points are replaced with points sampled from a uniform distribution within the range of the point cloud. \\ 
\bottomrule
\end{tabular}
\end{table}

\begin{figure}[h]
\centering
\begin{tabular}{cccccccc}
original && translation & rotation & stretch & shear & Gaussian & outliers \\
\includegraphics[width = 0.11\linewidth]{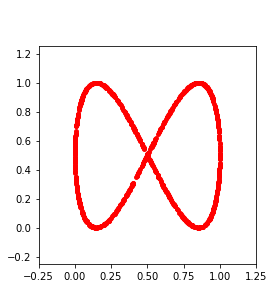} &&
\includegraphics[width = 0.11\linewidth]{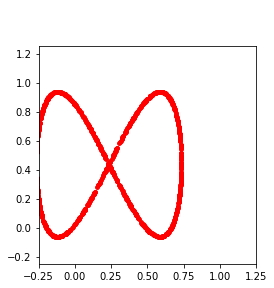} &
\includegraphics[width = 0.11\linewidth]{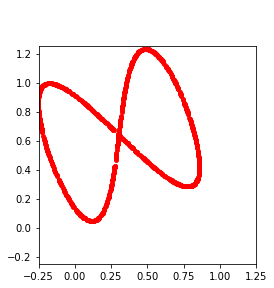} &
\includegraphics[width = 0.11\linewidth]{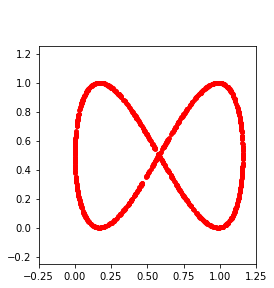} &
\includegraphics[width = 0.11\linewidth]{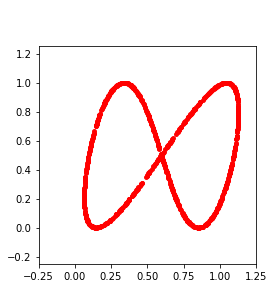} &
\includegraphics[width = 0.11\linewidth]{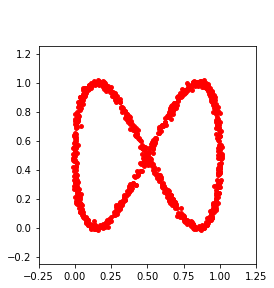} &
\includegraphics[width = 0.11\linewidth]{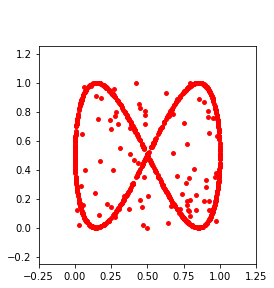} \\
\end{tabular}
\caption{An example point cloud under the considered transformations.}
\label{fig_transformations}
\end{figure}

\subsection{Pipeline}
\label{app_holes_pipeline}

Figure~\ref{fig_holes_pipeline} visualizes the PH pipeline. To reduce the computation times, we approximate point clouds at each scale with the alpha simplicial complex (discussed in Section~\ref{section_ph}, and in particular, in Section~\ref{section_ph_subsection_approx}). The DTM filtration on the point-cloud points is defined as the average distance from a number of nearest neighbors. Therefore, outliers appear only late in the filtration, so that their influence is smoothed out to a great extent. For the example point cloud in the figure, the 1-dimensional PD consists of four persistence intervals with non-negligible persistence or lifespan (PD points far from diagonal) which correspond to the four big holes, and many short persistence intervals that correspond to holes that are seen at some scales due to noise. This is clearly reflected in the vector of sorted lifespans of the 10 most persisting cycles.

\begin{figure}[h]
\centering
\fbox{
\begin{tikzpicture}[auto, >=latex']
\tikzstyle{block}=[draw, shape=rectangle, rounded corners=0.5em, fill = blue!15!white];
\node [block] (pc) {point cloud};
\node [block, right = 10em of pc] (filtration) {DTM filtration};
\node [block, right = 15em of filtration, anchor = east] (pd) {1-dim PD};
\node [block, below = of pd.east, anchor = east] (ph) {1-dim lifespans, PI or PL};
\node [block, below = of pc.west, anchor = west] (ml) {SVM};
\node [above = -0.25em of pc] {\includegraphics[height = 0.125\linewidth]{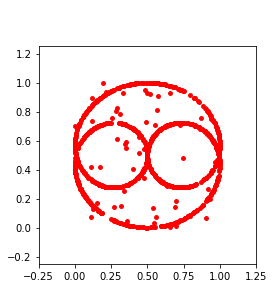}};
\node [above = -0.25em of filtration] {\includegraphics[height = 0.125\linewidth]{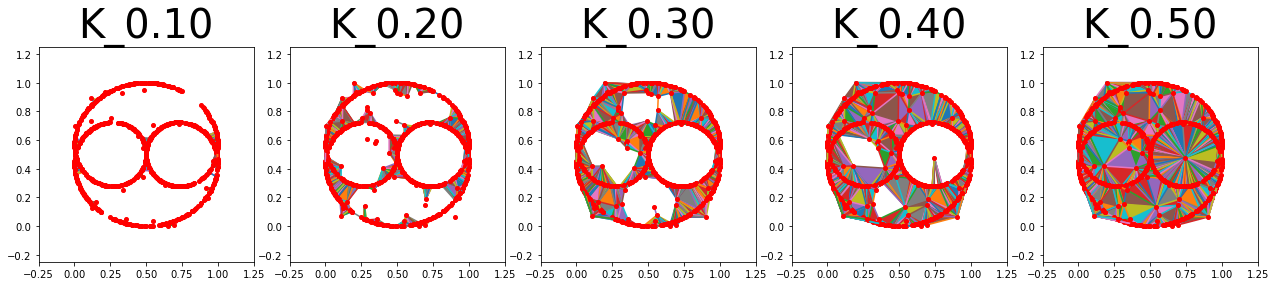}};
\node [above = -0.25em of pd] {\includegraphics[height = 0.125\linewidth]{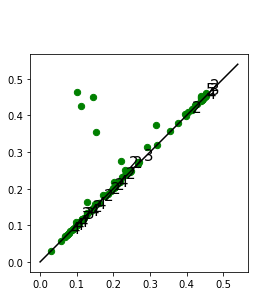}};
\node [below left = 0.3em and -11em of ph] {(0.36, 0.35, 0.33, 0.25, 0.05, 0.05, 0.02, 0.01, 0.01, 0.01)};
\path[draw,->] (pc) edge (filtration)
(filtration) edge (pd)
(pd) edge ([xshift=3.1em, yshift=0.9em]ph)
(ph) edge (ml);
\end{tikzpicture}
}
\caption{Persistent homology pipeline to detect the number of holes.}
\label{fig_holes_pipeline}
\end{figure}
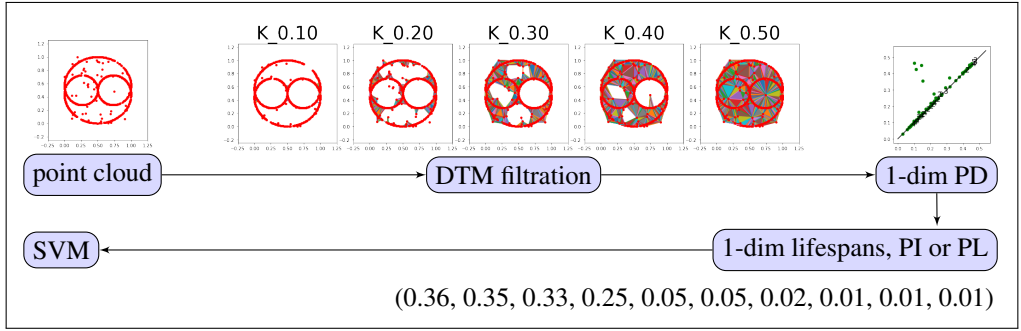

\subsection{Performance across multiple runs}
\label{app_holes_performance}

Table~\ref{tab_holes_accs} provides a detailed overview of the results for the detection of the number of holes, when the experiment is repeated multiple times. The accuracy for PointNet varies for different runs, but in any case, we can clearly see that PH performs the best for each individual run. Note also that the performance of ML, NN shallow and NN deep does not drop under affine transformations, since they take the normalized distance matrices as input.

\begin{table}[h]
\caption{Accuracy across multiple runs for the detection of the number of holes.}
\label{tab_holes_accs}
\centering
\begin{tabular}{lllllllll} 
\toprule

\textbf{Transformation} & \textbf{Run} & \textbf{PH simple} & \textbf{PH} & \textbf{ML} & \textbf{NN shallow} & \textbf{NN deep} & \textbf{PointNet} \\
\midrule

\multirow{5}{*}{original}
& run 1 & 0.94 & \textbf{1.00} & 0.67 & 0.52 & 0.50 & \textbf{1.00} \\
& run 2 & 0.94 & \textbf{1.00} & 0.67 & 0.51 & 0.50 & \textbf{1.00} \\
& run 3 & 0.94 & \textbf{1.00} & 0.67 & 0.56 & 0.50 & \textbf{1.00} \\
\cline{2-8}
& mean & 0.94 & \textbf{1.00} & 0.67 & 0.53 & 0.50 & \textbf{1.00} \\
& std dev & 0.00 & 0.00 & 0.00 & 0.03 & 0.00 & 0.00 \\
\midrule

\multirow{5}{*}{translation}
& run 1 & 0.94 & \textbf{1.00} & 0.67 & 0.52 & 0.50 & 0.23 \\
& run 2 & 0.94 & \textbf{1.00} & 0.67 & 0.51 & 0.50 & 0.17 \\
& run 3 & 0.94 & \textbf{1.00} & 0.67 & 0.56 & 0.50 & 0.21 \\
\cline{2-8}
& mean & 0.94 & \textbf{1.00} & 0.67 & 0.53 & 0.50 & 0.20 \\
& std dev & 0.00 & 0.00 & 0.00 & 0.03 & 0.00 & 0.03 \\
\midrule

\multirow{5}{*}{rotation}
& run 1 & 0.94 & \textbf{1.00} & 0.67 & 0.52 & 0.50 & 0.86 \\
& run 2 & 0.94 & \textbf{1.00} & 0.67 & 0.51 & 0.50 & 0.57 \\
& run 3 & 0.94 & \textbf{1.00} & 0.67 & 0.56 & 0.50 & 0.78 \\
\cline{2-8}
& mean & 0.94 & \textbf{1.00} & 0.67 & 0.53 & 0.50 & 0.74 \\
& std dev & 0.00 & 0.00 & 0.00 & 0.03 & 0.00 & 0.15 \\
\midrule

\multirow{5}{*}{stretch}
& run 1 & 0.97 & \textbf{0.98} & 0.64 & 0.49 & 0.47 & 0.85 \\
& run 2 & 0.97 & \textbf{0.98} & 0.64 & 0.48 & 0.45 & 0.70 \\
& run 3 & 0.97 & \textbf{0.98} & 0.64 & 0.52 & 0.51 & \textbf{0.98} \\
\cline{2-8}
& mean & 0.94 & \textbf{0.98} & 0.64 & 0.50 & 0.48 & 0.84 \\
& std dev & 0.00 & 0.00 & 0.00 & 0.02 & 0.03 & 0.14 \\
\midrule

\multirow{5}{*}{shear}
& run 1 & 0.95 & \textbf{1.00} & 0.66 & 0.54 & 0.51 & 0.94 \\
& run 2 & 0.95 & \textbf{1.00} & 0.66 & 0.51 & 0.50 & 0.72 \\
& run 3 & 0.95 & \textbf{1.00} & 0.66 & 0.56 & 0.51 & 0.96 \\
\cline{2-8}
& mean & 0.95 & \textbf{1.00} & 0.66 & 0.54 & 0.51 & 0.87 \\
& std dev & 0.00 & 0.00 & 0.00 & 0.02 & 0.01 & 0.13 \\
\midrule

\multirow{5}{*}{Gaussian}
& run 1 & 0.94 & \textbf{1.00} & 0.68 & 0.54 & 0.50 & \textbf{1.00} \\
& run 2 & 0.94 & \textbf{1.00} & 0.68 & 0.51 & 0.50 & \textbf{1.00} \\
& run 3 & 0.94 & \textbf{1.00} & 0.68 & 0.56 & 0.51 & \textbf{1.00} \\
\cline{2-8}
& mean & 0.94 & \textbf{1.00} & 0.68 & 0.54 & 0.50 & \textbf{1.00} \\
& std dev & 0.00 & 0.00 & 0.00 & 0.02 & 0.01 & 0.00 \\
\midrule

\multirow{5}{*}{outliers}
& run 1 & 0.82 & \textbf{0.93} & 0.62 & 0.55 & 0.49 & 0.70 \\
& run 2 & 0.82 & \textbf{0.93} & 0.62 & 0.51 & 0.50 & 0.51 \\
& run 3 & 0.82 & \textbf{0.93} & 0.62 & 0.54 & 0.41 & 0.44 \\
\cline{2-8}
& mean & 0.82 & \textbf{0.93} & 0.62 & 0.53 & 0.47 & 0.55 \\
& std dev & 0.00 & 0.00 & 0.00 & 0.02 & 0.05 & 0.13 \\

\bottomrule
\end{tabular}
\end{table}

\subsection{Training curves}
\label{app_holes_training_curves}

Figure~\ref{fig_holes_training_curves} shows the training curves for the NN and PointNet pipelines. The training set performance of MLPs (shallow and deep) continues improving over epochs, but the validation set performance quickly saturates and stops improving after a few epochs. PointNet performs well on this task, already after a short number of training epochs. We do not include training curves for the PH and ML pipelines, as these are based on SVMs.

\begin{figure}[h]
\centering
\begin{tabular}{ccc}
NN shallow & NN deep & PointNet  \\ 
\includegraphics[width = 0.3 \linewidth]{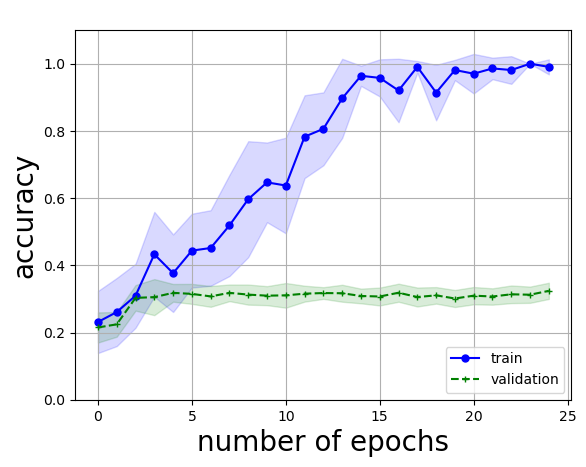} &
\includegraphics[width = 0.3 \linewidth]{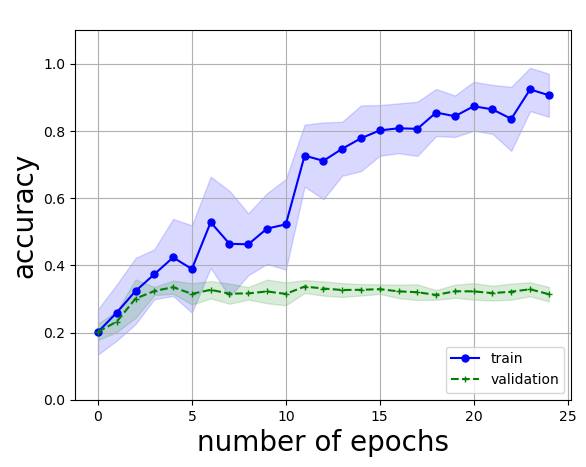} &
\includegraphics[width = 0.3 \linewidth]{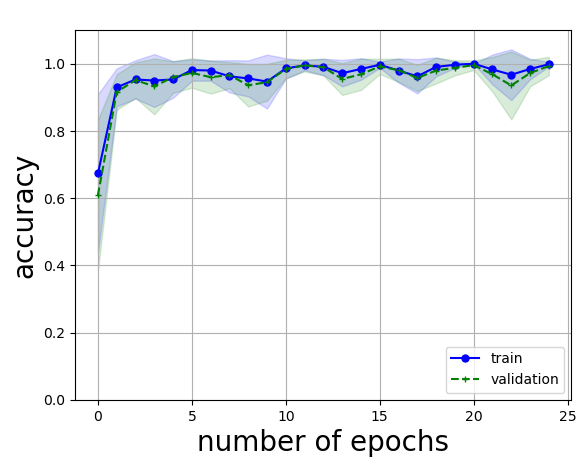} \\
\end{tabular}
\caption{Training curves for the detection of the number of holes.}
\label{fig_holes_training_curves}
\end{figure}

\subsection{Learning curves}
\label{app_holes_learning_curves}

Figure~\ref{fig_holes_learning_curves} shows the learning curves for every pipeline; i.e., the test accuracy of the trained pipelines depending on the total amount of training data. This serves to evaluate the data efficiency of the different methods. 
The PH-approaches perform well even for a small number of training point clouds. PointNet also has good performance, although it requires more training data. The other approaches (NN shallow, NN deep, and ML) have poor performance, which does not improve when more training data is available. An explanation for this is that these methods are based on distance matrices and hence cannot directly take advantage of the permutation symmetry of point clouds.

\begin{figure}[h]
\centering
\begin{tabular}{cccccc}
PH simple & PH & ML \\ 
\includegraphics[width = 0.3 \linewidth]{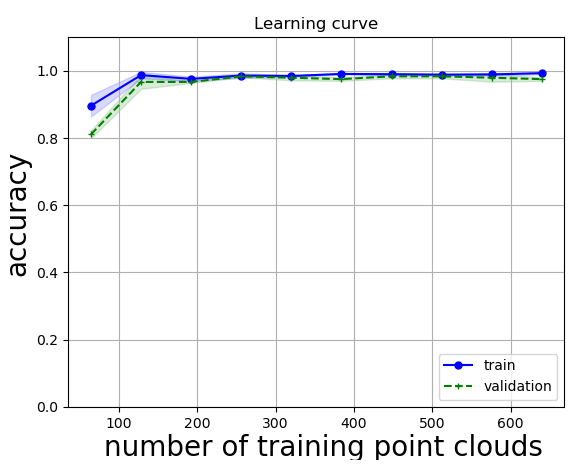} &
\includegraphics[width = 0.3 \linewidth]{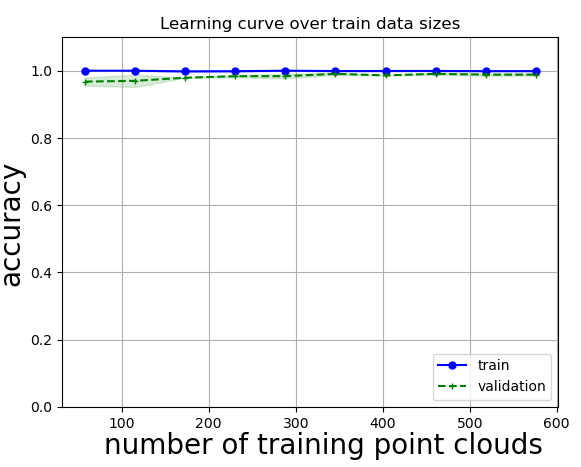} &
\includegraphics[width = 0.3 \linewidth]{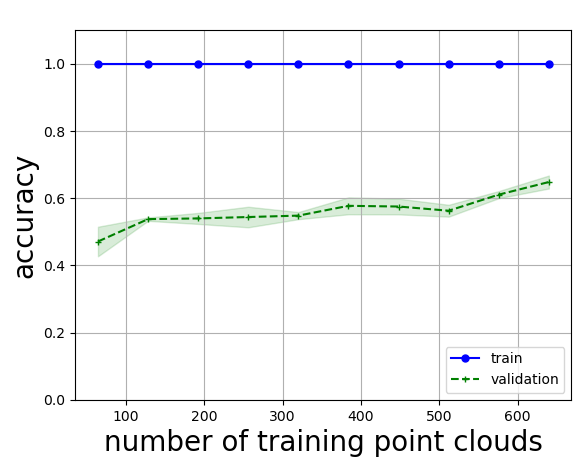} \\
\\
NN shallow & NN deep & PointNet  \\ 
\includegraphics[width = 0.3 \linewidth]{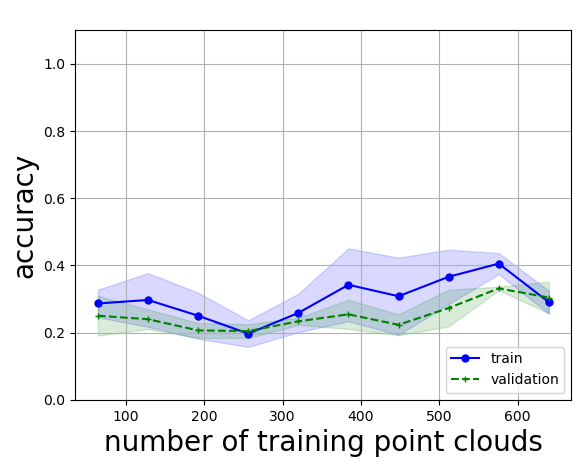} &
\includegraphics[width = 0.3 \linewidth]{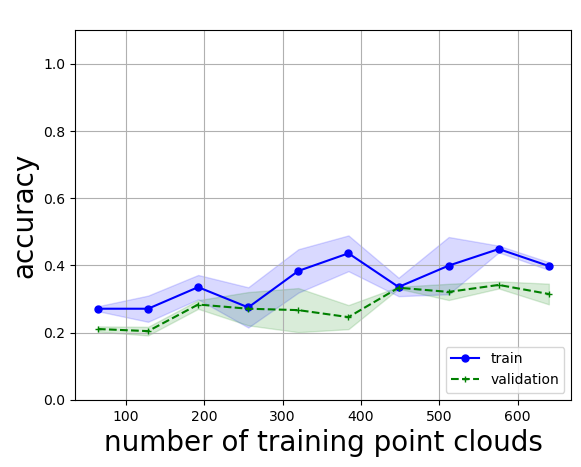} &
\includegraphics[width = 0.3 \linewidth]{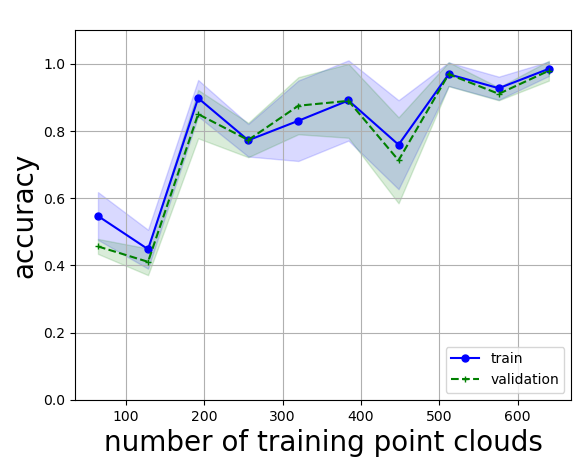} \\
\end{tabular}
\caption{Learning curves for the detection of the number of holes.}
\label{fig_holes_learning_curves}
\end{figure}

\subsection{Computational resources}
\label{app_holes_comp}

Figure~\ref{fig_holes_comp} visualizes the computational efficiency and memory usage. We see that PH pipeline also performs better with respect to these criteria in comparison to the other methods. The hyperparameter tuning of the PH pipeline does take time (as we consider a wide range of parameters for the different persistence signatures), but Figure~\ref{fig_holes_results} shows that even PH simple, where the SVM is used directly on the lifespans of the 10 most persisting cycles (without any tuning of PH-related parameters) performs well.

We note that the difference in the memory usage for data comes from the different types of input that are considered by different pipelines: PDs (lists of persistence intervals) for PH simple and PH, $100 \times 100$ distance matrices for ML and NNs, and $1000 \times 3$ point clouds for PointNet (Appendix~\ref{app_exp_pipelines}).

\begin{figure}[h]
\centering
\begin{tabular}{cc}
\includegraphics[height = 0.285 \linewidth]{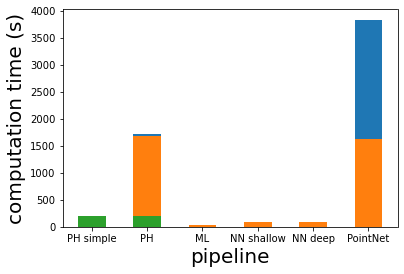} &
\includegraphics[height = 0.285 \linewidth]{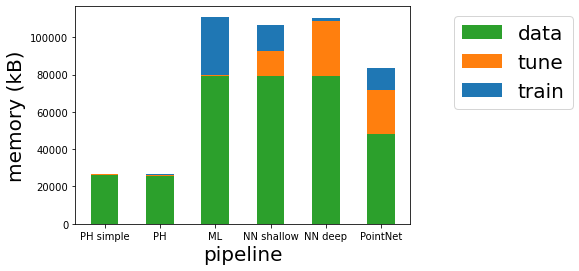} \\
\end{tabular}
\caption{Computational resources for the detection of the number of holes.}
\label{fig_holes_comp}
\end{figure}

\FloatBarrier 
\section{Additional experimental details for curvature}
\label{app_curvature}

\subsection{Pipeline}
\label{app_curvature_pipeline}

Before we visualize the PH pipeline, we give an illustrative figure that provides some intuition on why PH can detect curvature (Figure~\ref{fig_curvature_intuition}).

\begin{figure}[h]
\centering
\fbox{
\begin{tikzpicture}

\coordinate (A) at (0, 0);
\coordinate (B) at (2, 0);
\coordinate (C) at (1, 1.5);
\path [red, bend left] (A) edge (B);
\path [red, bend left] (B) edge (C);
\path [red, bend left] (C) edge (A);
\node [circle] at (A) {}; 
\node [circle] at (B) {};
\node [circle] at (C) {};

\coordinate (D) at (3, 0);
\coordinate (E) at (5, 0);
\coordinate (F) at (4, 1.5);
\path [blue] (D) edge (E);
\path [blue] (E) edge (F);
\path [blue] (F) edge (D);
\node [circle] at (D) {}; 
\node [circle] at (E) {};
\node [circle] at (F) {};

\coordinate (G) at (6, 0);
\coordinate (H) at (8, 0);
\coordinate (I) at (7, 1.5);
\path [green!50!black, bend right] (G) edge (H);
\path [green!50!black, bend right] (H) edge (I);
\path [green!50!black, bend right] (I) edge (G);
\node [circle] at (G) {}; 
\node [circle] at (H) {};
\node [circle] at (I) {};

\end{tikzpicture}
}
\caption{Intuition behind curvature detection with PH. For triangles embedded on a manifold with constant negative (left, in red), zero (middle, in blue) and positive (right, in green) curvature, the length of triangle edges clearly reflect the underlying curvature. Since persistence captures the length of these edges (when the triangle vertices merge into a  component), PH can be used to detect curvature.}
\label{fig_curvature_intuition} 
\end{figure}
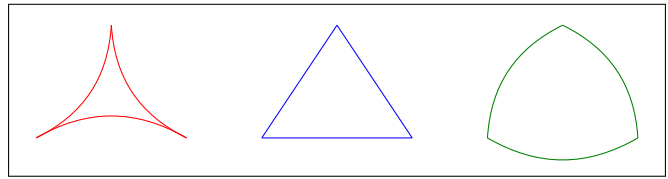

The PH pipeline to detect curvature (Section~\ref{section_curvature}) is visualized in Figure~\ref{fig_curvature_pipeline}. The example point cloud shown in the figure is in the Euclidean plane. We start by calculating the Euclidean distance matrix, and then construct the Vietoris-Rips filtration from these distances, which approximates the point cloud at different scales. 0-dimensional PD registers one persisting cycle reflecting the single component of the disk (which we ignore, since it is shared by every disk in the data and thus does not contribute to the classification), and many other  components which have a short lifespan as they get connected to other point-cloud points early in the filtration. There are no persistent 1-dimensional holes since disks are contractible, but there are many holes with short persistence. PDs are then transformed to a vector summary such as a PI.

\begin{figure}[h]
\centering
\fbox{
\begin{tikzpicture}[auto, >=latex'] 
\tikzstyle{block}=[draw, shape=rectangle, rounded corners=0.5em, fill = blue!15!white];
\node [block] (pc) {point cloud};
\node [block, right = 1.5em of pc] (distance) {Euclidean, spheric or hyperbolic distance matrix};
\node [block, right = 1.5em of distance] (filtration) {Vietoris-Rips filtration};
\node [block, below = of filtration.east, anchor = east] (pd) {0- and 1-dim PD};
\node [block, below = 1.25em of distance] (ph) {0- and 1-dim lifespans, PI or PL};
\node [block, below = of pc.west, anchor = west] (ml) {SVM};
\node [above = 0.25em of pc] {\includegraphics[height = 0.125\linewidth]{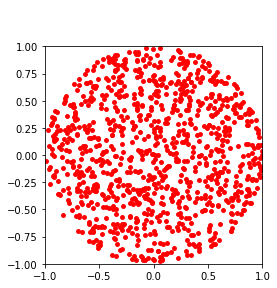}};
\node [above left = 0.25em and -7.5em of distance] {\includegraphics[height = 0.125\linewidth]{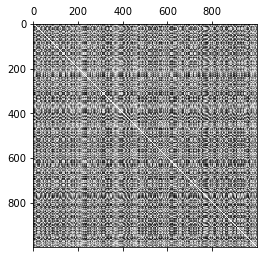}};
\node [above left = 0.25em and -10.25em of filtration] {\includegraphics[height = 0.125\linewidth]{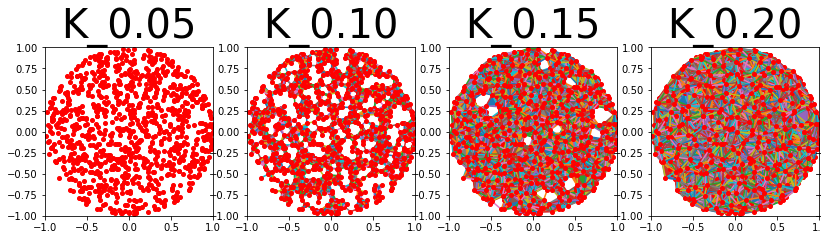}};
\node [below = -0.25em of pd] {\includegraphics[height = 0.125\linewidth]{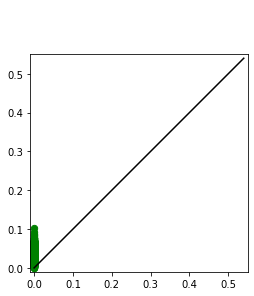} \includegraphics[height = 0.125\linewidth]{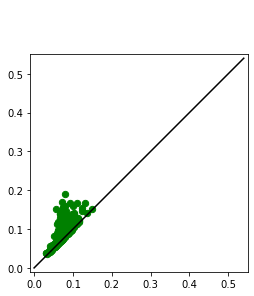}};
\node [below = 0em of ph] {\includegraphics[height = 0.115\linewidth]{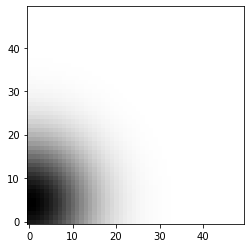} \includegraphics[height = 0.115\linewidth]{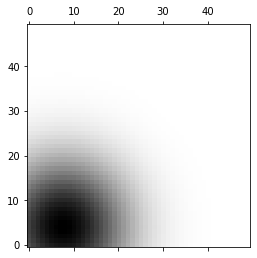}};
\path[draw,->] (pc) edge (distance)
(distance) edge (filtration)
(filtration) edge ([xshift=-1.1em, yshift=0.7em]pd)
(pd) edge (ph) 
(ph) edge (ml);
\end{tikzpicture}
}
\caption{Persistent homology pipeline to detect curvature.}
\label{fig_curvature_pipeline} 
\end{figure}
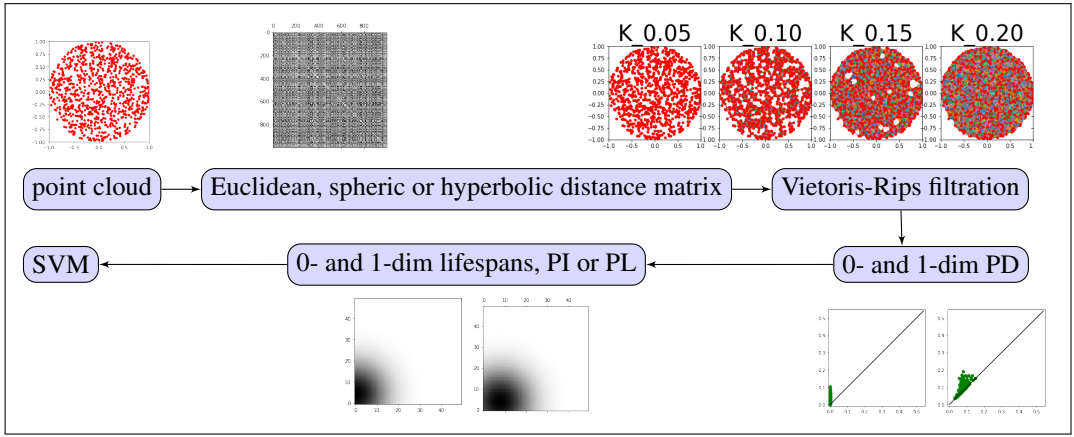

\subsection{Performance across multiple runs}
\label{app_curvature_performance}

Table~\ref{tab_curvature_accs} shows that (0-dimensional) PH outperforms the other machine- and deep-learning approaches for curvature detection, across multiple experimental runs. The poor performance of PointNet is not surprising, as it takes raw point clouds as input, i.e., their coordinates and Euclidean distances, and is therefore the only pipeline that does not exploit the spherical and hyperbolic distances for point clouds lying on manifolds with positive or negative curvature (for details, see Appendix~\ref{app_exp_pipelines}).

\begin{table}
\caption{Mean squared error across multiple runs for curvature detection.}
\label{tab_curvature_accs}
\centering
\begin{tabular}{lrrrrrrrrrr} 
\toprule

\textbf{Run} & \rotatebox{90}{\textbf{0-dim PH simple}} & \rotatebox{90}{\textbf{0-dim PH simple 10}} & \rotatebox{90}{\textbf{0-dim PH}} & \rotatebox{90}{\textbf{1-dim PH simple}} & \rotatebox{90}{\textbf{1-dim PH simple 10}} & \rotatebox{90}{\textbf{1-dim PH}} & \rotatebox{90}{\textbf{ML}} & \rotatebox{90}{\textbf{NN shallow}} & \rotatebox{90}{\textbf{NN deep}} & \rotatebox{90}{\textbf{PointNet}} \\
\midrule

run 1 & \textbf{0.06} & 0.21 & 0.08 & 0.34 & 0.29 & 0.18 & 0.34 & 0.42 & 0.46 & 12.28 \\
run 2 & \textbf{0.06} & 0.21 & 0.08 & 0.34 & 0.29 & 0.18 & 0.34 & 0.43 & 0.46 & 0.25 \\
run 3 & \textbf{0.06} & 0.21 & 0.08 & 0.34 & 0.29 & 0.18 & 0.34 & 0.66 & 0.43 & 578.28 \\
\cline{2-11}
mean & \textbf{0.06} & 0.21 & 0.08 & 0.34 & 0.29 & 0.18 & 0.34 & 0.50 & 0.45 & 196.94 \\
std dev & 0.00 & 0.00 & 0.00 & 0.00 & 0.00 & 0.00 & 0.00 & 0.14 & 0.02 & 330.31 \\

\bottomrule
\end{tabular}
\end{table}

\subsection{Computational resources}
\label{app_curvature_comp}

Figure~\ref{fig_curvature_comp} visualizes the computational time and memory usage of the different pipelines for this task. The superior performance of the PH pipelines in comparison to other methods (Figure~\ref{fig_curvature_results}) can come at a high cost with respect to the usage of computational resources. However, the simple 0-dim PH pipeline (that only focuses on the lifespans of the PH cycles), which achieves the best predictive power (Figure~\ref{fig_curvature_results}), is the most efficient.

\begin{figure}[h]
\centering
\begin{tabular}{cc}
\includegraphics[height = 0.35 \linewidth]{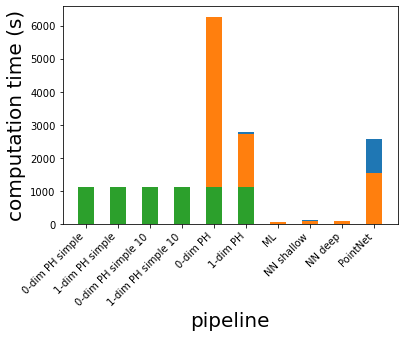} &
\includegraphics[height = 0.35 \linewidth]{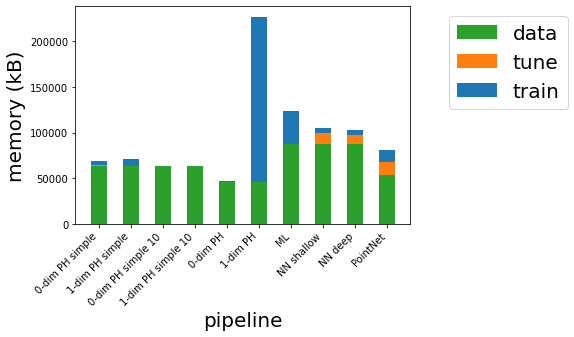} \\
\end{tabular}
\caption{Computational resources for the detection of curvature.}
\label{fig_curvature_comp}
\end{figure}

\FloatBarrier 
\section{Additional experimental details for convexity}
\label{app_convexity}

\subsection{Pipeline}
\label{app_convexity_pipeline}

A visual summary of the PH pipeline used for convexity detection (Section~\ref{section_convexity}) is given in Figure~\ref{fig_convexity_pipeline}. Every shape has at least one 0-dimensional cycle, i.e., connected component. For the given example point cloud, PD on the cubical complexes weighted by the tubular filtration from the line passing through  the bottom of the image will have a second persistent connected component. A positive persistence of the second most persisting cycle for at least some line indicates concavity.

\begin{figure}[h]
\centering
\fbox{
\begin{tikzpicture}[auto, >=latex']
\tikzstyle{block}=[draw, shape=rectangle, rounded corners=0.5em, fill = blue!15!white];
\node [block] (pc) {point cloud};
\node [block, right = 4em of pc] (image) {image};
\node [block, right = 13em of image] (filtration) {tubular filtration, for 9 lines};
\node [block, below = 4.35em of filtration.east, anchor = east] (pd) {0-dim PD, for 9 lines};
\node [block, text width = 5.5cm, above left = -1.9em and 2.75em of pd] (ph) {maximum lifespan of the 2nd most persisting cycle, across 9 lines};
\node [block, below = 4em of pc.west, anchor = west] (ml) {SVM};
\node [above = 0.25em of pc] {\includegraphics[height = 0.125\linewidth]{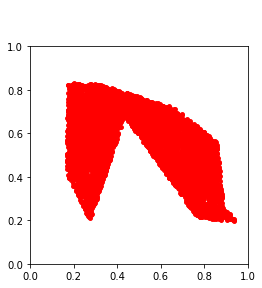}};
\node [above = 0.25em of image] {\includegraphics[height= 0.1\linewidth]{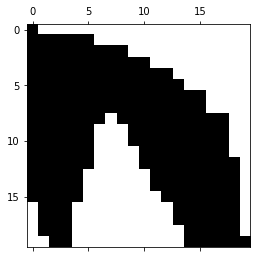}};
\node [text width = 5cm, above left = 2.25em and -7em of filtration] {\includegraphics[height = 0.2\linewidth]{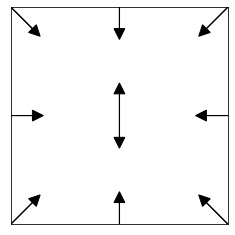}};

\node [text width = 5cm, above right = 6em and -16em of filtration] {\includegraphics[height = 0.25\linewidth]{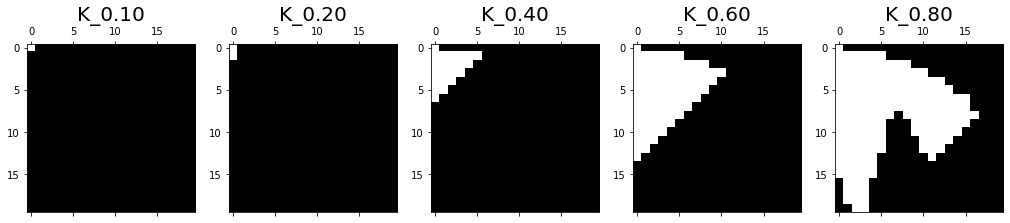} };
\node [text width = 5cm, above right = 4.5em and -16em of filtration] {\qquad \qquad \qquad \qquad \vdots};
\node [text width = 5cm, above right = 0.25em and -16em of filtration] {\includegraphics[height = 0.25\linewidth]{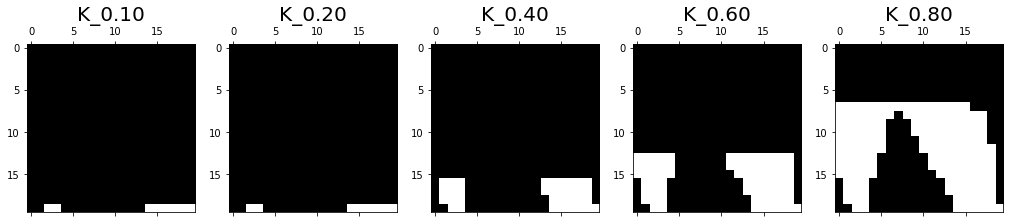}};

\node [below left = -0.25em and -3em of pd] {\includegraphics[height = 0.125\linewidth]{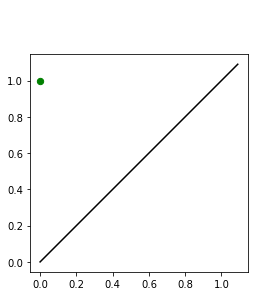}};
\node [below left = 2.5em and -5em of pd] {\dots};
\node [below left = -0.25em and -10em of pd] {\includegraphics[height = 0.125\linewidth]{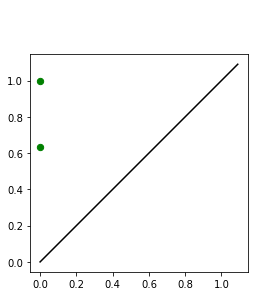}};

\node [below left = 0.5em and -17em of ph] {$0.63 = \max \{ 0.63, 0.3, 0.21, \dots, 0.00, 0.00 \}$};
\path[draw,->] (pc) edge (image)
(image) edge (filtration)
(filtration) edge ([xshift=-1.3em, yshift=0.9em]pd)
(pd) edge (ph) 
(ph) edge (ml);
\end{tikzpicture}
}
\caption{Persistent homology pipeline to detect convexity.}
\label{fig_convexity_pipeline}
\end{figure}
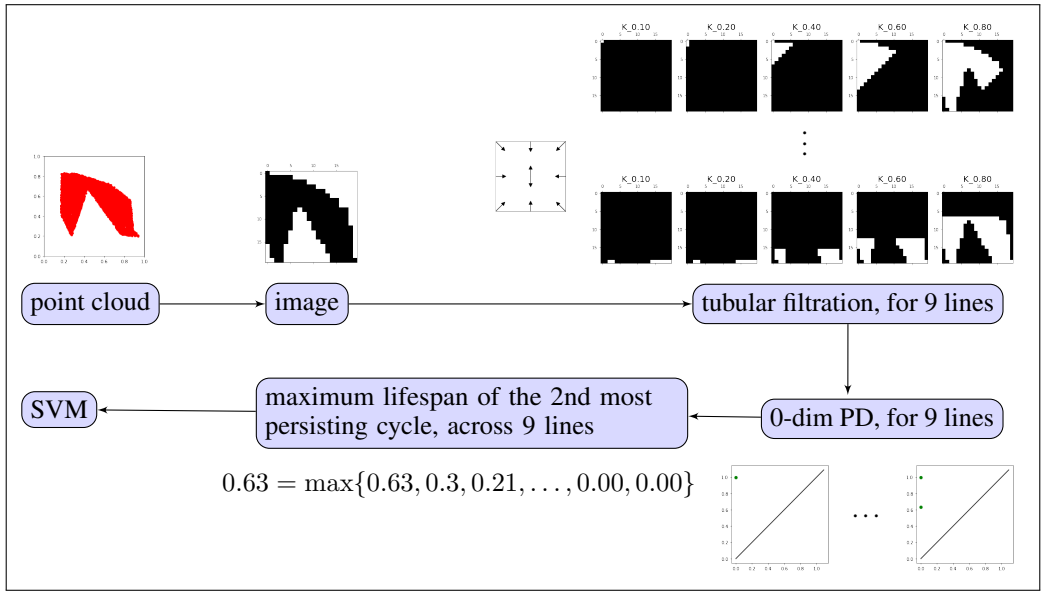

We note here that the convexity could also be detected with PH with respect to the Vietoris-Rips filtration, with some important adjustments. Indeed, \cite[Theorem 2]{chazal2011scalar} provides a guarantee that PH of any function $f$ and shape $M$ can be estimated using an algebraic construction based on Rips complexes from a point cloud $X$ which is a geodesic dense-enough sample of $M$ (and Theorems 3 and 4 in this paper obtain guarantees in scenarios where both function values and pairwise distances are approximate, i.e., defined on the point cloud). To do so to detect convexity, we cannot employ the standard (so-called vanilla) Vietoris-Rips simplicial complex that relies on the distance function, since all point cloud points show immediately at $b=0$ in the filtration (that all soon get connected into a single component), so that it never sees the two connected components in concave shapes, at any scale $r\in \mathbb{R},$ which are captured with cubical complexes. Filtering the point cloud points by their height (yielding a so called weighted Rips filtration) might capture the multiple connected components, but these components can get connected with an edge as soon as they are born, if they are close to each other with respect to Euclidean distance (Figure~\ref{fig_convexity_simplicial_cubical}). This can be resolved by considering the geodesic distance (the length of the shortest path along the manifold, or a graph), which will allow the multiple connected components to persist longer in the filtration. Indeed, the geodesic distance between points in the ``disconnected" regions in concave shapes (the clusters) is larger than the Euclidean distance, so that these only get connected later in the filtration.

\begin{figure}[h]
\centering
\includegraphics[width = \linewidth]{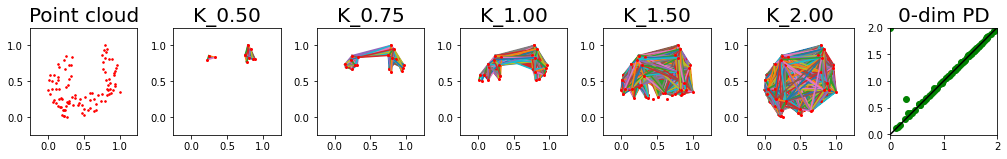} \\
\includegraphics[width = \linewidth]{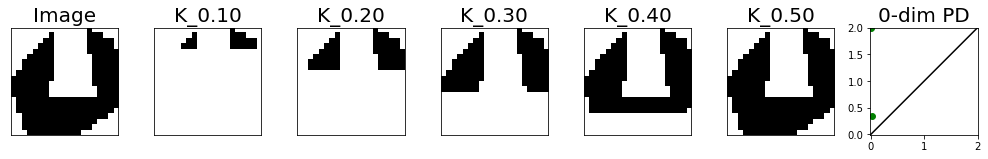} 
\caption{Convexity detection with PH on simplicial and cubical complexes. The concavity can be detected with the weighted Vietoris-Rips simplicial complex, with the tubular filtration function on the vertices (top row). In this figure, we consider the tubular function with respect to the horizontal line at the top of the point cloud. The filtration function on edges is defined according to the Euclidean distances, but in a way that ensures that an edge can only appear in the filtration after both vertices incident to this edge appear in the filtration (for details, see \cite{anai2020dtm}). However, these multiple connected components can still be connected with an edge, if they are close in the Euclidean space. This could be circumvented by considering the weighted Vietoris-Rips which relies on the geodesic distances (which are expensive to compute), or by considering cubical complexes instead (bottom row), where the connected components remain separate until they merge with the rest of the shape.}
\label{fig_convexity_simplicial_cubical}
\end{figure}

The important thing to keep in mind is to choose a filtration that will see disconnected components for concave shapes. We choose cubical complexes as they are more straightforward and do not involve the calculation of geodesic distances. Indeed, as the authors of \cite{chazal2011scalar} note, geodesic distances are not known in advance and have to be estimated through some neighborhood graph distance, and computing full pairwise geodesic distances is expensive \cite{mitchell1987discrete, kimmel2019processing} (e.g., there are deep learning efforts to estimate these geodesic distances on point clouds, such as \cite{he2019geonet, potamias2022graphwalks}).

Finally, we also note that the calculation of geodesic distances for curvature detection (Section~\ref{section_curvature}) was straightforward since the point clouds were sampled from unit disks from manifolds with constant curvature, which enabled us to directly rely on the analytical formulas for geodesic distance.

\subsection{Performance across multiple runs}
\label{app_convexity_performance}

For any experimental run, PH is better able to distinguish between convex and concave shapes than the other machine- and deep-learning pipelines (Table~\ref{tab_convexity_accs}).

\begin{table}
\caption{Accuracy across multiple runs for convexity detection.}
\label{tab_convexity_accs}
\centering
\begin{tabular}{llrrrrr} 
\toprule

\textbf{Experimental setting} & \textbf{Run} & \textbf{PH} & \textbf{ML} & \textbf{NN shallow} & \textbf{NN deep} & \textbf{PointNet} \\
\midrule

\multirow{5}{2cm}{train = regular test = regular}
& run 1 & \textbf{1.00} & 0.74 & 0.72 & 0.72 & \textbf{1.00} \\
& run 2 & \textbf{1.00} & 0.74 & 0.77 & 0.60 & \textbf{1.00} \\
& run 3 & \textbf{1.00} & 0.74 & 0.75 & 0.75 & \textbf{1.00} \\
\cline{2-7}
& mean & \textbf{1.00} & 0.74 & 0.75 & 0.69 & \textbf{1.00} \\
& std dev & 0.00 & 0.00 & 0.02 & 0.08 & 0.00 \\
\midrule

\multirow{5}{2.2cm}{train = random test = random}
& run 1 & \textbf{0.85} & 0.56 & 0.60 & 0.46 & 0.40 \\
& run 2 & \textbf{0.85} & 0.56 & 0.55 & 0.52 & 0.61 \\
& run 3 & \textbf{0.85} & 0.56 & 0.54 & 0.59 & 0.71 \\
\cline{2-7}
& mean & \textbf{0.85} & 0.56 & 0.56 & 0.52 & 0.57 \\
& std dev & 0.00 & 0.00 & 0.03 & 0.06 & 0.16 \\
\midrule

\multirow{5}{2cm}{train = regular test = random}
& run 1 & \textbf{0.78} & 0.59 & 0.59 & 0.59 & 0.51 \\
& run 2 & \textbf{0.78} & 0.59 & 0.56 & 0.60 & 0.47 \\
& run 3 & \textbf{0.78} & 0.59 & 0.57 & 0.57 & 0.57 \\
\cline{2-7}
& mean & \textbf{0.78} & 0.59 & 0.57 & 0.59 & 0.52 \\
& std dev & 0.00 & 0.00 & 0.02 & 0.08 & 0.05 \\
\midrule

\multirow{5}{2.2cm}{train = random test = regular}
& run 1 & \textbf{0.96} & 0.54 & 0.55 & 0.49 & 0.54 \\
& run 2 & \textbf{0.96} & 0.54 & 0.54 & 0.42 & 0.57 \\
& run 3 & \textbf{0.96} & 0.54 & 0.56 & 0.52 & 0.52 \\
\cline{2-7}
& mean & \textbf{0.96} & 0.54 & 0.55 & 0.48 & 0.54 \\
& std dev & 0.00 & 0.00 & 0.01 & 0.05 & 0.02 \\

\bottomrule
\end{tabular}
\end{table}

\subsection{Computational resources}
\label{app_convexity_comp}

Results related to the computational efficiency of the different approaches (trained on regular, and tested on regular shapes) are summarized in Figure~\ref{fig_convexity_comp}. In this case, the PH pipeline significantly outperforms the other methods, since it relies on a scalar summary of a point cloud (the maximum lifespan of the second most persisting connected component, across 9 tubular filtration function lines, see Section~\ref{section_convexity} and Figure~\ref{fig_convexity_pipeline}). On the other hand, PointNet relies on raw point clouds and therefore has a very high memory consumption, since point clouds have $5\,000$ points for this task (compared to $1\,000$ points for the detection of the number of holes, or $500$ points for curvature detection).

\begin{figure}[h]
\centering
\begin{tabular}{cc}
\includegraphics[height = 0.235 \linewidth]{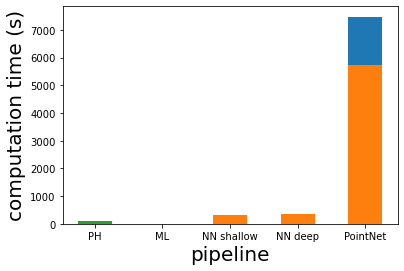} &
\includegraphics[height = 0.235 \linewidth]{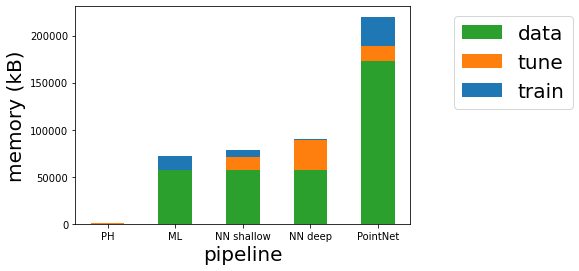} \\
\end{tabular}
\caption{Computational resources for the detection of convexity.}
\label{fig_convexity_comp}
\end{figure}

\subsection{Mislabeled point clouds}
\label{app_convexity_wrong}

In order to gain a better understanding of the performance and limitations of our PH pipeline, we look at some examples of mislabeled point clouds.  Figure~\ref{fig_convexity_wrong} shows a few point clouds sampled from concave shapes that are erroneously classified as convex by PH pipeline (trained on regular, and tested on random shapes). The figure also clearly suggests that considering additional lines for the tubular filtration function would resolve these issues.

\begin{figure}[h]
\centering
\begin{tabular}{ccccc}
\includegraphics[height = 0.16 \linewidth]{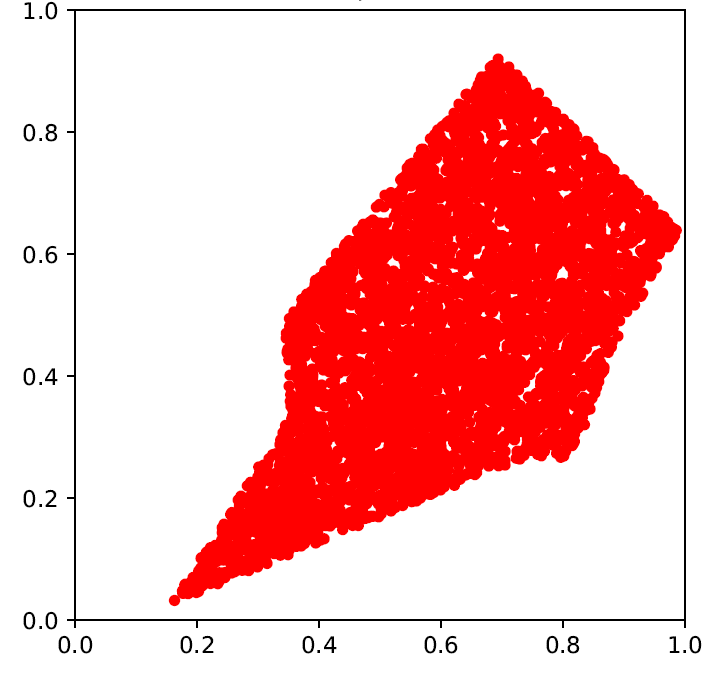} &
\includegraphics[height = 0.16 \linewidth]{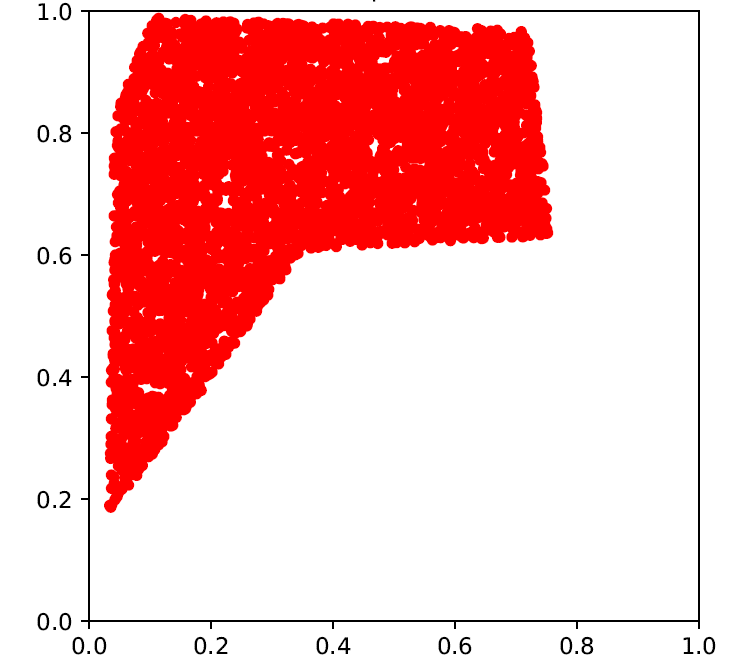} &
\includegraphics[height = 0.16 \linewidth]{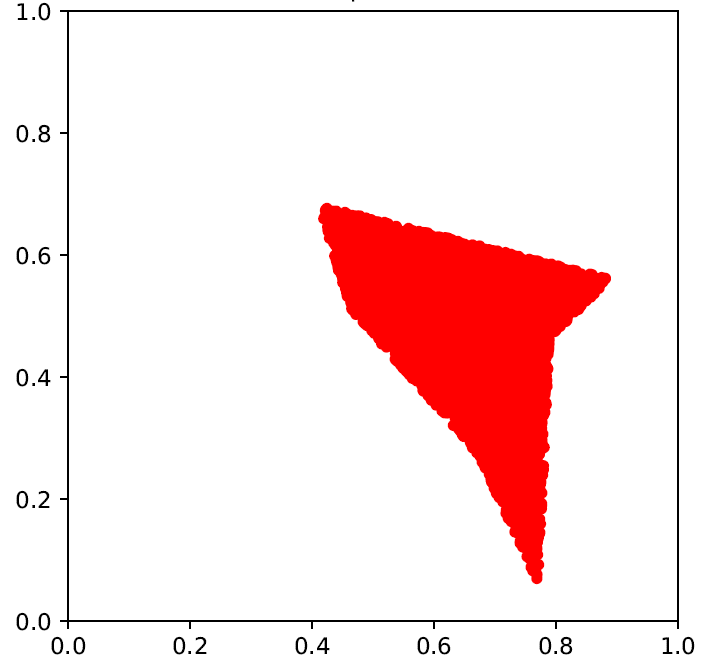} &
\includegraphics[height = 0.16 \linewidth]{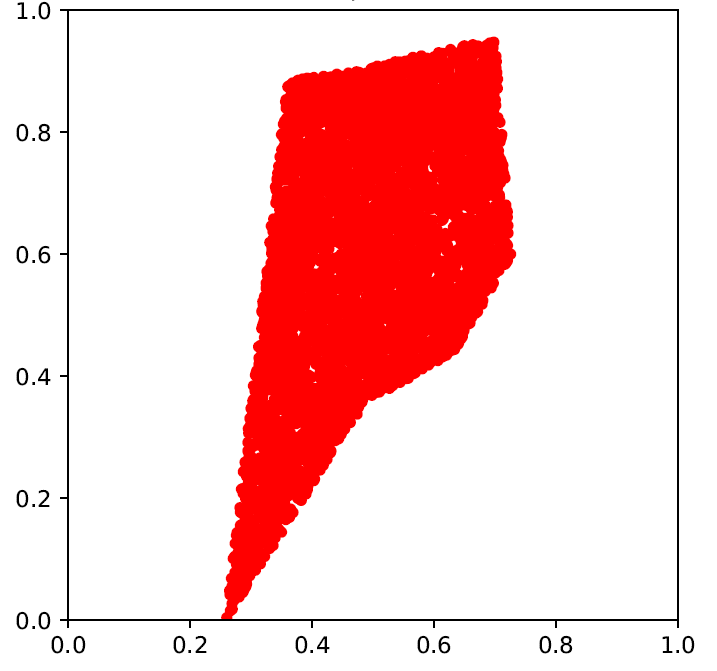} &
\includegraphics[height = 0.16 \linewidth]{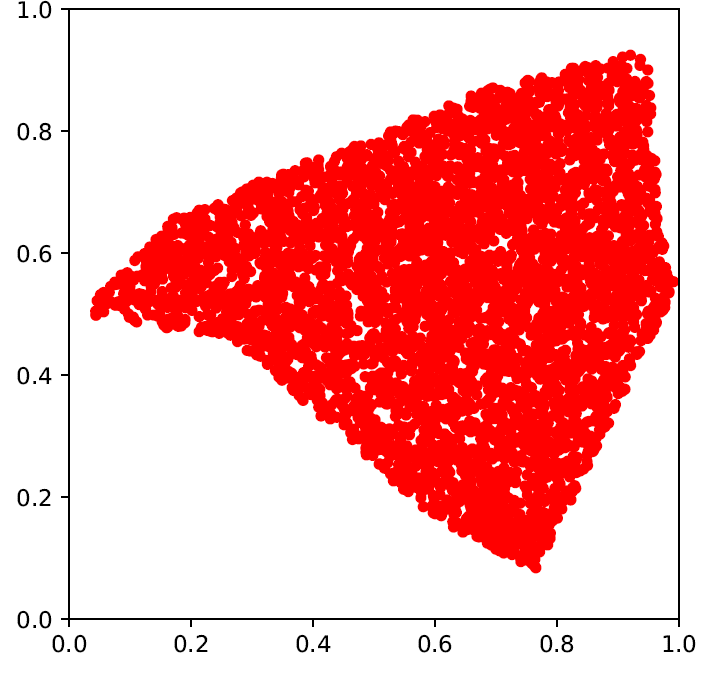} \\
\end{tabular}
\caption{Examples of point clouds from concave shapes incorrectly classified as convex.}
\label{fig_convexity_wrong}
\end{figure}

\FloatBarrier 
\section{Guidelines for persistent homology in applications}
\label{app_guidelines}

Our results demonstrate that PH can be successful in applications for which detecting the number of holes, curvature and convexity is important.  Based on our findings, we delineate guidelines for the choice of filtrations and signatures, the input and output of PH pipelines, and draw a better understanding of the topology and geometry properties that are captured by long and short persistence intervals (see Figure~\ref{fig_guidelines}).

We again note here that we use the alpha simplicial complex for the detection of number of holes in order to improve the computational efficiency, but that the same can be done with the standard Vietoris-Rips filtration. In addition, we discuss in Appendix~\ref{app_convexity_pipeline} that convexity can alternatively be detected with the weighted Vietoris-Rips filtration, filtered by the tubular function, and relying on geodesic distances.

\begin{figure}[h]
\centering

\fbox{
\begin{tikzpicture}[auto, >=latex', node distance=.75cm]
\tikzstyle{block}=[draw, shape = rectangle, rounded corners = 0.5em, fill = blue!15!white, text width=9em, align = center, minimum height=1.5em];
\tikzstyle{block vertical}=[draw, shape=rectangle, fill = white, text width = 2em, align = center];

\node [block] (holes) {(Betti $\beta_k$) number of $k$-dimensional cycles};
\node [block, right = of holes] (curvature) {curvature};
\node [block, right = of curvature] (convexity) {convexity};

\node [block, below = 4em of holes] (filtration-holes) {alpha simplicial complex, filtered by Euclidean distance};
\node [block, below = 4.5em of curvature] (filtration-curvature) {Rips simplicial complex, filtered by geodesic distance};
\node [block, below = 4.5em of convexity] (filtration-convexity) {cubical complex, filtered by tubular function};

\node [block, below = 11em of holes] (ph-holes) {longest $k$-dimensional cycles};
\node [block, right = of ph-holes] (ph-curvature) {many short $0$- and $1$-dimensional cycles};
\node [block, right= of ph-curvature] (ph-convexity) {second longest $0$-dimensional cycle};

\path[draw,->] (holes) edge (filtration-holes)
(filtration-holes) edge (ph-holes)
(curvature) edge (filtration-curvature)
(filtration-curvature) edge (ph-curvature)
(convexity) edge (filtration-convexity)
(filtration-convexity) edge (ph-convexity);

\node[block vertical, left = of holes] (signal) {\rotatebox{90}{SIGNAL}};
\node[block vertical, below = 1em of signal] (filtration) {\rotatebox{90}{FILTRATION}};
\node[block vertical, below = 1em of filtration] (ph) {\rotatebox{90}{SIGNATURE}};

\end{tikzpicture}
}

\caption{Persistent homology can be useful in applications where $k$-dimensional cycles, curvature or convexity are important features. The choice of filtration and persistence signature, including the focus on the long and/or short persistence intervals, depends on the signal of the particular application.}
\label{fig_guidelines}
\end{figure}
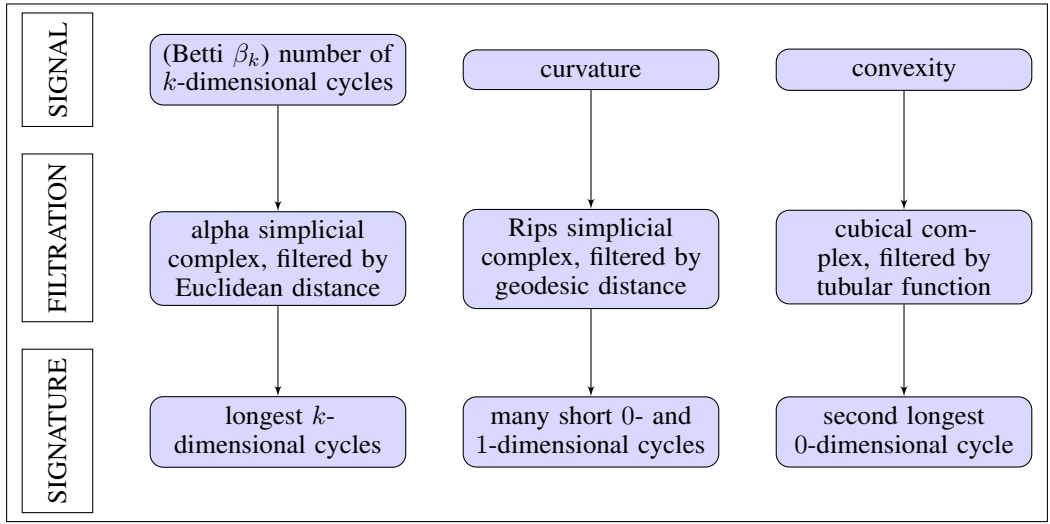

\subsection{Adjustments of PH pipeline for related applications}
\label{app_guidelines_adjustments}

Some obvious adjustments to the guidelines from Figure~\ref{fig_guidelines} can be made for applications related to the ones that we consider here. Some possible adjustments include the following.

\begin{itemize}[leftmargin=*]
\item If it is expected that the data set is noisy, the suggested filtration function should be weighted by density to achieve robustness to noise (as in Section~\ref{section_holes}).

\item In our experiments, we focused our attention on the PH information relevant to the individual problem at hand, but for other related applications, one might need to consider a different type of information given by PH. For example, if we do not only aim to distinguish between convex and concave shapes, but rather to capture more information about the possibly many concavities, we should not restrict our attention only to the second most persisting cycle, nor consider the maximum across filtration function directions. Instead, it would be useful to take all PH intervals into account for such an application.

\item If there are multiple sources of differences in the data, it can be a good idea to combine the different pipelines. For example, if two classes can be differentiated with some concavities, $0$-dimensional PH on the tubular filtration will be useful, but if it is also the shape curvature that can help make a distinction, this information can be concatenated with $0$- and $1$-dimensional PH on Vietoris-Rips filtration. 
\end{itemize}

\subsection{Discussion of PH pipeline for other applications} 
\label{app_guidelines_discussion}

\paragraph{Step 1: Signal} Figure~\ref{fig_guidelines} and the discussion above clearly indicate that, when faced with a new problem, it is essential to first try to identify the important information, the signal. To illustrate this more clearly, we list some examples of very different types of signal in Appendix~\ref{app_guidelines_ex_signal_top}, Appendix~\ref{app_guidelines_ex_signal_geo}, Appendix~\ref{app_guidelines_ex_signal_top_geo}. Once there is some understanding of the signal, the next steps are to choose the filtration and signature accordingly.

\paragraph{Step 2: Filtration} The aim is for the filtration to capture the signal. For instance, the Vietoris-Rips filtration encodes the size of cycles, while the height or tubular filtration encodes their position. The choice of filtration also influences which type of geometric properties will be captured by long or short persistence intervals. To illustrate the importance of the choice of filtration for the interpretation of long and short intervals, we consider the example point cloud in Figure~\ref{fig_ph_pipeline}. PH with respect to any meaningful filtration can detect the topology of the underlying shape, i.e., the two holes. However, for the height filtration function from the top of the image, the small circle would have a longer lifespan of the two (as it is born earlier in the filtration), and the large circle can have a seemingly very short lifespan (as it is only born at the bottom of the image). For the Vietoris-Rips filtration it is the opposite (a small cycle has short persistence), and PD on the height filtration from the bottom of the image would see cycles of comparable persistence.

\paragraph{Step 3: Signature} The choice of persistence signature and the corresponding metric further influences the emphasis on long or short persistence intervals. The  Wasserstein distances \cite{carlsson2014topological} between PDs place more importance on long persistence, and the same is true for $L_p$ or $l_p$ distances between other common choices of persistence signatures, with the standard choice of parameters. However, similarly to our discussion above in \ref{app_guidelines_adjustments}, one might want to focus only on short intervals, e.g., by considering only the intervals with lifespan below a certain threshold (so that the distances would be computed between this simplified PH information). Some alternative ways to prioritize short intervals, or intervals of any persistence, or even birth or death value, is via persistence images with the appropriate choice of the weighing function \cite{adams2017persistence}, or stable ranks with appropriate densities \cite{chacholski2020metrics}. Note, however, that the stability results also depend on the choice of filtration, persistence signature and metric \cite{turkevs2021noise}. Finally, we note that there has recently been a lot of effort in trying to train neural networks to learn what the best PH signature is for specific types of applications \cite{royer2019atol, carriere2020perslay, montufar2020can, de2022ripsnet}.

In the remainder of this section, we consider a few hypothetical applications to discuss the relevance of signal, filtration and signature that we hope will be useful for practitioners. In particular, the examples highlight that the importance of long and short persistence intervals depend on the particular application domain. In this context, it is sensible to try to understand the nature of information that is captured with PH (e.g., topological or geometric, any of which might or not be important). The examples thus help us to refine an ongoing discussion in the field about the information detected by intervals of a specific length in a PD:

\begin{itemize}
\item {\bf Long persistence intervals as signal.} Indeed, this is true in examples from Figure~\ref{fig_guidelines_ex1} (when long intervals capture important topology) or Figure~\ref{fig_guidelines_ex4} (where long intervals capture important geometry). However, an example in Figure~\ref{fig_guidelines_ex5} (together with results in Section~\ref{section_curvature} and Section~\ref{section_convexity}) highlights that important information can be encoded in short intervals. 

\item {\bf Short persistence intervals as noise.} Figure~\ref{fig_guidelines_ex5} is an example where important information is captured by short intervals (or in this case, a short interval). This can also be seen in the experimental results in Section~\ref{section_curvature} and Section~\ref{section_convexity}.

\item  {\bf Long intervals capture topology.} An example in  Figure~\ref{fig_guidelines_ex4} highlights that long intervals, next to topology, also capture geometric information.

\item {\bf Many short intervals capture geometry.} An example in Figure~\ref{fig_guidelines_ex5} (together with results in Section~\ref{section_convexity}) shows that even a single short interval can capture (important) geometry. 
\end{itemize}

\subsubsection{Topology is important, geometry is irrelevant}
\label{app_guidelines_ex_signal_top}

In some applications, it might be useful to make no distinctions between a circle, a circle with a dent, the circle under translation or scaling, or a square (Figure~\ref{fig_guidelines_ex1}). In this case, ``shape'' is understood through the lens of topology --- more precisely, what we are interested in is what is called  ``homology-type'' ---, where one object can be deformed into another by bending, shrinking and expanding, but not tearing or gluing. Indeed, to a topologist, a coffee mug and a donut have the same shape.

\begin{figure}[h]
\centering
\includegraphics[width = \linewidth]{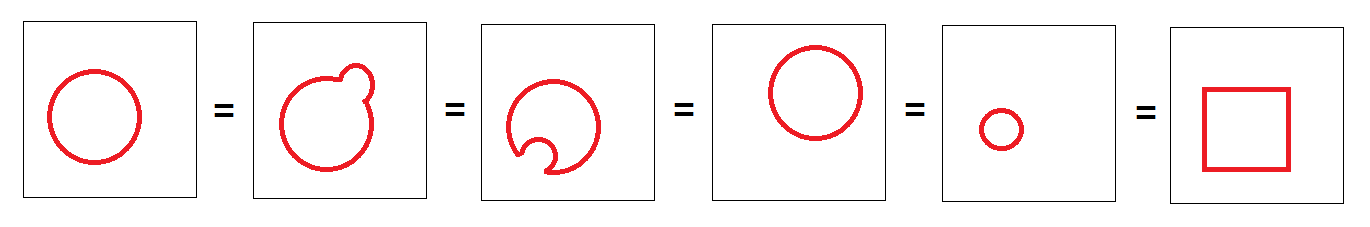}
\caption{An example of an application where topology is the signal. We consider all the shapes to be the same, i.e., to represent the same class of data, as they all have one connected component and one hole.}
\label{fig_guidelines_ex1}
\end{figure}

It is possible to obtain the same PH summaries for all of the shapes from  Figure~\ref{fig_guidelines_ex1}. Indeed, $1$-dimensional PDs with respect to the standard Vietoris-Rips filtration on a unit circle and a unit square sampled with same density (reflected in the birth values) could respectively be $\{(0.1, 1)\}$ and $\{(0.1, 1.41)\},$ since the death value reflects the size of the hole. However, we can focus on the cardinality $|\text{PD}|$ of PDs, that here only encodes topological information. Alternatively, we could rather consider PDs calculated on cubical complexes filtered by the binary or grayscale filtration.

Let us further assume that point clouds with multiple holes might be present in the data, but that the only relevant information is the presence of holes, and not their number (Figure~\ref{fig_guidelines_ex2}). An example of such application could be classification between chaotic and periodic (biological) time series, since the circularity of the so-called Taken's embedding point cloud reflect periodicity of the underlying time series \cite{perea2015sw1pers}. In this case, we can only focus on the maximum persistence $\max \{ l = d-b \mid (b, d) \in \text{PD} \}.$ Although PD captures a lot of topology and geometry, this choice of summary obviously ignores a lot of this information.

\begin{figure}[h]
\centering
\includegraphics[width = \linewidth]{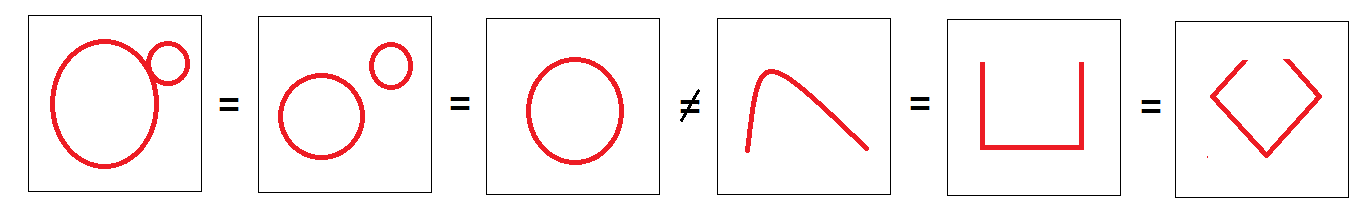}
\caption{An example of an application where the presence of holes is the signal. The first three shapes in the left part of the figure belong to the same class as there is at least one hole present, whereas the remaining three shapes belong to another class with no holes.}
\label{fig_guidelines_ex2}
\end{figure}

\subsubsection{Topology is irrelevant, geometry is important}
\label{app_guidelines_ex_signal_geo}

For other type of applications, the shapes from  Figure~\ref{fig_guidelines_ex1} might be representatives of different classes of objects (Figure~\ref{fig_guidelines_ex3}). Since they all have a single connected component and a single hole, the topological information has no use in discriminating between the classes. However, the geometric information about their size and position is useful.

\begin{figure}[h]
\centering
\includegraphics[width = \linewidth]{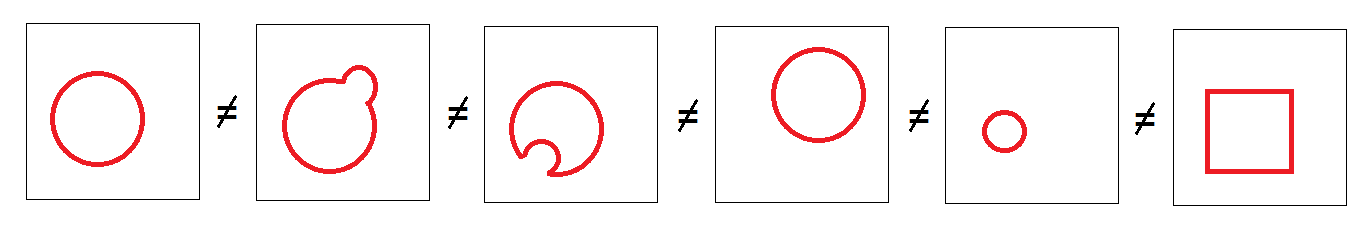}
\caption{An example of an application where geometry is the signal. In this case, every shape in the figure represents a different object, i.e., they all belong to different data classes.}
\label{fig_guidelines_ex3}
\end{figure}

\paragraph{Important geometry encoded in long intervals} Consider an example where every shape in the data set has only two holes (Figure~\ref{fig_guidelines_ex4}), and PH with respect to the Vietoris-Rips filtration. The two longest intervals reflect these two holes (topological information), but their lifespans reflect their size, and it is this geometric signal that can help discriminate between the shapes.

\begin{figure}[h]
\centering
\includegraphics[width = 0.45\linewidth]{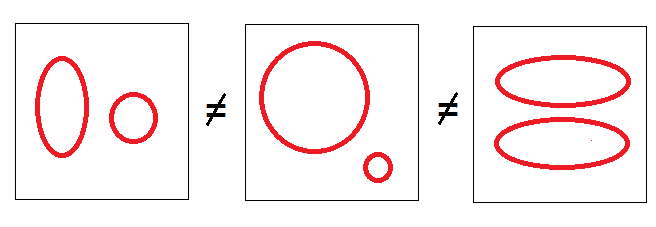}
\caption{An example of an application where the size of the holes is the signal. The three shapes all have two holes, but their size is meaningful for this application, so that they all belong to different data classes.}
\label{fig_guidelines_ex4}
\end{figure}

\paragraph{Important geometry encoded in a single interval with the shortest persistence} We consider a hypothetical cancer-detection application. Let us assume that the data set consists of medical images of some cells in the human body, which look like a certain number of holes (e.g., a grid-like structure). Now imagine that the only difference between the healthy and cancerous cells is the presence of a tiny hole somewhere in the image (which might correspond to some developing cancerous tissue) (Figure~\ref{fig_guidelines_ex5}). For PH with the Vietoris-Rips filtration, the lifespan of each cycle registers its size, but it is the very short persistence of the tiniest holes which would be the most important for this application, as it would be this local geometry signal that would allow to discriminate between the two classes of data, i.e., to detect the presence of cancer.

\begin{figure}[h]
\centering
\includegraphics[width = \linewidth]{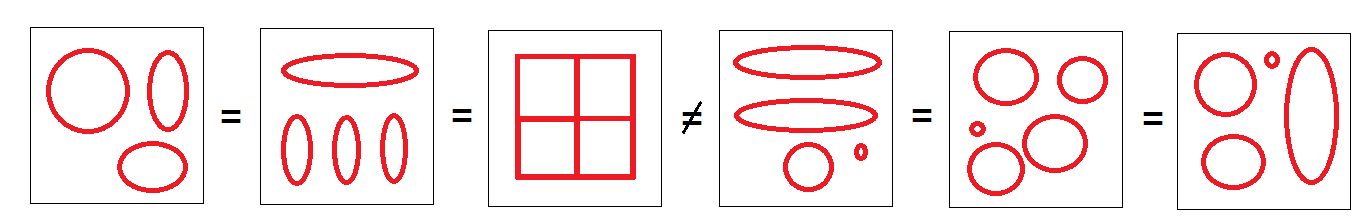}
\caption{An example of an application where the presence of a tiny hole is the signal. The first three shapes in the left part of the figure reflect the images of healthy cells, whereas the remaining three shapes indicate developing cancerous tissue.}
\label{fig_guidelines_ex5}
\end{figure}

For example, PH for healthy and cancerous cells can respectively have lifespans $(20, 20, 20, 20, 0)$ and $(20, 20, 20, 20, 0.05)$.  The stability theorems imply that the difference between the PH on the healthy and cancerous cells is ``small'' (or more precisely, it is limited by the difference in their filtrations), but this difference is important for this problem and hence not ``noise.''

PH can be successful for this task even if the number and size of holes varies across images of healthy cells. In this case, the lifespans for PH of healthy cells could, e.g., be $(20, 15, 12, 25, 0),$ $(13, 21, 15, 17, 0),$ and $(14, 15, 27, 20, 0.05),$ $(19, 21, 15, 17, 0.05),$ for cancerous cells. Here, the distance between the PH for healthy and cancerous cells is overwhelmed by the distance between the long cycles, that reflect irrelevant information for the problem. However, we could consider a PH signature that only focuses on short intervals, or choose PIs that give a greater weight to short intervals. Alternatively, if there is a number of labeled images available, the difference with respect to the short persistence interval can be learned.

In the same way, it might be the case that images can only be distinguished with a hole of medium persistence, and therefore the importance of different lifespans depends on the application, i.e., data set. If we know this a priori, we can use PIs and give the greatest weight to the intervals with the most distinctive persistence.

\subsubsection{Topology and geometry are important}
\label{app_guidelines_ex_signal_top_geo}

To conclude our guidelines, we consider an example of an application in which both topological and geometric information are important. Let us consider a classification problem where the shapes in Figure~\ref{fig_guidelines_ex6} represent different objects, i.e., different data classes. In this case, it is topology and geometry together that provide useful information. The standard choice of PH on the Vietoris-Rips filtration can help to distinguish between these objects. The PH signature should consider all persistence intervals, since, as discussed in Section~\ref{section_discussion}, geometry (reflecting the size of cycles) is captured in every persistence interval, while the topology is reflected in the number of long-enough intervals.

\begin{figure}[h]
\centering
\includegraphics[width = \linewidth]{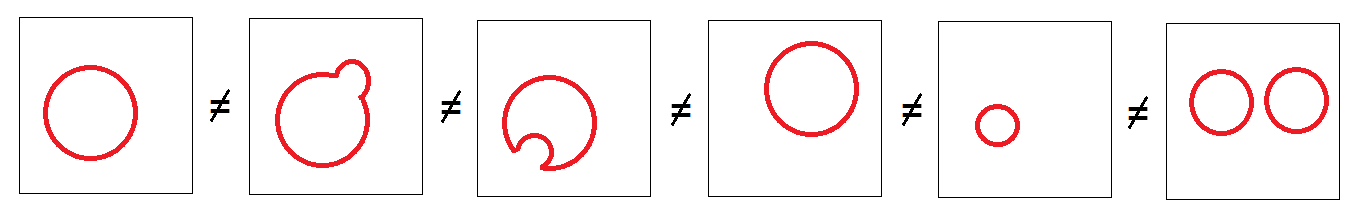}
\caption{An example of an application where topology and geometry are both signal.}
\label{fig_guidelines_ex6}
\end{figure}

\FloatBarrier 
\section{Persistent homology detects convexity in FLAVIA data set} 
\label{app_real_data}

In this section, we employ PH on the FLAVIA data set which consists of $1\,907$ $1200 \times 1600$ images of plant leaves \cite{wu2007leaf}. Figure~\ref{fig_flavia_data} shows a few examples of images in this data set. The goal of these experiments is to show that PH can be effective on real-world data, but also to illustrate the above guidelines about the appropriate choice of filtration and signature for a given application, and the importance of long and short intervals (Appendix~\ref{app_guidelines}). We focus on convexity detection, as this is the main contribution of our work.

\begin{figure}[h]
\centering
\begin{tabular}{p{2cm}p{2cm}p{2cm}p{2cm}p{2cm}p{3cm}}
& raw image & \hskip 0.75cm $X$ &\hskip 0.5cm $CH(X)$ & \hskip 0.25cm $CH(X) \setminus X$& \\
\end{tabular}
\includegraphics[width = 0.75 \linewidth]{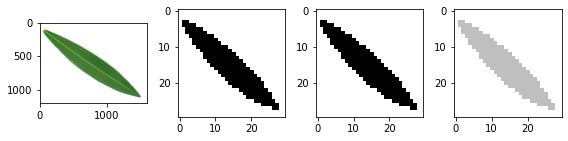}
\includegraphics[width = 0.75 \linewidth]{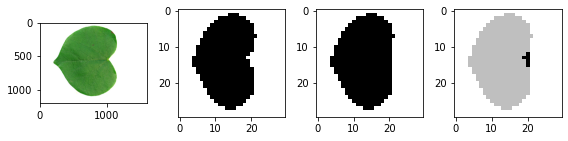}
\includegraphics[width = 0.75 \linewidth]{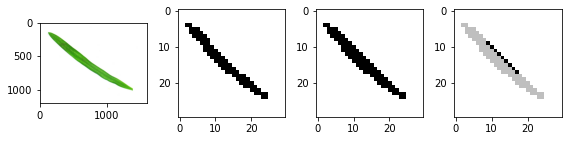}
\includegraphics[width = 0.75 \linewidth]{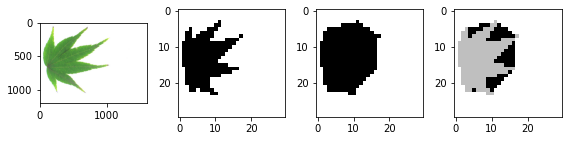}
\caption{A few example images from the FLAVIA leaf data set. The images are shown, from top to bottom, with a decreasing label, i.e., convexity measure $c(X),$ 1.00, 0.98, 0.89, 0.71. Note that the second image from the top is more convex than the third image, since the considered convexity measure $c(X)$ is given relative to the area of the leaf. The PH lifespans on the tubular filtration with respect to 9 different lines are respectively from top to bottom, [0.00, 0.00, 0.00, 0.00, 0.00, 0.00, 0.00, 0.00, 0.00], [0.00, 0.00, 0.00, 0.00, 0.00, 0.00, 0.00, 0.00, 0.00], [0.00, 0.00, 1.09, 0.00, 0.00,  0.00, 0.00, 0.00, 0.00], [0.00, 2.52, 10.08, 3.78, 0.00, 1.26, 0.00, 0.00, 0.00].}
\label{fig_flavia_data}
\end{figure}

We classify the leaves according to following measure of convexity: 

\begin{equation} 
\label{eq_conv_meas}
c(X) = \frac{\text{area}(X)}{\text{area} ( CH(X) ))},
\end{equation}

where $CH(X)$ is the convex hull of image $X.$ The convexity measure (\ref{eq_conv_meas}) is the most widely used in the literature \cite{zunic2004new}, and has been shown to be useful for plant species recognition \cite{kala2016leaf}. Note that we only use the above formula to properly label the data set (Figure~\ref{fig_flavia_data}), but that, deriving convexity information in such a way involves employing a convex hull algorithm.

In the simpler scenario of a binary classification between convex or concave shapes (i.e., signal is the simple: convex - yes or no), we could rely on the same pipeline as in Section~\ref{section_convexity}, where we consider the lifespan of the second most persisting connected component, and then store the maximum such value across all $9$ tubular filtration lines (Figure~\ref{fig_convexity_pipeline}, Figure~\ref{fig_flavia_ph}). This is sufficient information, since we are only interested in whether PH sees multiple connected components - source of a concavity, for \emph{at least one} line. However, the convexity measure (\ref{eq_conv_meas}), the signal in our application, provides a more detailed level of information (regression problem), so that we keep the lifespan of the second most persisting connected component \emph{for all} lines, in order to capture information about sources of concavities seen with respect to  any of the lines.

\begin{figure}[h]
\centering
\scalebox{0.86}{
\begin{tabular}{ccc}
\includegraphics[height = 0.125 \linewidth]{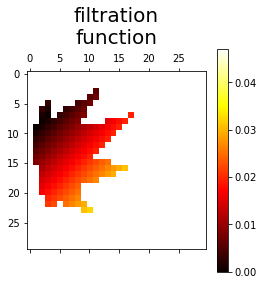} &
\includegraphics[height = 0.125 \linewidth]{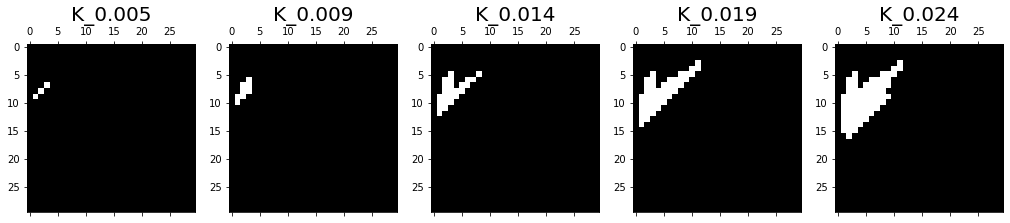} &
\includegraphics[height = 0.125 \linewidth]{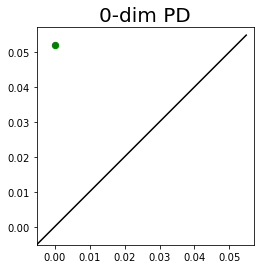} \\

\includegraphics[height = 0.125 \linewidth]{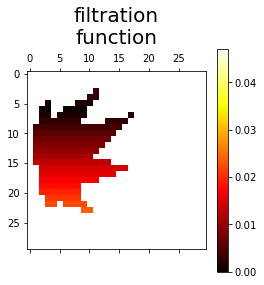} &
\includegraphics[height = 0.125 \linewidth]{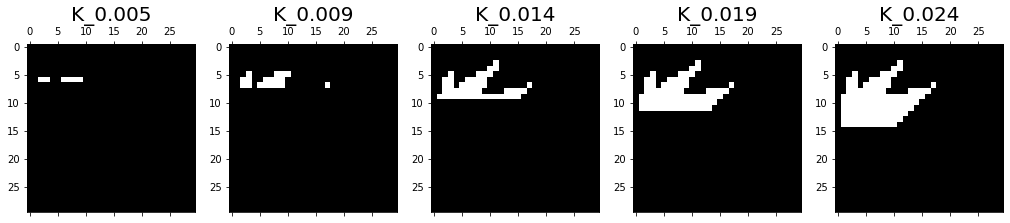} &
\includegraphics[height = 0.125 \linewidth]{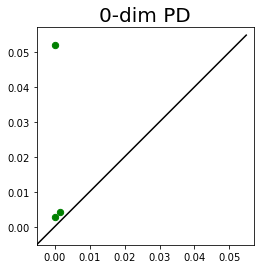} \\

\includegraphics[height = 0.125 \linewidth]{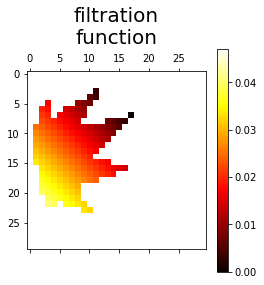} &
\includegraphics[height = 0.125 \linewidth]{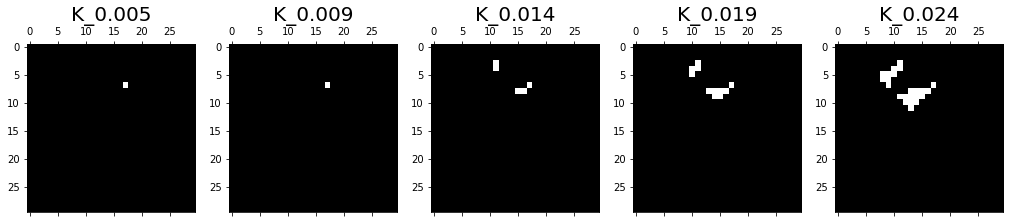} &
\includegraphics[height = 0.125 \linewidth]{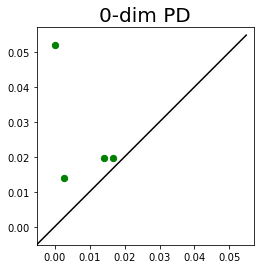} \\

\includegraphics[height = 0.125 \linewidth]{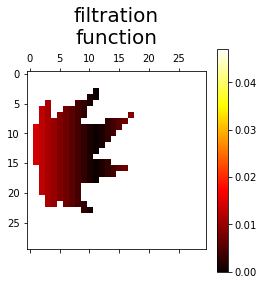} &
\includegraphics[height = 0.125 \linewidth]{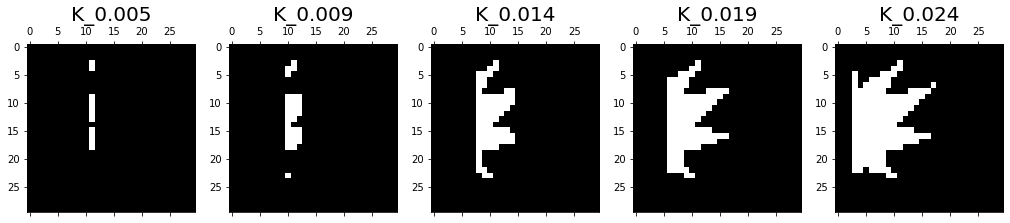} &
\includegraphics[height = 0.125 \linewidth]{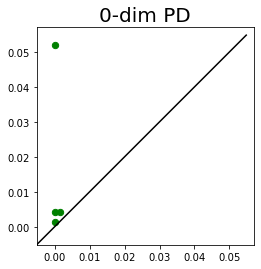} \\

\includegraphics[height = 0.125 \linewidth]{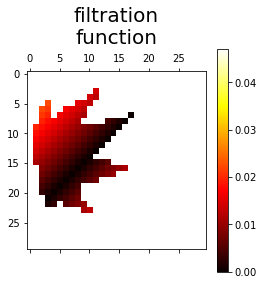} &
\includegraphics[height = 0.125 \linewidth]{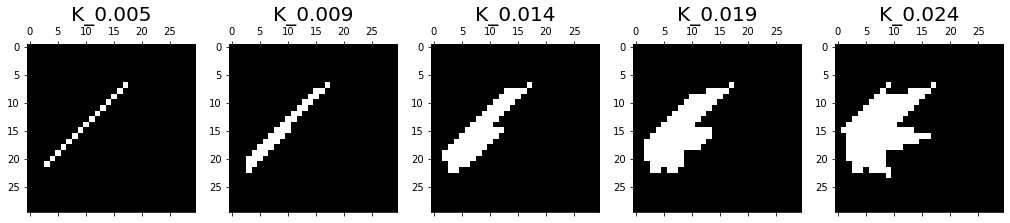} &
\includegraphics[height = 0.125 \linewidth]{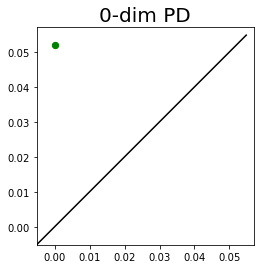} \\

\includegraphics[height = 0.125 \linewidth]{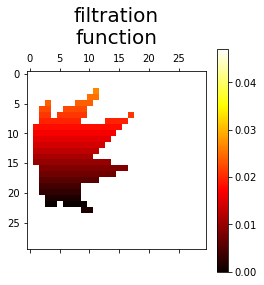} &
\includegraphics[height = 0.125 \linewidth]{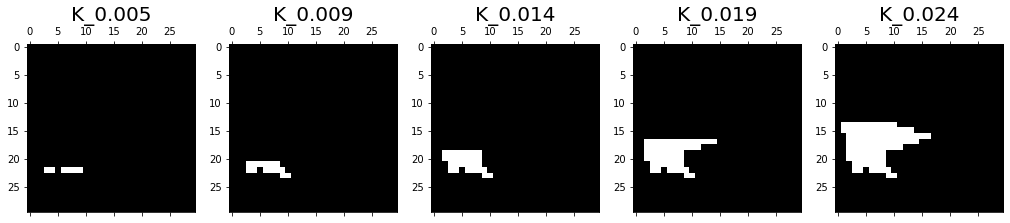} &
\includegraphics[height = 0.125 \linewidth]{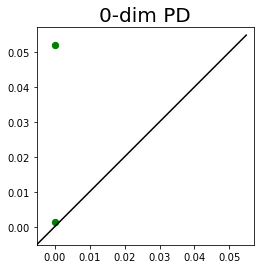} \\

\includegraphics[height = 0.125 \linewidth]{figures/flavia/fil_fun_3.png} &
\includegraphics[height = 0.125 \linewidth]{figures/flavia/fil_3.png} &
\includegraphics[height = 0.125 \linewidth]{figures/flavia/pd_3.png} \\

\includegraphics[height = 0.125 \linewidth]{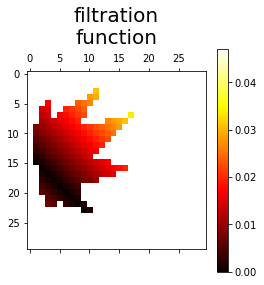} &
\includegraphics[height = 0.125 \linewidth]{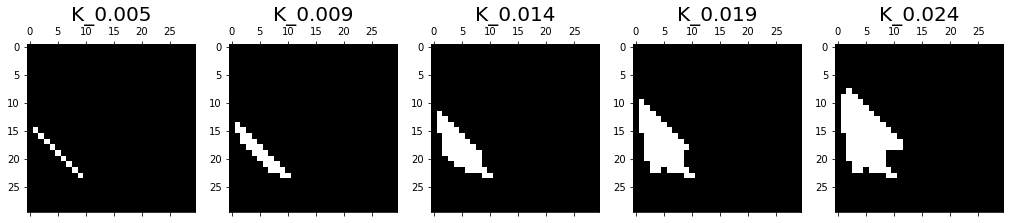} &
\includegraphics[height = 0.125 \linewidth]{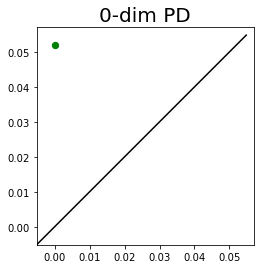} \\

\includegraphics[height = 0.125 \linewidth]{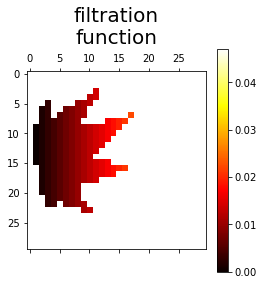} &
\includegraphics[height = 0.125 \linewidth]{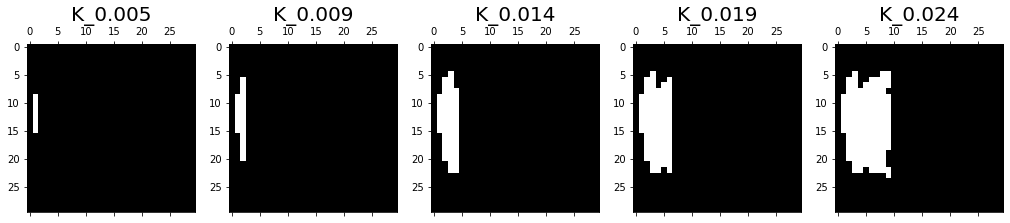} &
\includegraphics[height = 0.125 \linewidth]{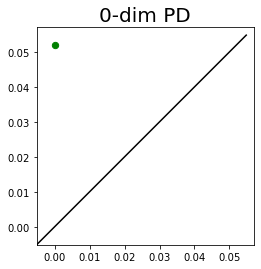} \\

\includegraphics[height = 0.125 \linewidth]{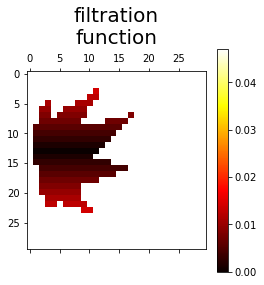} &
\includegraphics[height = 0.125 \linewidth]{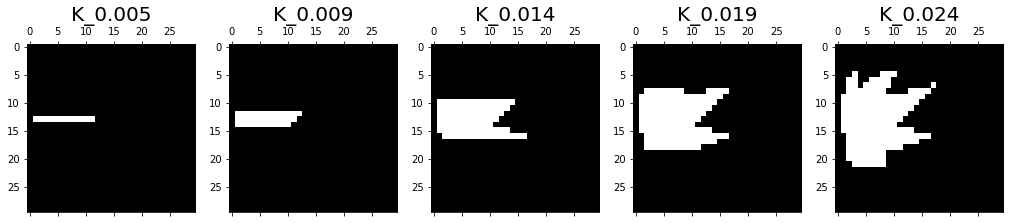} &
\includegraphics[height = 0.125 \linewidth]{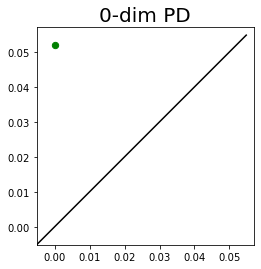} \\

\end{tabular}
}
\caption{The tubular filtration for the 9 considered lines, and the resulting 0-dimensional PDs for an example image. The concavity is detected with multiple connected components that are seen for a few lines.}
\label{fig_flavia_ph}
\end{figure}

Moreover, since the convexity measure is calculated relative to the size of the leaf, the tubular filtration \emph{directions} remain the same, but 8 filtration lines pass through the corners of the leaf rather than the corners of the image (Figure~\ref{fig_convexity_pipeline}), and the lifespans are normalized relative to the area of the leaf (total number of black pixels in the binary image). In this way, PH depicts information about concavities for any line and relative to the leaf size, and it is invariant under translation and scaling. We use $30 \times 30$ cubical complexes (on binary images), to capture a higher level of detail for the leaves of different convexity, in comparison with the $20 \times 20$ resolution for the cruder differences in our synthetic data set in Section~\ref{section_convexity}. Figure~\ref{fig_flavia_ph} visualizes the tubular filtration for the 9 different lines, and the resulting 0-dimensional PDs, for an example leaf image (bottom image from Figure~\ref{fig_flavia_data}).

Linear regression on the FLAVIA data set, trained on $70\%$ of random images, with each image represented with the 9-dimensional vector of lifespans of the second most persisting component across all tubular filtration lines, obtains a mean square error of $0.00065$. The regression line in Figure~\ref{fig_flavia_results} shows that PH is effective in classifying the FLAVIA leaves according to a measure of convexity. The convexity of some thin leaves (such as the image in the third panel in Figure~\ref{fig_flavia_data}) gets overestimated with our PH pipeline, since concavity is not captured well with our crude resolution, that could easily be improved.

\begin{figure}[h]
\centering
\centering
\begin{tabular}{ccc}
\includegraphics[height = 0.2 \linewidth]{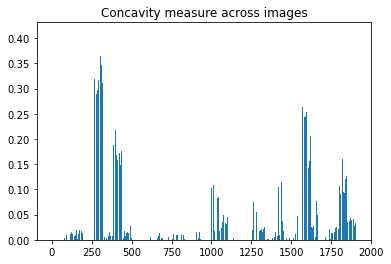} &
\includegraphics[height = 0.2 \linewidth]{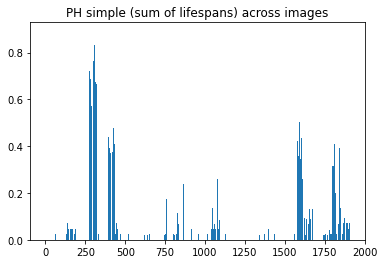} &
\includegraphics[height = 0.2 \linewidth]{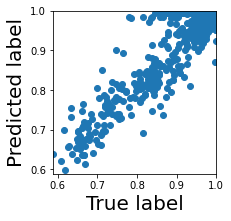} \\
\end{tabular}
\caption{Results on the FLAVIA data set. The first two plots show that there is a good correspondence between the concavity measure $(1-c(X))$ (left panel) and the simple PH signature that only considers the sum of lifespans of the second most persisting connected component, across the tubular filtration lines (middle panel). The regression line on lifespans from all 9 tubular filtration lines shows good performance of our PH pipeline.}
\label{fig_flavia_results}
\end{figure}

Furthermore, even more detailed information can be captured if the lifespans of the  third, fourth, ... most persistent connected component would be kept, because some leaves have more than two sources of concavity for a single line, that result in more than two connected components. For example, the 0-dimensional PD of the image in Figure~\ref{fig_flavia_ph} has more than two persistence intervals for some tubular filtration lines. The accuracy can thus be improved by considering the lifespans of \emph{all} short intervals (and across all lines), and again, by considering more tubular filtration lines (Section~\ref{section_convexity}). This clearly illustrates how the choice of filtration and signature, the input and output of PH, should be guided by the signal in the given application. Moreover, it shows that one short interval can be sufficient for some applications, but that in other cases, many short intervals might store the needed additional level of (geometric) information.

\end{document}